\documentclass{amsart}

\usepackage{amsmath}

  \usepackage{graphics} %% add this and next lines if pictures should be in esp format
  \usepackage{epsfig} %For pictures: screened artwork should be set up with an 85 or 100 line screen
 \usepackage[colorlinks=true]{hyperref}
\hypersetup{urlcolor=blue, citecolor=red}
\usepackage{hyperref}

\newtheorem{theorem}{Theorem}[section]
\newtheorem{corollary}{Corollary}

\newtheorem{lemma}[theorem]{Lemma}

\theoremstyle{definition}
\newtheorem{definition}[theorem]{Definition}
\newtheorem{remark}{Remark}

\newcommand{\ben}{\begin{eqnarray*}}
\newcommand{\een}{\end{eqnarray*}}
 \newcommand{\sgn}{\operatorname{sgn}}

\begin{document}

\title[Posterior convergence]{Posterior convergence for approximated unknowns  in non-Gaussian statistical inverse problems }

\begin{abstract}
The statistical inverse problem of estimating  the probability distribution of an infinite-dimensional  unknown  given its noisy indirect   observation  is studied in the Bayesian framework.     In practice, one often considers only finite-dimensional unknowns and investigates numerically their probabilities.   As many unknowns are function-valued,  it is of interest to know whether the estimated probabilities converge when the finite-dimensional  approximations of the unknown are refined. In this work,  the generalized Bayes formula  is shown to  be  a powerful tool in the  convergence studies.  With the help of the generalized Bayes formula,  the question of  convergence of  the posterior distributions  is returned to the convergence of the finite-dimensional (or any other) approximations of  the unknown.  The approach allows many prior distributions  while the restrictions are mainly for the noise model and the direct theory.  Three modes of convergence of posterior distributions are considered --  weak convergence, setwise convergence and convergence in variation.  The convergence of conditional mean estimates is studied.     Several  examples  of  applicable infinite-dimensional non-Gaussian noise models
are provided,   including a generalization of the Cameron-Martin formula for certain non-Gaussian measures.
  Also,  the well-posedness of  Bayesian statistical inverse problems is studied.   
     \end{abstract}

\author{Sari Lasanen}
%\thanks{Department of Mathematical Sciences, University of Oulu, Finland.   email:sari.lasanen@oulu.fi}

\keywords{Statistical inverse problems,  posterior distributions,  convergence of measures, Bayesian methods,  measures on linear spaces, non-Gaussian distributions}

\subjclass[2010]{Primary:  60B10, 65J22; Secondary: 60B11, 62C10}

\maketitle

\section{Introduction}

Statistically oriented infinite-dimensional inverse problems  are often described as problems where one  wants to estimate an unknown function given its randomly perturbed indirect observation  \cite{cava,evans,frank,siltanen,osul,luis}.  We prefer the following  description which suits well in the Bayesian framework. 

{\it The  statistical inverse problem  is to estimate the probability distribution of the unknown given its randomly perturbed indirect observation}. 
 
In this paper,  the unknown  $X$   and its observation $Y$  are modeled as random mappings from a complete probability space $(\Omega,\Sigma,P)$ into some locally convex Souslin   topological vector spaces $F$ and $G$ equipped with their Borel $\sigma$-algebras $\mathcal F$ and $\mathcal G$, respectively. Recall, that a  Souslin   space is a Hausdorff topological  space that is an image of a complete separable metric space under a   continuous mapping.  The observations are taken to be of the form   $Y=L(X)+\varepsilon $, where $\varepsilon $ represents random noise, $\varepsilon $ and $X $ are statistically independent,  and $L:F\rightarrow G$ is a continuous mapping.  
   The image measure  $\mu_X:=P\circ X^{-1}$  on $F$  is called the prior distribution,  and it represents our beliefs about  the unknown  without any  given observations.    
  
  Typically, we are given  a sample  $Y(\omega_0)=L(X(\omega_0))+\varepsilon(\omega_0)$,  which is produced by an  unknown $X(\omega_0)$ and a perturbation  $\varepsilon(\omega_0)$  for some $\omega_0 \in \Omega$.  Ultimately, we pursue after the probability measure $ U\mapsto1_U(X(\omega_0))$ defined on the Borel sets $U \subseteq F$.  This measure  would determine  the unknown $X(\omega_0)$ uniquely since  $F$ is  a Hausdorff space, which implies that the  singletons are closed sets and belong therefore to the Borel $\sigma$-algebra $\mathcal F$. We get  a simple  approximation of  the function $\omega \mapsto 1_U(X(\omega))$ on the basis of the given $Y(\omega_0)$  by taking its orthogonal projection from    $ L^2(\Omega,\Sigma,P)$  onto $L^2(\Omega, \sigma(Y), P)$, where $\sigma(Y)=Y^{-1}(\mathcal G)$ denotes the $\sigma$-algebra generated by $Y$.   Recall, that  for any $f\in L^2(\Omega,\Sigma,P)$, this projection coincides $P$-almost surely with the conditional expectation $\mathbf E[f|\sigma(Y)]$ of $f$ given the   $\sigma$-algebra generated by $Y$  (see  \cite{dud}). Moreover,  there exists a measurable real-valued function $\lambda_f$ on $G$ such that   $\lambda_f(Y(\omega))=\mathbf E[f|\sigma(Y)](\omega)$ $P$-almost surely.   We  take $\mathbf E[1_U(X)|\sigma(Y)](\omega_0)$ (or more precisely, $\lambda_{1_U}(Y(\omega_0))$  as an estimate of the  probability that the unknown $X(\omega_0)$ belongs  to the set $U\in  \mathcal F$. 
    
     When the  mappings $U\mapsto \mathbf E[1_U(X)|\sigma(Y)](\omega_0)$ form a probability measure on $(F,{\mathcal F})$, which is denoted here with $\mu(U,Y(\omega_0))$, this measure is  called the posterior distribution of $X$ given a sample $Y(\omega_0)$ of $Y$.   From the posterior distribution  one may extract information about the unknown $X(\omega_0)$.   For example,  the posterior mean  may serve as  an estimate of the unknown.
    
   The above estimation  of the probability distribution of the unknown  is  generally known as  the statistical inverse theory (also known as  the statistical inversion or the Bayesian inversion).  We postpone a literature review on the statistical inverse theory to  Section \ref{review}.  The  present paper concentrates on  the following  three topics in  the statistical inverse theory inspired by a paper of Lassas et al \cite{siltanen2}.
\begin{enumerate}
\item[(i)] Applicability  of  the generalized Bayes formula for  statistical inverse problems in locally convex Souslin   topological vector spaces.
\item[(ii)] Well-posedness of the Bayesian statistical inverse problem.
\item[(iii)]   Convergence  of  posterior distributions  and posterior means for approximated unknowns. Especially,  finding  conditions that  guarantee  the convergence of the   posterior distributions when  the corresponding approximated prior distributions converge. 
\end{enumerate}

\subsection{Case (i): The generalized Bayes formula}

  When $X$ and $Y$ have continuous probability densities with respect to the Lebesgue    measure, the conditional expectations lead to the Bayes formula  
 \begin{equation}\label{bab}
 D(x|y)D(y)=D(x,y)=D(y|x)D(x)
\end{equation}
which defines  the unique continuous posterior probability density $D(x|y)$   for any occurred observation   $y$ such that  $0<D_Y(y)<\infty$ (see   \cite{kaip}).   In \eqref{bab},   the functions $D(x), D(y)$, and $D(x,y)$  denote the probability densities of    $P\circ X^{-1}$, $P\circ Y^{-1}$, and $P\circ (X,Y)^{-1}$ at $x$, $y$, and $(x,y)$, respectively.  If the observation is of the form  $Y=L(X)+\varepsilon$, where $X$ and  the noise $\varepsilon$ are statistically independent, then the conditional density   of $Y$ given $X=x$ has the special form $D(y|x)=D_\varepsilon (y-L(x))$, where $D_\varepsilon $ is the continuous probability density of  the noise $\varepsilon$. 

 The availability of the conditional density $D(y|x)$  from the relationship between  unknowns and observations  is the key element for the statistical inverse theory. It makes the expression of the posterior density $D(x|y)$ explicit,   opening  the way for exploring the posterior distribution numerically. Unfortunately, infinite-dimensional probability  measures lack  probability density functions since there is no infinite-dimensional Lebesgue measure.  Instead of  \eqref{bab},  we have
 \begin{equation}\label{yhh}
 \mathbf E[ \mathbf E[1_U(X)|\sigma(Y)] 1_V(Y)]= P(X\in U\cap Y\in V)= 
 \mathbf E[ 1_U(X) \mathbf E[1_V(Y)|\sigma(X)]]
 \end{equation}
 for all Borel sets $U\in \mathcal F$ and $V\in \mathcal G$. The distributions of  $X$, $Y$, and $(X,Y)$ are, in principle,  known.   However, determining  $\mathbf E[1_U(X)|\sigma(Y)]$ explicitly from the first equality in 
 \eqref{yhh} is in  general a hard task, where  an explicit expression of the distribution of $(X,Y)$ is helpful, as in the case of  linear Gaussian problems \cite{markku,lus,mandelbaum}.  On the other hand,   the second equality  in \eqref{yhh} looks easy enough.   For instance, the dominated convergence of simple functions to the exponential function shows that 
    \begin{equation*}
     \mathbf E[e^{i\langle X,\phi\rangle +i\langle Y, \psi\rangle}]=
      \mathbf E[e^{i\langle X,\phi\rangle}\mathbf  E[e^{i\langle Y,\psi\rangle} |\sigma(X)]]
       \end{equation*}
 for all $\phi$ and $\psi$ in the dual spaces $F'$ and $G'$, respectively.  This  suggests  that,   after verifying  some measurability conditions,  we may   take   \begin{equation*}
  \mathbf  E[1_V(Y)|\sigma(X)](\omega)= \mu_{\varepsilon +L(X(\omega))}(V)
  \end{equation*}
$P$-almost surely since $X$ and $\varepsilon$ are statistically independent.  
Does  knowing the conditional probabilities $  \mathbf  E[1_V(Y)|\sigma(X)](\omega)$   help in determining the posterior distribution? 
The answer is positive in some cases.  If  the $\sigma$-algebra  $\mathcal G$  in question  is  countably generated and the conditional distributions  of $Y$ given $X$ are   regular and   $P$-almost surely absolutely continuous with respect to some fixed $\sigma$-finite measure  $\lambda$ on $G$ (i.e. they are dominated by $\lambda$), then the generalized Bayes formula 
 \begin{equation}\label{bab2}
\mu(U,y) =\frac{\int_U \frac{d\mu_{Y|X} (\cdot, x)}{d\lambda} (y) d\mu_X(x) }{\int_F \frac{d\mu_{Y|X} (\cdot,x)}{d\lambda} (y) d\mu_X(x)}
 \end{equation}
 is known to hold   for  $U\in \mathcal F$ and $\mu_Y$-almost every given observation $Y=y$ such that the denominator is finite and non-zero  \cite{kallianpur2,schervish}.   In \eqref{bab2}, it is required that the Radon-Nikodym  densities  $\frac{d\mu_{Y|X}(\cdot,x)}{d\lambda} (y)$  of the conditional measure  $\mu_{Y|X}(\cdot,x)$ of $Y$ given $X=x$  with respect to $\lambda$ are  jointly measurable. This is sometimes achieved by defining the Radon-Nikodym densities with the help of a fixed joint density as is done  in \cite{schervish}.
     In  \eqref{bab2},   the form of $Y$ is allowed to be more general than  in our restricted case  of $Y=L(X)+\varepsilon$, where the posterior distribution has, for suitable $L,X$, and $\varepsilon$,  the form 
 \begin{equation}\label{pooo}
\mu(U,y) =\frac{\int_U \frac{d\mu_{Y+L(x)}}{d\lambda} (y) d\mu_X(x) }{\int_F \frac{d\mu_{Y+L(x)}}{d\lambda} (y) d\mu_X(x)}
 \end{equation}
 for all $U\in \mathcal F$ and  $\mu_Y$-a.e. $y\in G$ such that the denominator is finite and non-zero. When the Radon-Nikodym densities in \eqref{pooo} are known,   the posterior distribution on $F$ has  an explicit representation for all 
 admissable $y\in G$.

In statistical inverse problems,  the generalized Bayes formula for function-valued unknowns has been used before  in the case  of  finite-dimensional noise models that have probability density functions with respect to the  Lebesgue    measure \cite{cotter,fitz, siltanen,stuart} and in the case of infinite-dimensional  Gaussian noise models, using in \eqref{pooo} the Cameron-Martin formula 
  \cite{helin,siltanen2,stuart}. The starting point in \cite{cotter,cotter2,stuart}  is that  the posterior distribution is assumed to have  Radon-Nikodym density with respect to the 
prior distribution.  Therefore,   a similar  formula  like  \eqref{bab2} is used in \cite{cotter,cotter2,stuart}, but not derived.   In \cite{helin,siltanen2,siltanen}, the unknown and the noise are statistically independent. The same seems to be the case in same examples in  \cite{cotter,cotter2,stuart} but the fact is not emphasized.

  Fitzpatrick  \cite{fitz} studied (separable) Banach-space valued unknowns, and wrote the  expression  \eqref{pooo} in the case of finite-dimensional observations.  As a concrete example,  he used  a  Gaussian prior distribution on $C([0,1])$  in  the ill-posed inverse problem of determining  the  function  $q$   in the  differential equation $-(qu')'=f$ on $(0,1)$ from finitely many noisy values of the solution $u$   satisfying the Dirichlet boundary condition.  
 Lassas and Siltanen  \cite{siltanen} used the generalized Bayes formula for certain   prior random variables  on  $C([0,1])$ and assumed the finite-dimensional noise to be Gaussian.  Lassas et al \cite{siltanen2} and Helin  \cite{helin} had emphasis on edge-preserving prior distributions and used  linear forward theory with Gaussian noise, but they  allowed  in \eqref{pooo} also other separable  Banach and  Hilbert space-valued unknowns, respectively.   The forward mapping  $L$  was assumed to be linear  in \cite{helin,siltanen2,siltanen}.    Cotter et al \cite{cotter} studied the case of finite-dimensional observations and Banach space-valued unknowns, and required $L$ to be   measurable. Stuart \cite{stuart} assumed $L$ to be locally Lipschitz continuous and aasumed   finite-dimensional  or Gaussian noise.  Stuart allowed    prior distributions that  are absolutely continuous with respect to  some Gaussian measure.  Theorem 4.1 in \cite{stuart} is an abstract generalization towards allowing  certain infinite-dimensional  non-Gaussian noise distributions but  the identification of  the used notation to any  statistical inverse problem is omitted. The same approach is used in \cite{cotter2}.
 
  In the present paper, we  provide (abstract) assumptions on the forward theory  and the noise that are sufficient for the generalized Bayes formula in the case of   statistical inverse problems  in   locally convex Souslin   topological vector spaces.   However,  such a generalization   is not  particularly novel by itself, and the generalized  Bayes formlula is treated in this work as an important tool for achieving other results.  For example,  the study of Case (iii)  exploits  the generalized Bayes formula.

In this work, we allow infinite-dimensional  noise models (similarly as in \cite{helin,sari,siltanen2,petteri}). One may ask, what are the benefits of such models because any feasible measuring instrument produces only finite-dimensional observations.  For example, an analog-to-digital converter performs the weighted averaging and quantization of the signal;  an X-ray imaging device has a finite number  of projection angles and a limited resolution of     the projection images. In this light, there is no immediate need for infinite-dimensional  noise models.  However,  changes in the measuring instrument can  lead to different posterior distributions and one may wish to choose the best finite-dimensional measurement configuration for the problem.  As noted in 
\cite{siltanen2},  the  mathematical formulation of the infinite-dimensional noise model, when possible,  may be helpful, as it provides an overall framework for the studies.  For some noise sources  there even exists physically motivated infinite-dimensional noise models, like  the  model of the thermal noise in electric circuits, which arises from the thermal motion of the charge carriers.

Particular emphasis  in this work is on finding tools for dealing with non-Gaussian noise in infinite-dimensional statistical inverse problems. 
 There are three reasons why    the Gaussian noise model is not satisfactory. 
  \begin{enumerate}
  \item  Noise does not always follow well enough a  Gaussian distribution.  In Section \ref{spheri} we  discuss  the  appearance of  $\alpha$-stable noise  in statistical inverse problems.  An evaluation of finite-dimensional  noise models in medical imaging can be found in \cite{ctnoise}.  
  \item   Model approximations --  which were studied first by Kaipio and Somersalo \cite{ks} for finite-dimensional observations        -- can also produce non-Gaussian errors (cf.  Remark \ref{model}). 
  \item  Some noise statistics  may not be exactly known. In statistical inverse theory  the inaccuracies in the noise model are  further   modeled with  hierarchical  distributions (see Section \ref{sasse} for a special case).  
     \end{enumerate}
      A  wrong noise distribution may  cause  poor performance of the  estimators of the unknown.

  In Section  \ref{sec5}, we are able to derive (with the help of the generalized Bayes formula)  explicit posterior distributions  in some new cases where the noise has  non-Gaussian infinite-dimensional distribution (Sections \ref{ridi}--\ref{peri}).  It turns out that  in some  cases (Section \ref{spheri})   the posterior  distribution has a simple expression  for the infinite-dimensional observations but not  for the  truncated finite-dimensional observations.  As a further motivation for the study of infinite-dimensional noise models,  we  suggest that the solutions of  infinite-dimensional Bayesian problems may give rise  to new numerically feasible, but non-Bayesian, approximations of the finite-dimensional posterior distributions.

\subsection{Case (ii): Well-posedness of the Bayesian statistical inverse problem.}

The projection operator from $L^2(\Omega,\Sigma,P)$ onto 
$L^2(\Omega, \sigma(Y), P)$  determines posterior probabilities $\mu(U, Y(\omega))$ only up to 
$P$-almost every $\omega\in \Omega$. The  uniqueness of the posterior distribution  for a given $y\in \mathcal R(Y)$ is  therefore unsettled (note that such  form of nonuniqueness  has nothing to do with   the uniqueness of  the deterministic inverse problem of recovering $x_0$ from $L(x_0)$).  The nonuniqueness is fairly  well understood in   Gaussian linear problems \cite{sari,markku,lus,mandelbaum,simonithe},  where  the posterior mean is  known to be determined up to a set of probability zero,  but  has  received limited attention in the general case.    For a Bayesian scientist,  such   nonuniqueness  is discomforting. Two Bayesians using the  same   prior distribution and the same observations can, in principle, have different  posterior distributions for some  observations (in a set of probability zero).   One aim of the present work is to make the two Bayesians agree on the form of their posteriors for a given $y\in G$, at least in some special  cases. In Theorem \ref{ex},   we first carefully identify the nonuniqueness of the posterior distributions in locally convex Souslin topological vector spaces by adopting   a new concept, the  essential uniqueness,  from the theory of conditional measures  to the statistical inverse theory.   Then,  we apply a   choice first made in \cite{siltanen2} and appearing also in  \cite{cotter,cotter2,helin,stuart}, which is to work, if possible, with  a fixed version of the posterior distribution   depending continuously on observations in certain sense.  Evans and Stark suggested even earlier  that certain non-uniqueness problems with conditional expectations could  be avoided by  using  dominated probabilities (see Remark 3.7 in \cite{evans}).

The original part of this work begins  in  Section \ref{sec22},  where we achieve   partial uniqueness of those posterior distributions that  depend continuously on   observations in  the sense that posterior probabilities of Borel sets depend continuously on  observations (cf. Theorem \ref{uniq}).  The partial uniqueness gives an unambiguous meaning to the posterior distribution at a fixed observation. Moreover,  it shows that then the  Bayesian statistical inverse  problem is well-posed -- there exists a unique posterior distribution that depends continuously on the observations.     The method of using continuous probability densities  is widely used  in the finite-dimensional case (see    \cite{kaip}), but seems not to have been taken before within the  abstract infinite-dimensional problems.   

     The  posterior distributions are further studied in  Theorem \ref{support2}, where  it is shown that the continuous dependence of the posterior probabilities of Borel sets  on the observations implies  the absolute continuity of  the posterior distribution with respect to the prior distribution. We remark that Theorem \ref{support2}  clarifies some of  the differences between  the  undominated  and  the $\mu_X$-a.s. dominated  cases. Indeed,  if  $F$ and $G$ are Polish vector spaces,  a result of Macci \cite{macci} says that   the absolute continuity of $\mu_Y$-almost all  posterior distributions  with respect to the prior  distribution is equivalent to the  absolute continuity of the   measures $\mu_{\varepsilon+L(x)}$  with respect to the measure   $\mu_Y$ for $\mu_X$-a.e. $x\in F$. Hence, the  continuous dependence of the  posterior probabilities of Borel sets  on observations is  possible only when the measures $\mu_{\varepsilon+L(x)}$ are   dominated  by some $\sigma$-finite measure for $\mu_X$-a.e. $x\in F$.    What does this mean for the undominated cases?  The  posterior probability of at least  one Borel set will be   discontinuous as a  function of observations.  Hence, the corresponding Bayesian problem is ill-posed -- small perturbations of the given sample can lead to large  perturbations of some posterior probabilities.   The ill-posedness in the linear Gaussian statistical inverse problems has been considered before by  Florens and   Simoni  \cite{simoni,simonithe}, who noted that the posterior mean in the Gaussian linear case can be ill-posed. Florens and Simoni also showed that the regularizing effect of the prior distribution has  a limited power in such a case. They  suggested using an additional Tikhonov  regularization in the Gaussian linear case in order to obtain  approximations of  the posterior means that depend continuously on the observations.
     
  We note that the worst-case scenario for discontinuous posterior distributions on  complete separable metric spaces  is somewhat characterized in \cite{burgess}, where it is proved that    either the set of all  $y\in G$ such that the  posterior distributions $\mu(\cdot,y)$ are mutually singular is (at most) countable or there exists a non-empty compact perfect set  $C\in G$  and a Borel set $B\in F\times G$ such that  $1=\mu(B_y,y)=\mu(F\backslash B_{y},y')$ for all   $y,y'\in C$ such that $y\not=y'$. Here $B_y=\{ x\in F: (x,y)\in B\}$.

   In   Theorem \ref{uni}, we  present some  sufficient conditions that guarantee continuous dependence  of the posterior posterior probabilities of Borel sets on observations  by using  the generalized Bayes formula. Cotter et al \cite{cotter,cotter2} have shown a closely related  result which states that  under certain conditions (including domination and Gaussian prior distribution), their version of the posterior distribution is Lipschitz continuous  in finite-dimensional observations  with respect to  the Hellinger distance. Our  proof relies on the Borel measurability of separately continuous functions -- a result first obtained  by  Lebesgue and later generalized by  Rudin \cite{rud}. The author was unable to find the  proof  of the measurability of separately continuous  Souslin space-valued functions  so  the proof  is included.  
   
   Note that even in a  dominated case the posterior probabilities of Borel sets need  not be  continuous  on  any measurable linear subspace of full $\mu_Y$-measure (see Section \ref{spheri} and Remark \ref{epa} for an example).  However,  the posterior probabilities of Borel sets in a dominated case are always continuous on certain compact sets of nearly full measure 
   (cf. Theorem \ref{lusin}). Unfortunately,  in infinite-dimensional normed spaces the interior of any compact set is empty. The partial uniqueness  of Theorem \ref{uniq} is therefore not generic for infinite-dimensional normed spaces,  unless there is a locally finite union of compact sets $K_i$ such that $\cup_{i=1}^\infty K_i$ has full $\mu_Y$-measure and the restriction of $\mu(U,\cdot)$ onto each $K_i$ is  continuous, which guarantees that $\mu(U,\cdot)$ is continuous on the whole $\cup_{i=1}^\infty K_i$.   However, in Remark \ref{samplespa} we note that  for any version of the posterior distribution there always  exists some  stronger topology on  $G$   that generates the same Borel sets, but makes  the version continuous.

\subsection{Case (iii): Posterior convergence}\label{sec13}
  For computational reasons, the unknown  $X$  is often replaced with a finite-dimensional approximation $X_n$,  where  $X_n$ is an $F$-valued random variable   on  the probability space $(\Omega,\Sigma,P)$ with finite-dimensional range.  Instead of exploring the posterior distribution of $X$ given $Y(\omega_0)$, we  would like to explore  the   finite-dimensional posterior            distribution of $X_n$ given $Y_n(\omega_0):= L(X_n(\omega_0))+\varepsilon (\omega_0)$,  which is   $\mu_n(U,Y_n(\omega_0))= \mathbf E [1_U(X_n)| \sigma(Y_n)](\omega_0)$. As noted in \cite{sari,siltanen2}, the value $Y_n(\omega_0)$ is  not given and  the common procedure is to replace $Y_n(\omega_0)$   with $y=Y(\omega_0)$ in  the expression of $\mu_n$. 
  We continue to call $\mu_n(\cdot,y)$ the posterior distribution of $X_n$, even though the replacement -- strictly speaking -- 
   brings us out from the Bayesian world.  One should  note  that continuity of the posterior distribution $\mu_n$  may additionally  diminish the  distortions in posterior opinions on $X_n$ that are caused by  replacing  $Y_n(\omega_0)$ with   a close  observed  value  $Y(\omega_0)$.     The question is then, do the  posterior distributions  $\mu_n(\cdot,y)$ on $F$ (and the posterior  means)  converge   when the approximations are refined?

 Positive results for  the convergence of either  the posterior distributions  $\mu_n(\cdot,y)$  or $\mu_n(\cdot, Y_n(\omega))$  have  been given by   Fitzpatrick \cite{fitz} in the case of finite-dimensional observations of  separable Banach space-valued  unknowns,   Lasanen \cite{sari}  in the linear  Gaussian case, Lassas and Siltanen \cite{siltanen} for the total variation prior on 
 C([0,1]), Pii\-roinen \cite{petteri} in the framework of statistical experiments for the  Souslin space-valued random variables,   Lassas et al  \cite{siltanen2} for certain Banach space-valued priors (including the Besov prior),  Helin \cite{helin} for certain Hilbert space-valued priors (including an edge-preserving  hierarchical prior), and   Stuart \cite{stuart} for a special form  $f_n  d\mu_0$ of  the approximating posterior distributions, where $f_n\in L^1(\mu_0)$ for a Gaussian measure $\mu_0$.   
 
 The convergence in \cite{fitz} is proved for  the  posterior probabilities of sets  $P_n(U)$ where   $U$ is a  Borel sets and the approximating operators  $P_n$, where $n\in \mathbf N$    are  continuous and   converge  to the identity on the Banach space.  We note that since image of  a Borel set under a continuous mapping in a Polish space is Souslin (see  Theorem A.3.15 in \cite{bog}), the class of all sets  $P_n(U)$, where $ U\in \mathcal F$,   is a subclass of all universally measurable sets. 
  The convergence results in \cite{helin,sari,siltanen2,petteri} hold    with respect to the weak convergence of measures i.e. $\lim_{n\rightarrow \infty} \mu_n(f) = \mu(f)$ for all continuous bounded real-valued  functions $f$ on $F$.  The     convergence results in \cite{cotter2,stuart} are formulated for the Hellinger distance of the posterior distributions and  the convergence results in \cite{siltanen} for  weak convergence of the posterior distributions of the  pointwise values of  the unknown continuous function (i.e.  the weak convergence in distribution). In  \cite{fitz,siltanen},  the  observations $Y$  and the random variables $Y_n=L(X_n)+\varepsilon$ were assumed to have  continuous probability densities with respect to the Lebesgue   measure.The  convergence results in \cite{helin,sari,siltanen2,siltanen} are formulated for  a linear forward theory  $L$ in the case of  Gaussian noise.    The   converging posterior distributions in  \cite{sari} are evaluated either at  samples  of     $Y=L(X)+\varepsilon$ or the points $Y_n=L(X_n)+\varepsilon$, the converging  posterior distributions in \cite{petteri} are evaluated  at  the points $Y_n=L(X_n)+\varepsilon$,  and the convergence results in \cite{helin,siltanen2,stuart}  are formulated  for  fixed versions of the posterior distributions at given samples of $Y=L(X)+\varepsilon$.

  The limit of the posterior distributions  may not always be what one suspects.  The famous example  is the case of the so-called  finite-dimensional total variation priors   whose  highly appreciated  non-Gaussian posterior distributions  converge weakly to  a Gaussian distribution as the approximations are refined.  Lassas and Siltanen \cite{siltanen} showed that this problem actually originates from the  behavior of  the prior distributions --   the  random  variables $X_n(t)$ obeying  total variation priors converge to  Gaussian limits when the disrcretizations are refined.  This discovery, first conjectured by Markku Lehtinen, has  changed the view on how   Bayesian statistical inverse problems should be solved for infinite-dimensional unknowns --  one should construct an infinite-dimensional prior distribution and check that  the corresponding finite-dimensional  posterior distributions converge to the right limit. Otherwise one risks  the consistency of  the prior knowledge and the consistency of the  posterior distributions with respect to the increase in dimensionality. This guideline is followed in 
  \cite{helin,sari,siltanen2,petteri}.

 Lassas et al  \cite{siltanen2} introduced   a deterministic function on $F$, called the reconstructor $R_g$, that coincides a.s. with  the conditional expectation of  $g(X)$ given  $Y^{-1}(\mathcal F)$, where $g$ is a  measurable function having values in some separable Banach space. Lassas et al used  a clever choice of their reconstructors $R_{1_U}$, $U\in \mathcal F$, which allowed them to state posterior convergence results  for any given observation.  The framework effectively transformed a question of originally probabilistic nature,  the convergence of the  conditional expectations $\mathbf E [1_U(X_n)|Y_n]$,  into a question in analysis, the convergence of integrals. Moreover, it  was possible to replace the samples of   $Y_n$ with samples of 
 $Y$ in the posterior distribution.  The same technique is extensively used in the present work. 
 
 Unlike in  \cite{siltanen2},   the convergence of the posterior distributions in    Helin's work \cite{helin}  is not based on approximating the prior random variables directly but on approximating the prior  probability distributions in the weak topology.   We adopt his viewpoint, since the posterior distribution depends on the prior random variable only through its  distribution, assuming that the  noise and the unknown are statistically independent
  (see Theorem \ref{ex} and Lemma \ref{eka2}).

  Positive results for the convergence of   posterior  means  for approximated unknowns have been obtained in the linear Gaussian case  \cite{sari} (a.s. convergence in the Schwartz space $\mathcal D'([0,1])$),  for the total variation prior on $C([0,1])$ \cite{siltanen} (for the pointwise-values),  for  exponentially integrable separable  Banach space-valued priors \cite{siltanen2}  (in the norm topology),   for uniformly discretized separable Hilbert space-valued  priors  with exponential weights   \cite{helin} (in the norm topology for all exponentially bounded functions),  for polynomially bounded functions and  those posterior approximations that have the form  $f_n  d\mu_0$, where $f_n\in L^1(\mu_0)$ for some Gaussian measure $\mu_0$ \cite{cotter2,stuart}.  
  
  In Section \ref{sec4}, some results in \cite{cotter,helin, siltanen2,petteri,stuart}  that  concern  the  weak convergence of posterior distributions  and convergence of   the posterior  means   are extended  in several directions.  
  
  Firstly, we  allow prior distributions   to be  probability measures on a locally convex Souslin topological vector space $F$, whereas  Lassas et al  \cite{siltanen2}  and Helin \cite{helin} applied the generalized Bayes formula for separable Banach  and Hilbert space-valued unknowns, respectively.  
Posterior distributions in locally convex Souslin (not metrizable) topological vector  spaces have been considered before  in the special case of  Gaussian  distribution-valued random variables  in \cite{sari,markku} and in the general abstract case of Souslin space-valued random variables only in \cite{petteri}.   The first part of Section \ref{sec2} is therefore devoted to the basics of the abstract  statistical inverse theory in locally convex Souslin topological vector  spaces.     Unlike in the work of Piiroinen \cite{petteri}, where general Souslin space-valued random variables were first studied in statistical inverse problems, we use the generalized Bayes formula for the proofs. We also  derive the generalized Bayes formula  from the equation $Y=L(X)+\varepsilon$, which supplements also the formulation presented in  \cite{cotter,cotter2,stuart} where  the starting point is a given form of the posterior distribution. This makes it easier to recognize  situations where the generalized Bayes formula holds.

  In this work, the use of locally convex Souslin   topological vector spaces is mostly motivated by the fact that in Souslin  spaces   the Borel $\sigma$-algebras are  regular enough for the existence of regular conditional measures  \cite{bogm}.  Moreover,   the class of  such spaces  contains many useful spaces, like complete separable  metric vector spaces and   spaces of (Schwartz) distributions \cite{schwarz}.  We remark   that the distribution spaces are  sometimes preferred since   the convergence of the characteristic functions of measures implies  the weak convergence of measures for them  (this fact is shown e.g. in \cite{bogm} and used in \cite{sari}).   We require $G$ to be a  topological vector space  since  $Y$ is defined as  the sum of  two $G$-valued random variables.  However, it is well-known that  the sum of two   random variables is not always a random variable in  arbitrary topological vector spaces. In Lemma \ref{yeke}, we check that $Y$ is indeed a random variable because the sample space $G$ is a Souslin space (the fact is known but the author was unable to find  a reference  for the proof in the literature).  We require $G$ and $F$ to  be locally convex topological vector spaces  since locally convex spaces have rich enough dual spaces that for example allow the use of characteristic functions in the  identification of    measures. In Remark \ref{samplespace},  we note that  a  locally convex Souslin  sample space  of  $Y$ is allowed to be mis-specified by a continuous linear injection without altering the posterior distributions. This holds for any statistical inverse problem,
   not just for those admitting the representation \eqref{bab2}.

   The main  difference of the present  Theorem \ref{coco2} to  Theorem 4.8 in   \cite{petteri} is that   we do not require the conditioning $\sigma$-algebras $Y^{-1}_n (\mathcal G)$ to be increasing. This is a significant difference as it allows more general approximation schemes.  On the other hand, the present approach  utilizes  the generalized Bayes formula, which was not needed in   \cite{petteri}.  Hence, the results in \cite{petteri} are valid for many noise models  that are  not covered by our assumptions, like  additive undominated noise,  multiplicative noise, or noise that is  statistically dependent on the unknown (we assume that the noise is statistically independent from $X$ and all of its approximations $X_n$).  Another difference from Theorem 4.8 in  \cite{petteri} is  that we work with  one fixed  sample 
  $y$ of  $Y=L(X)+\varepsilon$ whereas in \cite{petteri} it is assumed that  we have a sequence  $\{y_n\}$  consisting of samples  of the random variables   $Y_n=L(X_n)+\varepsilon$, which is a drawback when one considers realistic observations. The reason for this is that in \cite{petteri} the convergence is shown 
  for the conditional expectations $\mathbf E [f(X_n) | \sigma(Y_n)](\omega)$ (which are equivalent to $ \mu_n (f,Y_n(\omega))$) for   $P$-almost every $\omega\in\Omega$).  However, in   Lemma 4.15 in  \cite{petteri}, it is explained how Theorem 4.8 can be used in certain cases where the observation  $Y=L(X)+\varepsilon$ is given.  Namely,   the prior distribution of the Souslin space-valued random variable $X$  is assumed to be  concentrated on a separable Hilbert space $H$  and its approximations are of the form $X_n=P_nX$, where  the operators $P_n$  are   some finite-dimensional linear operators on $H$ that converge to the identity at every $x\in H$.  Assuming that  the range of the linear operator $L$ is  some separable Hilbert space $\widetilde H$ and the noise $\varepsilon\in \widetilde H$ with probability one,   Piiroinen constructed certain projection  operators $R_n$, and  showed that    the sequence of the posterior distributions of   $X_n$  given   $R_{n} (LX+\varepsilon)$   converges weakly to the posterior distribution of $X$ given   $LX+\varepsilon$ as $n\rightarrow \infty$.    The result of Piiroinen shows,  remarkably, that in some cases less data is adequate -- and  easier to manage -- than full data. We remark that the required  assumptions exclude   injective compact linear operators of infinite rank as $L$.   Indeed,   if  $L:H\rightarrow \widetilde H$ is any compact linear  bijection, then its inverse operator is bounded by  the open mapping theorem. Considering $LL^{-1}=I$ on $\widetilde H$  and $L^{-1}L=I$ on $H$, we see that   $L$  can not be compact unless the Hilbert spaces $H$ and $\widetilde H$ are finite-dimensional. On the other hand, compactness of $L$ is a typical feature leading to the ill-posedness of the inverse problem.  But under the assumptions of  Lemma 4.15 in  \cite{petteri}, the convergence result of Piiroinen 
  is stronger than ours in the sense that   it  allows any noise model with the property $\varepsilon \in \widetilde H$ with probability 1. However, the compactness of $L$  is not a restriction for the present convergence results.
   The weak convergence of  posterior distributions in the linear Gaussian case in \cite{sari} is not  fully covered by the present results since some of the cases appearing in \cite{sari} are not $\mu_X$-a.s. dominated.

  Secondly,  we allow a wider class  of noise models  than the Gaussian models applied in \cite{helin,siltanen2}. 
    Theorem \ref{coco2}  gives a positive answer to the  weak convergence of the  posterior distributions in locally convex Souslin topological vector spaces, when the translations of the noise distribution by $L(x)$  are $\mu_X$- and 
  $\mu_{X_n}$-a.s. dominated and the corresponding Radon-Nikodym densities satisfy certain measurability and uniform integrability conditions.    Examples of suitable noise models are given in Section \ref{sec5}, which include the well-known cases of finite-dimensional noise and Gaussian infinite-dimensional noise but also four novel models such as  spherically invariant noise and   periodic signals in decomposable noise.

 Theorem \ref{kass}  extends  the convergence of the  posterior means  in \cite{helin,siltanen2}  for     more general noise models,  and relaxes slightly the  integrability properties  imposed on the  approximations of the unknown in \cite{helin}.  Theorem  \ref{kass}  extends also assumptions in Theorem 4.10 of \cite{stuart} for more general approximations of the prior distributions (but the mode of convergence is different). Some  conditions,   which imply  the posterior convergence of continuous linear functionals for weakly  converging  prior distributions, are presented in  Theorem \ref{kass0}.  The  mentioned conditions are indebted to the well-known criteria for the convergence of  integrals with respect to measures that converge weakly.

 Thirdly, we consider stronger modes of convergence for   posterior distributions than  the weak convergence considered in \cite{helin,siltanen2}.  In  Theorem \ref{coco}, we give sufficient conditions  under which  the posterior distributions  inherit also the setwise convergence or  the convergence in variation  of the approximated prior distributions.  Recently, Stuart has established  (see  Theorem 4.6 in \cite{stuart}) an estimate for  the speed of  convergence in Hellinger distance of the  posterior distributions  for the approximated posterior distributions of the restricted form $\mu_{n}(dx,y)=f_n(x,y)d\mu_0(x)$, where $f_n(\cdot,y)\in L^1(\mu_0)$ and $\mu_0$ is a Gaussian measure on  a Banach space $F$. In the present  Theorem \ref{coco}, the approximated posterior need not be absolutely continuous with respect to a Gaussian measure, and the approximations $X_n$  need  not  be measurable functions of $X$.   
  
  Moreover,  we  allow the direct theory  $L$ to be nonlinear  (as  in the  less general cases in \cite{cotter,cotter2,stuart}), which is   a minor modification of  the linear case in \cite{helin,siltanen2}, but indicates that  nonlinearity  does not necessarily complicate the mathematical convergence,  although the exploration of  the posterior distribution becomes more difficult.  The result is not surprising since nonlinearities are frequently handled  in the  stochastic filtering problems \cite{kallianpur,oek}, which  have connections to the statistical inverse problems.  Throughout  the paper, we  consider continuous  forward mappings $L:F\rightarrow G$,  although  the existence of  posterior distributions requires only their measurability. However,   the  continuity of $L$ is utilized  in the main results of the present paper  on  the convergence of the posterior distributions  and  on  the well-posedness of the Bayesian statistical inverse problem. It is also consistent with  the usual description of the deterministic inverse problems where continuity holds.

Unlike in  \cite{helin,sari,siltanen2,petteri},   the  case of approximated observations is not studied (nor reviewed) in the present work.  By focusing on the approximated unknowns, we hope to single out their essential properties.
Section   \ref{sec6} contains some examples  of  prior approximations.

{\it Notations:} When $G$ is  a topological vector  space, we denote with  $G'$ its topological dual space. If $m$ is a measure on 
$G$, we sometimes denote $m(f):=\int_G  f(x)dm(x)$.
 If $Z:\Omega\rightarrow G$ is a random variable, 
 its image measure $P\circ Z^{-1}$ on $G$ is denoted with 
 $\mu_Z$. A Borel measure and its Lebesgue's completion are denoted with the same symbol.  

\subsection{A literature review} \label{review}
 
  Statistical inverse theory became a popular  method for solving geophysical problems in 1980's \cite{tarantola,tarantol}, and has  since spread into many other fields (see   \cite{kaip, stuart}).
   In this short review, we focus on general theoretical developments  that lead to the modern description of    statistical inverse theory.  A more problem-oriented  review of   infinite-dimensional Bayesian statistical inverse problems can be found in  \cite{stuart}, and  reviews of statistically oriented  inverse problems can be found in \cite{cava2,luis}.        A good reference to the computational aspects of finite-dimensional statistical inverse problems is  \cite{kaip}, to  Bayesian statistics \cite{flo,robert,schervish}, and  to  measure theory \cite{bogm}.

  The statistical background of the statistical inverse theory belongs to  the field of nonparametric Bayesian inference.   Nonparametric statistics is  concerned with making inferences about   infinite-dimensional unknowns whereas parametric statistics studies finite-dimensional unknowns \cite{bernardo}.  The function-valued prior models in statistical inverse problems are therefore well within  the scope of  nonparametric statistics.   We briefly review     Bayesian  nonparametric statistics and clarify its relations to  statistical inverse problems.  
    
 Important nonparametric  problems  are  the density estimation problem and the regression problem  \cite{muller}.  These two problems have  guided the modern  development of  Bayesian nonparametric statistics.    
  
  In  the density estimation problem, the observations  are i.i.d. samples  obeying   some  unknown probability distribution that has a density function $f$ (usually on $\mathbf R$), and the objective is to estimate  the density function $f$.  This problem is not directly related to our statistical inverse problem but is  connected to  the general development of the research field. It should be mentioned  that    Wolpert et al \cite{wol1,wol2}  have described    a semidiscrete Fredholm integral equation of the first kind  as a Bayesian density estimation problem.  On the other  hand, the  positron emission tomography (PET) imaging  is an inverse problem that is  usually described as a special density estimation problem where only indirect  samples are available \cite{silver}.   
  Hence, certain inverse problems lead to density estimation problems.
  
  In the regression problem, the observations are  of the type
  $$
  y_i=K(x_i)+\varepsilon_i.
  $$
  where $x_i\in \mathbf R^n$,  $i=1,...,n$, and the noise terms $\varepsilon_i$ are typically independent  and identically distributed. The objective is to estimate the unknown function $K$.   This problem has connections to statistical inverse problems. For example, if the realizations of $X$ and $Y=L(X)+\varepsilon$ are functions on $\mathbf R$ and  $L$  is the identity mapping, then $X$ is identified as $K$.

One difference between the density estimation problem and the regression problem is the nature of given samples. 
In the regression problem, a single sample can also be infinite-dimensional, at least in theory.  When the noise is Gaussian, such  infinite-dimensional observation models  are often called white noise models (see the short review in \cite{zhao}).

The main questions in  Bayesian nonparametrics have been the  construction of  the  prior models,  the  utilization of  the posterior distribution, and the consistency of  the  posterior distributions.

\subsubsection{Prior models}

The problem of finding good infinite-dimensional prior models  has a long history. An early application  of a function-valued  prior model was carried out in 1896 by    Poincar\'e \cite{poin}, who applied  a random  series of the type    $X(t)=\sum_{i=1}^\infty X_i t^i$  in a  regression problem on $[0,1]$. He assumed independent normal distributions on coefficients  $X_i$ and   calculated the posterior mean estimate on the basis of  the given values $y_i=X(t_i)$, $i=1,...,n$.   In  Section V of Chapter 11 in \cite{poin2}    Poincar\'e  discussed,  in his visionary manner, the noisy regression problem. He proposed that the smoothness of the regression curve follows from the  prior information of the unknown curve described in the form of  probability distributions.  
In 1950, Grenander \cite{grenander} applied a Gaussian process prior in a linear  regression problem with additive  Gaussian process noise. In 1957-58 Whittle   \cite{whittle1,whittle2} discussed   prior information on the  smoothness of the unknown  in certain density estimation problems, and later Kimeldorf and Wahba \cite{wahba2} clarified  the relations between smoothing and  Gaussian prior models. Nowadays, regularity of functions is one of  the most important guidelines in constructing infinite-dimensional prior models in statistical inverse problems.  This follows from the fact that the priors in statistical inverse problems have two objectives. They express the prior beliefs about the unknown and  are countermeasures against the ill-posedness of the deterministic inverse problem. 

In general, the knowledge on  infinite-dimensional random variables (and on their distributions) started to increase  after Wiener published  his  construction of  the Brownian motion in the beginning of  1920's \cite{wiener}.   A decade later,  Kolmogorov \cite{kol} introduced  a constructive method   for defining  general infinite-dimensional random variables  in  the abstract setting. His method suited well for   countably many random variables,  but Doob noticed that  the constructed $\sigma$-algebra was somewhat limited: for continuous parameter processes certain interesting sets, such as the set of all continuous functions,  were not measurable with respect to  the constructed $\sigma$-algebra.  Doob's remedy was the careful definition  of  the continuous-parameter stochastic processes  in 1937 \cite{doob}.  The theory of stochastic processes    Doob's definition of the separable stochastic  processes provides the tools but not immediate  answers for certain  questions in statistical inverse problems. Namely,   can the stochastic process be interpret as a function-valued random variable that has values in some nice function space?  The question is quite relevant  since the direct theory is a mapping between two function spaces.  One can  e.g. apply Kolmogorov's continuity theorem (proven by Kolmogorov in 1934, see \cite{slu}). Another approach  is to directly define probability measures  on function spaces. Jessen  \cite{jessen} carried out   integration on infinite-dimensional dimensional torus equipped with the coordinate-wise convergence.  M. Fr\'echet  initiated  the study of  random variables  in metric spaces (see   \cite{frechet}). His emphasis was on different modes of  convergence and typical values of  random variables, like the mean and  the median.   Significant contributions to the theory of probability measures   on topological spaces were given by Alexandrov and Prohorov (see   \cite{vara}).    Later devolopements can be found in the books  of  Bogachev \cite{bog,bogm}, Gelfand and Vilenkin \cite{gel},   Gihman and Skorokhod \cite{gih},  Kahane \cite{kahane},   Kuo \cite{kuo},  Ledoux and Talagrand \cite{ledoux},  Schwartz \cite{schwarz}, Vakhania \cite{vakha},  Xia \cite{xia}, and Yamasaki \cite{yamasaki}. Typical points discussed in these books are the  existence of measures,   invariance properties of measures,  topological supports,  the equivalence  and  the equality  of measures,  the convergence of random series and  the convergence of measures, all relevant properties for   prior distributions.   The existence of measures is often based on the Bochner-Minlos theorem that gives conditions for the one-to-one correspondence between measures $\mu$ and  their  characteristic functions $L(\phi)=\int e^{i\langle x,\phi\rangle} d\mu(x)$ on certain spaces.   From the point of view of statistical inverse problems, it is unfortunate that direct connections between the characteristic function and  the included prior information are not known. Therefore, it is no wonder that popular prior models have been  described by other means, for example with  infinite product measures and random series expansions.  The works of  Karhunen \cite{karhunen} in 1940's on  a series expansions of Gaussian random variables, nowadays known as the Karhunen-Lo\'eve expansion,  are in this sense important.    The Karhunen-Lo\'eve expansion was first used  in 1950's and 1960's for expanding infinite-dimensional data  \cite{davenport,grenander,kelly}, which made  the  Bayesian method of conditional mean estimation  and  the non-Bayesian method of  likelihood ratio testing tractable. It  was later adopted to  describing   infinite-dimensional  unknowns (see for example \cite{cox}),  but its main application has been  in providing finite-dimensional  approximations of Gaussian random variables.  At present, other orthogonal expansions  of Gaussian random variables are  available \cite{bog}. The pioneering work of Mandelbaum \cite{mandelbaum} from 1984 on linear Gaussian statistical  inverse problems relies on such series expansions of the Gaussian random variables.  Other works on Gaussian priors in statistical inverse problems are \cite{simoni,sari,markku,lus,simonithe}.

   In  1963,  Freedman introduced  the class of tail-free   priors  for the density estimation problem \cite{freed}.   In  1970's the density estimation and the  regression problem evolved further  in different directions.     Wahba  et al \cite {wahba2, wahba} took the approach with  smoothing splines and Gaussian random series in the regression problem, and   Ferguson \cite{fer} constructed Dirichlet process  priors, which are  certain random measure-valued unknowns,  for the  density estimation problem.  In the case of Dirichlet processes, the space of the unknowns is the space of all probability measures on 
   the fixed measure space equipped with the Borel $\sigma$-algebra with respect to the weak topology of measures. The Dirichlet process priors have similar properties in the density estimation problem as    Gaussian priors have  in the linear statistical inverse problems. Namely, the posterior distribution is the distribution of another Dirichlet process with  updated parameters.  In the both cases, the calculations of the posterior distribution  are based on similar elements,which are  the properties of the finite-dimensional distributions and the properties of the martingales \cite{fer,mandelbaum}.
    The  Dirichlet process priors were generalized later to  mixtures of Dirichlet processes (see \cite{esco}).
   Summaries of the prior distributions applied in modern density estimation problems can be found in \cite{gos,walter}. 
 
In 1990, Steinberg \cite{stein} suggested  a prior model defined  as a random series in which Hermite polynomials were multiplied by either  improper    or Gaussian  coefficients.  During 1998-2000, Abramovich et al \cite{silve2,silve3,silve} suggested   random wavelet expansions in Besov spaces,  with hierarchical coefficients  whose hyperparameters guaranteed the sparseness  of the expansions,  as priors for  the regression problem.   In 1990's also mixtures of Gaussian measures were suggested  as priors for the regression problem  (see   \cite{zhao}).  Recently, Lassas et al \cite{siltanen2} and Helin \cite{helin}  constructed non-Gaussian edge-preserving priors  suitable for statistical inverse problems.  Besov space priors introduced by Lassas et al  are defined with random wavelet expansions and Helin's  hierarchical prior distributions as mixtures of Gaussian  measures. In 2010, Stuart \cite{stuart} applied prior distributions of the type   $f(x) \mu (dx)$, where $f\in L^1(\mu)$ and $\mu$ is a Gaussian measure. In the abstract setting,   the statistical inverse theory was applied for unknowns described as   Souslin space-valued random variables in \cite{petteri}.  
     
 It should be mentioned that some combinations of  prior informations  do not have faithful probabilistic descriptions.  In 1987,  Backus \cite{backus} pointed out  that hard constraints,  such as  the boundedness of an infinite-dimensional random variable $X$ in norm, can lead to troubles if one assumes also isotropy. A well-known example is  a Gaussian random variable  $X$ that is  invariant with respect to rotations (e.g. orthogonal transformations) on an infinite-dimensional separable  Hilbert space $H$ but satisfies $\Vert X\Vert_H =\infty$ with probability one \cite{bog}.

\subsubsection{Utilization of the posterior distribution}
  
 We first look at the history   of infinite-dimensional  posterior distributions in nonparametric statistics and in statistical inverse problems. 
 
In  Bayesian nonparametrics, the both problems, the density estimation and  the regression problem, are solved  with  conditional probability measures. This part of the solution mechanism is exactly the same as in  statistical inverse theory.  In 1930's, the rigorous definition of the conditional expectation by Kolmogorov  \cite{kol}   made it  possible to define conditional probability measures in the abstract infinite-dimensional setting but it was soon noted that such conditioning did not always produce  a probability measure. The works of  Doob \cite{doob1} and  Dieudonn\'e \cite{die1,die2}     lead to the definition  of  a regular conditional probability, which is a random probability measure with probability one. The existence of regular versions of   all conditional probabilities  was  verified by applying certain  properties of the space of the unknowns  in the works of   Rohlin \cite{roh},  Ji{\v r}ina \cite{ji1,ji2}, and Sazonov \cite{sa}.
 Nowadays, one either checks  the properties of  the space of the unknowns (as in \cite{petteri}) or checks  always the regularity of  the acquired conditional measure for the chosen prior distribution (as in \cite{fer}). The former is  used in theoretical studies 
  for avoiding pathological cases \cite{helin, sari, siltanen2, petteri}  while the latter is convenient in practical solutions where a fixed version is needed \cite{cotter,stuart}.    We remark that the  non-existence of a regular version is known  only for some  conditional measures in   exemplifying cases (see \cite{bogm,regu}).

A major step for the statistical inference for stochastic processes was the emergence of  the so-called filtering problems
in 1940's  by Wiener \cite{wiener2}, Kolmogorov \cite{kol1} and Krein \cite{krein1,krein2}.  Especially, Wiener's  \cite{wiener2}   straightforward method of solution (by ergodicity and  least squares estimation) encouraged  others to take later further steps towards the Bayesian nonparametric approach \cite{foster,grenander,whittle2}. A good review on  developments  in the filtering theory is \cite{kailath}.   An interesting  work in the filtering theory is  \cite{bhat}, where it is shown that 
the solution of the filtering problem depends continuously on the distribution of the unknown. A nice collection of nonlinear filtering problems with Gaussian noise can be found in \cite{mandal}.

The first deliberate unions of  inverse problems and Bayesian statistics were seen in  1960's in the form of statistical regularization i.e.  minimum mean squared error estimation for the  Gaussian linear inverse problem \begin{equation}\label{gaus} Y=LX+\varepsilon\end{equation} with the finite-dimensional  unknown $X$ and  the finite-dimensional observation $Y$.  That is, one pursues after the estimator $\widehat X(Y)$ that minimizes   $\mathbf E [\Vert \widehat X-X\Vert ^2 ] $  (i.e.  the conditional mean).  In 1961, motivated by Wiener's filtering theory,  Foster \cite{foster} presented a solution to the estimation problem \eqref{gaus}.   Other motivation for the Gaussian  approach arose from  the regularization method of Philips \cite{phil} generalized later by Twomey for  Fredholm integral equations of the first kind \cite{twomey} and from the Tikhonov regularization method.   During 1967-71   Turchin et al (see \cite{turchin} and references therein), independently with  Strand and Westwater \cite{strand},  replaced the regularization method by a statistical framework that utilized a Gaussian prior distribution.    The approach lead to  Franklin's  infinite-dimensional description \cite{frank} of the minimum mean squared error estimator of a Hilbert space-valued Gaussian unknown whose linear observations were corrupted by an additive  Gaussian white noise. The connection between \cite{frank} and  regularization methods  in reproducing Hilbert spaces were studied by  Prenter and Vogel  \cite{pren}. The first work that contained   the  existence of regular conditional probabilities and an explicit formula for the posterior distribution in a  linear infinite-dimensional inverse problem  was the seminal paper of   Mandelbaum \cite{mandelbaum} on Hilbert space-valued Gaussian random variables.    The value of the result for inverse problems was first recognized by  Lehtinen et al \cite{markku} who  generalized it   for the Gaussian (Schwartz) distribution-valued random variables. This work of Lehtinen et al  can be considered  as the starting point of the  infinite-dimensional  Bayesian inverse problems.  The  case of Banach space-valued  Gaussian random variables was later  considered by Luschgy \cite{lus}.    In these works,  the expression of the  posterior mean is obtained by using  the equivalence between   statistical independence of Gaussian random variables and  their orthogonality in $L^2(P)$. The key factor is the orthogonal random series expansion of the   Gaussian observation -- a method  used by Grenander \cite{grenander}, and even by Poincar\'e \cite{poin}.   Cox \cite{cox} applied Gaussian separable Banach space-valued
unknowns in a linear regression problem with additive Gaussian noise. The approach of Cox differs from that of Mandelbaum  since it uses the generalized Bayes formula rather than the special  properties of Gaussian random variables
 (see Proposition 2.1 in \cite{cox}).   An abstract formulation of Bayesian statistical inverse problems for  Souslin space-valued random variables was given  by Piiroinen \cite{petteri}, who only required the observation and the 
unknown to be Souslin space-valued random variables, thus allowing nonlinear direct problems and more complicated noise terms.

 Little is known about the form of   posterior distributions   in  infinite-dimensional statistical inverse problems  outside  the   Gaussian linear case \cite{sari,markku,lus, mandelbaum,simonithe} and the dominated case with Gaussian noise \cite{helin,siltanen2,zhao}.  When $F$ and $G$ are complete separable metric spaces, a result of Macci  \cite{macci} tells that  the Lebesgue decomposition of the posterior distribution  with respect to the prior distribution contains a nontrivial singular part  in  undominated cases. Namely,   the Lebesgue decomposition of the posterior distribution with respect to the prior distribution  is of the form  
   $
  \mu (\cdot ,y)= \mu ^{(ac)}(\cdot,y) +\mu^{(s)}(\cdot,y),$
  where the absolutely continuous part $\mu^{(ac)}(\cdot,y)$ is determined by  the absolutely continuous part of $\mu_{Y|X}(\cdot ,x)$ with respect to $\mu_{Y}$ through the equations 
 \begin{eqnarray*}
 \mu^{(ac)}(U,y)= \int_U  \frac{d\mu_{Y|X}^{(ac)}(\cdot , x)}{d\mu_{Y} }(y) d\mu_X(x), 
 \end{eqnarray*}
 where   $ U\in \mathcal F$.  Moreover,  the singular part $\mu^{(s)}(\cdot,y)$ is determined by  the singular part of $\mu_{Y|X}(\cdot , x)$ with respect to $\mu_Y$ through the equations
 
 \begin{eqnarray*}
  \mu^{(s)}(U,y)= \frac{d \left( \int_U  \mu_{Y|X} ^{(s)} (\cdot,x)  d\mu_X \right) }{ d\mu_Y}(y),   \; U \in \mathcal F,
  \end{eqnarray*}
  from which one chooses  a  regular version. We remark  that in such undominated cases, one may expect  to meet  some surprises.  The posterior distribution  presents then some things that seemed to be  {\it a priori} impossible.

The  extraction of information from  the posterior distribution involves   decision theory, including  point estimation  and hypothesis testing (for the general  description of the Bayesian decision theory, see   \cite{schervish}).      A decision theoretic view towards   Bayesian inversion  is given in  \cite{evans}, where  one performs the estimation of the (separable Banach space-valued) unknown  by  first fixing the prior distribution of the unknown $X$ and choosing a loss function $\ell:F\times F\rightarrow \mathbf R$ that penalizes the  inaccuracies  in the  estimates  of the unknown,  and then choosing the so-called Bayes  estimator $\widehat X:G\rightarrow F$, which is a deterministic function  that gives  the smallest  averaged loss   $\mathbf E [ \ell(\widehat X(Y),X)]$. This is equivalent  to taking  as  each $\widehat X(Y(\omega))$ the value $d$ that minimizes the posterior expected loss  $ \mathbf E [ \ell(d,X)|\sigma(Y)](\omega)$.  

Common point estimators in  finite-dimensional statistical inverse problems  are the maximum a posteriori (MAP)  estimator  and  the conditional mean (CM) estimator (i.e. the posterior mean) \cite{kaip}.   The CM estimator $\widehat X(Y)=\mathbf E [X|\sigma(Y)]$ minimizes the posterior risk for the squared error loss function $\ell(x',x)=  | x'-x| ^2$ (when $X:\Omega \rightarrow \mathbf R^n$ is suitably integrable) \cite{kaip}.  Conditional means have appeared also in the framework of infinite-dimensional statistical inverse problems  \cite{helin,sari,siltanen2,siltanen, markku,lus,mandelbaum}. However, the decision-theoretic justification  is  often neglected, and the conditional mean is  reported just as a typical value of the posterior distribution.   
Other notions of typical values for  distributions on  separable metric spaces  were considered by 
Frech\'et \cite{frechet}.

 The  mean of a locally convex Hausdorff topological  vector space-valued random variable  can arise  from different definitions, depending on the   space in question. In general,   the (weak) mean of a  locally convex  Hausdorff topological  vector space-valued random variable $X$  is a vector $m\in F{''}$  (or more generally,  $m$   in the algebraic dual space of $F'$)  such that $\langle m,\phi \rangle_{F{''},F'}=\mathbf E [\langle X,\phi \rangle_{F,F'}]$ for all $\phi \in F'$  $P$-a.s.  (see   \cite{bog}).  Such notion of vector-valued integration was developed by 
Pettis \cite{pet}  in 1933  for reflexive separable Banach spaces $F$.  Gelfand  used  a similar  definition  for  distribution-valued random variables (see   \cite{gel}).  The Pettis-Gelfand integral was generalized for quasi-complete Souslin space-valued functions  by Thomas \cite{thomas}.    For  Banach-space valued random variables having integrable norm,  a  mean can  be defined also  as the Bochner integral $m=\int_F xd\mu_X(x)$, introduced  in early 1930's by Bochner  (see   \cite{diestel} and references therein).

  When the posterior distribution $\mu(\cdot, y)$  is known for a given sample $y$ of $Y$,  the (weak) conditional mean $m\in F''$ is a vector that satisfies  $\langle m,\phi \rangle_{F'',F'}= \int \langle x,\phi \rangle_{F,F'} \mu(dx,y)$ for all $\phi \in F'$.   When  $F$ is a separable reflexive Banach space and $\Vert X \Vert$ is integrable, the same  posterior mean can also be defined as the Bochner integral  (see Proposition V.2.5 in \cite{neveu}).

 We remark that the weak posterior mean  $E[ X|\sigma(Y)]$   is  a Bayes  estimator in a weak  sense i.e.   it gives the smallest averaged loss for the family of loss functions $ \ell_\phi(x,x')= |\langle x-x',\phi\rangle |^2$, where $\phi\in F'$.  Franklin \cite{frank} used such requirement, when he defined the 
best linear estimator in a Gaussian linear inverse problem in 1970.  An  earlier  approach to the best linear estimator in function-valued Gaussian case was given  by   Grenander in 1950 (see Chapter 6 in \cite{grenander}).    He considered a Gaussian linear regression problem and identified the best linear estimator (with respect to the  pointwise squared error loss $\ell(X(t), \widehat X(t))=|X(t)- \widehat X(t)|^2$)   with the   posterior  mean. Grenander used infinite-dimensional observations, but he made simultaneous inferences on only  finitely many values of the unknown function.   Moreover, he required, but not proved, the regularity of the conditional probabilities. In this sense, his  approach to the posterior means was still far from the description of Mandelbaum from 1984 \cite{mandelbaum}.  Remark,  that the technique of  estimating  the value of $X(t)$ on the basis of infinitely many observations is still the standard in the modern filtering theory \cite{oek}.
 Also  in the Bayesian density estimation,   the estimation  is sometimes carried   out  in  either in the  form  $\widehat X (t) = \mathbf E [X(t)| Y_1,...,Y_n]$, where $X$ is  the unknown probability density function  on, say  [0,1] (see \cite{whittle2}),   or in the form $\widehat X(U)=\mathbf E[X(U)|Y_1,...,Y_n]$, where  the sets $U\subseteq [0,1]$ are Borel set and $X$ is  an unknown random probability measure (see    \cite{muller} and Proposition 4.2.1 in \cite{gho}).  The  density  estimator  $\widehat X$  is a Bayes estimator  with respect to the squared error loss function for each $t$ or for each Borel set $U$, respectively.  An other option is to use a  weighted $L^2$-loss function   $\ell(\widehat X,X)=\int_0^\infty |\widehat X(t)- X(t)|^2 dw(t)$ \cite{fer}. The two estimators coincide when   
   $X$ is suitably integrable.

In the works of Mandelbaum \cite{mandelbaum} and Luschgy  \cite{lus}, the space $F$ is a Hilbert or Banach space, and the posterior mean  is defined as a Bochner integral.   However,  the emphasis is on the Gaussian nature of the prior, and the posterior mean is  calculated as   $\mathbf E [\sum_{i=1^\infty } X_i e_i |\sigma(Y)]= \sum_{i=1}^\infty \mathbf E[X_i|\sigma(Y)] e_i$.  Similar approach appears in  \cite{sari,markku}  for the distribution space, where the posterior mean is defined in the weak sense. The  weak definition of the posterior  mean is used  also in \cite{siltanen} for the space $C([0,1|)$.     In \cite{helin,siltanen2}, the conditional mean of a separable Banach space-valued random variable is defined as a Bochner integral with respect to the posterior distribution.   Before Luschgy, Krug  \cite{krug} determinded  the posterior mean of a separable Banach space-valued Gaussian unknown in a linear Gaussian case, but he   assumed  that the given observation was finite-dimensional.

We remark that when  $F$ is a Hilbert space, one can take $\ell(x',x)=\Vert x-x'\Vert_F^2$ as the loss function that gives the CM estimator.  As in the finite-dimensional case, the  main point is that 
\begin{eqnarray*}
&&\mathbf E[ \Vert \hat X(Y)- X\Vert_F^2]   =\mathbf E[ \Vert \hat X(Y)-E[X|\sigma(Y)]\Vert_F^2 ] +\\
&&\mathbf E [ ( \hat X (Y)- E[X|\sigma(Y)],  E[X|\sigma(Y)]- X  )_F]+
 \mathbf E [\Vert X \Vert_F^2], 
 \end{eqnarray*}
 and the additional difficulty is just in  checking that $\mathbf E [(f,X)| \sigma(Y)] = (f,\mathbf E [X|\sigma(Y)])$. 
 Such loss functions   have been used  in the regression problem  for   the Gaussian mixture priors when  $F=L^2([-1,1])$  \cite{zhao}.  Instead of an $L^2$-loss function, Abramovich et al \cite{silve2}  used 
an $L^1$-loss function in a regression problem  for a discretized Besov space-valued unknown.
We note that a common approach in the regression problem is to  present only the Bayes estimates instead of the whole posterior distribution. 

 Luschgy \cite{lus} made an (unproven) remark that for  Gaussian  posterior distributions  the conditional mean is the Bayes estimator for every symmetric quasi-convex (measurable) loss function $\ell(x,x')=\ell(x-x')$. A proof can be found in   \cite{bul}, where it is derived  from the Anderson property of Gaussian measures (for the property, see   \cite{lewa}).

In finite-dimensional spaces,   the MAP estimator can be interpreted as a limit of  Bayes estimators for  the 0-1-valued losses $\ell_\epsilon (x',x)= 1_{ F\backslash \overline B(x,\epsilon)} (x')$, where  $\epsilon\rightarrow 0$ \cite{robert}.   Here $\overline B(x,\epsilon)$ is the closed ball in $F$  that is centered at $x$  and has radius $\epsilon$.   Lassas and Siltanen \cite{siltanen} showed that    MAP estimates can behave inconsistently as dimensionality of the   unknown increases, even though the posterior distributions converge at the same time.    In their example, the MAP estimates  actually vanish at the limit, regardless of the given observation. Similar result is proved in \cite{helin2} for a hierarchical edge-preserving prior.  The MAP and CM estimates coincide for the finite-dimensional Gaussian priors, and numerical results demonstrate 
that they can practically coincide  for the finite-dimensional approximations of  Besov-priors  \cite{siltanen3}.  
   Cotter et al   \cite{cotter}  discussed MAP estimation in the  context of infinite-dimensional Bayesian problems. They    showed that there exists a minimizer for a  penalized log-likelihood function, which has similar form as in the case of finite-dimensional Gaussian unknown.  However, the  conditions that would relate the penalized log-likelihood function to any posterior density were  omitted in \cite{cotter}, which leaves open the question what connections  the minimizer has  to the  infinite-dimensional posterior distributions. Recalling the result of Lassas and Siltanen \cite{siltanen} arises at least  some caution. Another attempt towards MAP estimation with infinite-dimensional Gaussian priors  is given by Hegland \cite{hegland}. Unfortunately,  the proof of Proposition 1 in \cite{hegland} is not rigorous, as it  involves subtraction of two  numbers that are 
 infinitely large with probability 1 (i.e. the Cameron-Martin norms of arbitrary vectors in the space of the unknowns).   
  
In infinite-dimensional statistical inverse problems the hypothesis testing has been largely neglected, although several interesting question could be raised. For example,  Fitzpatrick \cite{fitz} has made an initiative  on  testing if the evidence supports the homogeneity of  the unknown diffusion coefficient.  Hypothesis testing  was proposed also for some  nonparametric statistical inverse problems in \cite{biz} within the classical framework. However, it was  pointed out in  \cite{biz} that the problems can  be similarly handled also by the  (finite-dimensional) Bayesian methods but this remark is not elaborated further.

  Another approach to  exploiting  posterior distributions was   given by Piiroinen \cite{petteri}.  He  interpret the  posterior distributions  as   statistical measurements, which allowed comparisons  of   information contents of  different posterior distributions. The result is  especially useful in experimental design \cite{koodi}.

\subsubsection{Posterior consistency}

 The consistency of the posterior distributions  (with respect to repeated independent observations) is closely connected to the 
 uniqueness of the deterministic inverse problem of determining $x$ from  $y=L(x)$. 
  The pioneering work of Doob \cite{doobm} on martingales  touched   the question of consistency of  the posterior distributions. Doob's results imply that under model identifiability (i.e. the measures $\mu_{\varepsilon+L(x)}$ are different for different $x\in F$)  the posterior  distributions would   concentrate (in the weak topology of measures) on the true unknown $x_0$   $\mu_{L(x_0)+\varepsilon}$-almost surely for $\mu_X$-a.e. $x_0$  when  infinitely many i.i.d. observations would be available.   The consistency of the posterior distribution is an important topic because it shows that enough data will  guide  a Bayesian scientist almost surely to the true answer.      The words $\mu_X$-a.s. made Doob's approach slightly impractical as they left open the frequentist case where the observations are not samples of $L(X)+\varepsilon$ but samples of 
 $L(x)+\varepsilon$ for some fixed $x$.   Freedman  \cite{freed} demonstrated that inconsistency  could  hold on  topologically large sets.  The problem was approached by  Schwartz \cite{lorraine} who  described   a  set of unknowns $x$ for which  consistency  holds $\mu_{\varepsilon+L(x)}$-almost everywhere under some decision theoretic conditions  and domination (i.e. all measures $\{ \mu_{\varepsilon+L(x)}: x\in F\}$ are assumed to be absolutely continuous with respect to some common $\sigma$-finite measure). The required property is  the positive prior probability  of all Kullback-Leibler neighborhoods of the   unknown $x$.    Consistency has  been studied also  in other topologies, beside of the weak topology.  Barron et al \cite{barron}  proved consistency in Hellinger distance. Summaries of consistency result in density estimation  can be found  in \cite{diaco,dia,xi}. The case of Gaussian regression has been studied in \cite{vaart}, where certain probabilities  are shown to converge.  
 
  Consistency  issues in inverse problems are discussed in 
\cite{evans}. For our statistical inverse problem,  the consistency corresponds to observing one sample of  $Y= L(X)+ \frac{1}{\sqrt n} \varepsilon$, where  $n$  represents the number of i.i.d. observations of $L(x)+\varepsilon$.  The works  of Hofinger and Pikkarainen \cite{pikkarainen3,pikkarainen2}, and Neubauer and Pikkarainen  \cite{pikkarainen}  on  finite-dimensional   Gaussian statistical inverse problems  concern  the question of posterior consistency. They studied the convergence of posterior distributions and the  posterior means  in  linear  Gaussian inverse problems for finite-dimensional   random variables as the variance of the noise decreases.   In particular, it was  shown in \cite{pikkarainen3}  that the posterior distributions given   observed values  $\widetilde Y_{\delta_n}=Lx+\delta_n \varepsilon(\omega)$  of $Y_{\delta_n}= LX+ \delta_n \varepsilon$  for a  sequence  $\delta_n\rightarrow 0$,  converge to the  point mass on the true value $x$  in the Ky Fan metric, assuming that also the prior distribution are modified appropriately.      Hofinger and Pikkarainen  \cite{pikkarainen} studied  posterior  convergence rates for finite-dimensional approximations of  Hilbert-space-valued random variables when the approximation level increases in certain manner as the noise level $\delta_n$ approaches to zero.   However, the convergence was shown only for unknowns in  an  {\it a priori} zero measurable set (the Cameron-Martin space of the prior distribution).     Also   Florens and Simoni \cite{simoni}  studied the  posterior consistency for the  infinite-dimensional  linear Gaussian inverse problems when  the variance of the noise diminishes.  They were able to show the posterior consistency  if the posterior measures  with respect to the weak topology (and give estimates for  the speed of convergence of the posterior means)  by assuming that the  direct theory is regular enough and the  prior distribution  depends suitably on   the noise level.

 Another  convergence topic that   has received more attention  in statistical inverse problems is the posterior convergence for approximated unknowns  and observations \cite{helin,sari,siltanen2, siltanen,petteri,stuart}. This case has been discussed above in the introduction.  
 
 Almost all known  convergence results for posterior distributions  \cite{maen,helin,sari,roininen,siltanen2, siltanen, stuart} are  based on the  known form of the posterior distribution. There are also some measure-theoretic  approaches  for convergence of conditional expectations.   The results of  G\"anssler and Pfanzagl \cite{gans} showed  that the  conditional expectations  $E[1_U(X_n)|Y]$ converge  when the joint distributions of the observation and the approximated  unknowns are dominated  by some  $\sigma$-finite measure and the corresponding  Radon-Nikodym densities converge almost everywhere. Furthermore, they also showed that there exists a regular version for which the convergence holds almost surely in variation. Here one should pay attention to the fact that the conditioning $\sigma$-algebra does not depend on $n$ which is not satisfactory from the point of view of numerical solutions of statistical inverse problems.   Landers et al \cite{land} generalized this result for monotonic sequences of   conditioning $\sigma$-algebras  $\sigma(Y_n)$.  It should be noted that in statistical inverse problems, the  $\sigma$-algebras $\sigma(L(X_n)+\varepsilon )$ are usually not increasing.   Krikkeberg  \cite{kri} proved a martingale type convergence theorem for   not necessarily  monotonic  $\sigma$-algebras, but his conditions seem to be  too abstract  for the statistical inverse problems in the present form. A reformulation of his conditions in terms of random variables $(X_n,Y_n)$ would give valuable information on the  almost sure posterior convergence in  the general undominated case.  Goggin \cite{goggin} and Crimaldi et al \cite{crimaldi1,crimaldi2} studied conditions under which  the  convergence of $(X_n,Y_n)$ to 
 $(X,Y)$   implies the convergence of the conditional expectations $\mathbf E [f(X_n)|Y_n]$ to  $\mathbf E [f(X)|Y]$ in distribution or in probability.  Their results are not satisfactory for statistical inverse problems, since they do not say anything about the almost sure convergence of posterior distributions for fixed samples of $Y_n$, but the necessary conditions in \cite{crimaldi2} are valid for also a.s. convergence.  For example,  the results in \cite{crimaldi2}  imply that setwise convergence of the prior distributions  is necessary for the setwise convergence of the posterior distributions (given samples of $Y_n$).  We remark  that samples of $Y$, not  $Y_n$, are  usually given.    For the undominated case, a result of Berti et al \cite{berti} somewhat simplifies the study of posterior convergence. Under quite general conditions,  their result reduces the problem of  almost sure weak convergence of  random measures to the study of only countably many sequences of  conditional expectations. Piiroinen  gave a  sufficient condition that guarantees  the convergence of  posterior distributions  when  unknowns and  observations are approximated \cite{petteri}.
The emphasis in his results was on obtaining with probability 1 the posterior  convergence  for the  Souslin space-valued  approximated unknowns   given   samples of multi-indexed  observations of the corresponding   approximated unknowns.    His proof relies on improving the a.s. convergence of the conditional expectations of  each function of the type   $f(X_n)g(X)$, where $f$ and $g$ are continuous and bounded, to almost sure weak convergence of posterior distributions.
 
  A  concept close to the posterior convergence  is the so-called discretization invariance,  which was first used by Markku Lehtinen in  1990's  (see    \cite{siltanen2}).  It  asks that the prior knowledge is consistent at all discretization levels  and aims to the stability  of   posterior knowledge on different discretization levels. Definitions for discretization invariance in   statistical inverse problems are given in \cite{siltanen2,siltanen}.    In \cite{siltanen2}, Lassas et al defined  a proper linear discretization  $X_n=P_nX$ of a Banach space-valued  random variables $X$, where  $P_n$ are bounded linear operators  on the Banach space $F$ having finite-dimensional ranges and  the  random variables $\langle P_n X,\phi\rangle $ converge in distribution to $\langle X,\phi\rangle$ for all $\phi\in F'$.     Gaussian priors  and   Besov space priors were shown to be discretization invariant in \cite{siltanen2} in the sense that they have  proper linear discretization for which the conditional mean estimates converge.     An important example  was studied by   Lassas and Siltanen \cite{siltanen} who showed that  the finite dimensional total variation priors converge   to a Gaussian measure and the corresponding CM estimates converge to the CM estimate obtained with a Gaussian prior. The total variation priors are  not discretization invariant as   the   finite-dimensional prior distributions  lead to  unwanted effects.  
  A special method for obtaining stable posterior knowledge was 
suggested by Kaipio and Somersalo \cite{ks}, who proposed the approximation  error approach for  statistical inverse problems. 
 In approximation error approach, the conditioning random variable 
 $Y=L(X)+\varepsilon $ is written as $Y=L(X_n)+ (L(X)-L(X_n))+\varepsilon$, where $L(X)-L(X_n)$ is taken to be an additional noise term $\tilde \varepsilon$. For example, if $X$ is Gaussian and $X_n=P_nX$, where $P_n$ are  linear projection operators,  the CM estimators take a  consistent form $\mathbf E [X_n|Y]=P_n \mathbf E[X|Y]$.  The problem becomes computationally more tractable if  $X_n$ and $\tilde \varepsilon $ are statistically independent in which case only  the   distribution of $\tilde \varepsilon $  needs to be additionally  determined. This condition is often forced on $\tilde\varepsilon$ together with a numerically feasible approximated   distribution \cite{ks,tanja}.

  \section{Conditional probabilities and posterior distributions}\label{sec2}

\subsection{Solution of the statistical inverse problem}

We define what we exactly mean by a statistical inverse problem and its solution.  We begin by recalling the definition of the conditional expectation.  

 The conditional expectation  of $f\in L^1(\Omega,\Sigma,P)$ given  a
 sub-$\sigma$-algebra $\Sigma_0\subset \Sigma$ is a $\Sigma_0$-measurable function $\mathbf E[f|\Sigma_0]$ such that 
 \begin{equation*}
\int_A f dP=\int_A \mathbf E[f|\Sigma_0] dP
 \end{equation*}
 for all $A\in \Sigma_0$. 
  Conditional expectations exist  due to  the Radon-Nikodym theorem as  the densities of the (signed) measure  $fdP$   with respect to the  measure  $P$ on $\Sigma_0$, but they are only defined up to sets  $N\in \Sigma_0$ of $P$-measure zero. 
  We denote $\mathbf E[\cdot | Y]= \mathbf E[\cdot | Y^{-1}(\mathcal G)]$.

\begin{definition}\label{si}
Let   $(\Omega,\Sigma,P)$ be a complete probability space. Let $F$ and $G$ be two Souslin spaces equipped with their Borel  $\sigma$-algebras $\mathcal F$ and $\mathcal G$, respectively.  Let $X:\Omega\rightarrow F$ and    $Y: \Omega\rightarrow G$ be measurable mappings. We  call a mapping $\mu:\mathcal F\times G\rightarrow [0,1]$   {\it  a solution of  the statistical inverse problem of estimating the distribution of  the unknown $X$ given the observation  $Y$} if 
\begin{enumerate}
\item  {   $\mu(U,Y(\omega))= \mathbf E[1_U(X)| Y](\omega)$  $P$-almost surely for every $U\in \mathcal F$, }
\item { $y\mapsto \mu(U,y)$ is $\mu_Y$-measurable for every $U\in \mathcal F$, and}
\item{ $ U\mapsto \mu(U,y)$ is a probability measure on 
$(F, \mathcal F)$ for   every $y\in G$.}
\end{enumerate}
 The distributions  $ \mu (\cdot, y)$ are called {\it posterior distributions} of $X$ given$Y=y$.   \end{definition}
 Strictly speaking, the posterior distributions are defined {\it a posteriori} of the observation  $Y(\omega)$   but we feel that there is no  harm in calling    $\mu(\cdot,y)$ posterior distributions also for  $y\notin R(Y)$ since $\mu_Y(G)=1$.

The solution is just a regular  conditional distribution of $X$ given the  sub-$\sigma$-algebra $Y^{-1}(\mathcal G)$, where the regularity holds  in the sense of Doob  i.e. the solution $\mu$ is $\mu_Y$-measurable in the second variable (see Remark 10.6.3 in \cite{bogm} for a further discussion).    
    The nature of the mapping  
  $\omega\mapsto \mu(U,Y(\omega))$, which need not be $\sigma(Y)$-measurable, is verified in the following simple lemma. 
  
  \begin{lemma}\label{leu}
Let  $(G,\mathcal G)$ be a measurable space. Let   $Y: \Omega\rightarrow G$ be a measurable mapping from a complete probability space $(\Omega,\Sigma,P)$ into $G$.
 If $f:G\rightarrow \mathbf R $ is a $\mu_Y$-measurable function  then $f(Y)$ is a random variable on $(\Omega,\Sigma,P)$, and   $\mathbf E [f(Y)]= \int  f(y) d\mu_Y(y)$. Moreover, if $\tilde f:G\rightarrow \mathbf R$ is a Borel measurable function such that $f=\tilde f$  $\mu_Y$-a.s., then   $f(Y(\omega))=\tilde f( Y(\omega))$  $P$-almost surely 
 and $\mathbf E [ f(Y)|\Sigma_0](\omega)= \mathbf E [\tilde f(Y)|\Sigma_0](\omega)$  $P$-almost surely  for any  sub-$\sigma$-algebra 
 $\Sigma_0\subset \Sigma$.  
 \end{lemma} 
\begin{proof}
Every  $\mu_Y$-measurable function  has  a Borel measurable version    (see Proposition  2.1.11 in \cite{bogm}). Denote with  $\tilde f$  a Borel measurable version of $f$.
 The set  $N=\{ y\in G :  f(y)\not=\tilde f(y)\}\in \mathcal G^{\mu_Y}$ is then $\mu_Y$-zero measurable and, by definition, there exists a $\mu_Y$-zero measurable Borel set $B\in \mathcal G$ such that $N\subseteq B$.  Especially,  $Y^{-1}(N)\subset Y^{-1}(B)$, which has $P$-measure $P(Y^{-1}(B))=\mu_Y(B)=0$ so that also $Y^{-1}(N)$ belongs to the complete $\sigma$-algebra  $\Sigma$  and has   $P$-measure zero. Therefore,
$$
f(Y(\omega))=\tilde f(Y(\omega))
$$
  P-almost surely.   By the completeness of $\Sigma$, also the mapping $\omega\mapsto f (Y(\omega))$ is $\Sigma$-measurable.  
By the almost  sure equivalence of the functions, we get
$$
\mathbf E[f(Y(\omega))] = \mathbf E [\tilde f(Y(\omega))]= \int \tilde  f(y) d\mu_Y(y) = \int f(y) d\mu_Y(y).
$$   
The conditional expectations of equivalent random variables coincide, since they have the same integrals over 
$\Sigma_0$-measurable sets. 
\end{proof}

From the point of view of the  posterior analysis, Condition 3 of Definition \ref{si} may  give a false sense of security.  Any $\mu$ that satisfies  Conditions 1 and 2 but is a probability measure only for $\mu_Y$-a.e. $y$ can be redefined on a negligible set  in such a way that it is a    solution.  For example, if $N$ is the $\mu_Y$-zero measurable set that contains all $y$'s  for which $\mu(\cdot,y)$ is not a probability measure,  we  may  redefine $\mu(U,y)$  as  $1_U(x_0)$ for some fixed $x_0\in F$ and  all $y\in N$. Then $\mu$ satisfies   Conditions 1, 2 and 3, but   $\mu(\cdot,y)$  is not  related to  the unknown when  $y\in N$.

 We briefly compare the solution $\mu$ with other formulations of Bayesian inverse problems.    Clearly,  any regular conditional distribution $\mu$  of $X$ given $Y$ (such that $y\mapsto \mu(U,y)$ is Borel-measurable for any $U\in \mathcal F$)  qualifies as a solution. Especially,  posterior distributions obtained by the Bayes  formula  \eqref{bab} on $\mathbf R^n$  for positive continuous probability densities form a solution of the form $\mu(U,y)=  \int_U D(x|y)dx$ \cite{kaip}.   The Gaussian conditional probabilities   in \cite{sari, markku, lus,mandelbaum} are also solutions that are allowed to be $\mu_Y$-measurable in the sense of Condition 2.   Our approach is similar to the work of  Piiroinen \cite{petteri}, where a general formulation of the statistical inverse problem for Souslin space-valued random variables first appeared.  The difference  is that Piiroinen chose the  posterior probabilities $\mu(U,y)$   to be universally measurable with respect to the second variable, that is,    $m$-measurable for any  finite  Radon measure $m$ on $(G,\mathcal G)$  whereas  we prefer  to take all  $\mu_Y$-measurable versions as  solutions, since it helps to avoid the  somewhat artificial modifications of $\mu_Y$-measurable functions (encountered for example in the Gaussian case \cite{markku}) to any universally measurable or  $\mathcal G$-measurable functions.   Lassas et al \cite{siltanen2}  used   a different  approach  where the posterior distribution was obtained by defining reconstructors.     A mapping  $ y\mapsto \mathcal R(g|y)$ is called  a reconstructor of $g\in L^1(\mu_X)$ (more generally, a Bochner integrable $g$)  given the observation $Y$  if  $\mathcal R(g,Y(\omega))= \mathbf E [g(X)| Y^{-1}(\mathcal G)](\omega)$  almost surely \cite{siltanen2}.  The concept of a reconstructor is more elemental  than our solution.   However, the reconstructors  that were used for  solving the statistical inverse problem in \cite{helin,siltanen2}   were chosen to be more regular.  They depend continuously on observations and  satisfy also   Conditions 1 and 3.  Hence,  they form a regular conditional distribution and are especially solutions  in the sense of Definition \ref{si}.   A common point of the reconstructor and our solution is that both are defined for all $y\in G$, not   only for samples  $Y(\omega)\in R(Y)$. However, the simplicity of the reconstructors comes with some disadvantages. Namely, if  the reconstructor   of the unknown $X$ does not originate from a regular conditional distribution,   some power of the Bayesian inference is lost, as there is no posterior probability distribution to draw from.     Furthermore,   two reconstructors $\mathcal R_1$ and $\mathcal R_2$  of the same function $f$   may differ on a "large" set  $Y(N)\subset G$, where  $N= \{\omega: \mathcal R_1(f,Y(\omega))\not= \mathcal R_2(f,Y(\omega)) \} \in \Sigma$  has probability zero. The set   $Y(N)$ might not belong to $\mathcal G^{\mu_Y}$  and $Y(N)$ may  have positive $\mu_Y$-outer measure.  Indeed, we provide a simple example of this situation  with the help of the so-called image measure catastrophe (see p. 30 in \cite{schwarz}).  Let $U\subset [0,1]$ be a nonmeasurable set such  that the  Lebesgue outer measure  $m^*(U)=1$.  Let $(U,\mathcal B_U([0,1]), m|_U)$ be the restriction of the Lebesgue measure $m$ on $U$ i.e. the Borel $\sigma$-algebra $\mathcal B_U([0,1])$ contains all sets 
$B\cap U$, where $B\in \mathcal B([0,1])$, and for such sets  $m_U(B\cap U)= m(B\cap \widetilde U)$, where 
$\widetilde U\in \mathcal B([0,1])$ is such that $m^*(U)=m(\widetilde U)$, say $\widetilde U=[0,1]$. Take $(\Omega,\Sigma,P)$ to be the completion of 
$(U,\mathcal B_U([0,1]), m|_U)$ and $F=G=[0,1]$ equipped with its Borel $\sigma$-algebra.   Let  $Y:U\rightarrow [0,1]$ be the identity and 
take $X=Y$.  Then the image measure  $\mu_Y$ is the Lebesgue measure on $[0,1]$.  Moreover, the conditional expectation  of  the measurable function $\omega\mapsto 1_{[0,1]}(X(\omega))$  given  
$\sigma(Y)$ is 
$$
\mathbf E [1_{[0,1]} (X)|\sigma(Y)](\omega)= 1_{[0,1]}(Y(\omega)) 
$$ 
which is equal to  $1_U(Y(\omega))$ $P$-almost surely.  However,  the  reconstructor $\mathcal R_1(1_{[0,1]},\cdot):y\mapsto 1_U(y)$ is not $\mu_Y$-measurable on  $([0,1],\mathcal B([0,1]))$, and  the two reconstructors $\mathcal R_1(1_{[0,1]},\cdot) $ and $\mathcal R_2(1_{[0,1]},\cdot): y\mapsto 1_{[0,1]}(y)$ differ on the set $N_0:=[0,1]\backslash U$ which has  positive $\mu_Y$-outer measure.  Condition 2  helps us to avoid this small shortcoming. If the reconstructors are $\mu_Y$-measurable,   the set $N_0=\{y\in G: \mathcal R_1(f,\cdot)\not= \mathcal R_ 2(f,\cdot)\}\in \mathcal G^{\mu_Y}$ and  $\{\omega: \mathcal R_1(f,Y(\omega))\not= \mathcal R_2(f,Y(\omega)) \} =  Y^{-1 }(N_0)$. Then   $N_0$ has zero $\mu_Y$-measure.

A regular conditional distribution  is  not unique in general because of the  non-uniqueness of  the  conditional expectations.  For our theoretical considerations, the following concept (adapted from \cite{bogm} in  context of  regular conditional measures) is useful.

\begin{definition}\label{esse}
We say that a  solution  $\mu$ of  the statistical inverse problem of estimating the distribution of  $X$ given  the observation $Y$ is {\it essentially unique} if for any other   solution   $\tilde \mu$  of the same statistical inverse problem  there exists a set  $C=C(\mu,\tilde \mu)\in \mathcal G^{\mu_Y}$ with  $\mu_Y(C)=1$ such that $\tilde \mu$ agrees with $\mu$   on  $\mathcal F\times C$.  Similarly, we say that the posterior distribution $\mu(\cdot,y)$ is essentially unique if $\mu$ is essentially unique.
\end{definition}

In other words, an essentially unique  solution $\mu$  may be arbitrary on  the sets of the form $\mathcal F\times N$, where $N\subset G$ is  a set of $\mu_Y$-measure zero. In a sense,  this makes  the  posterior  distribution  $\mu(\cdot,Y(\omega))$ a  relevant estimate of the distribution of $X$ with probability 1.

Next, we recall some  results on the existence and essential uniqueness of regular conditional  distributions in Souslin spaces.   The existence of regular conditional distributions of $X$ given $Y$ has been shown in  Lemma 4.2 of \cite{petteri} (by using  the definition of the Souslin space and the existence of regular conditional distributions on Polish spaces, leading to a universally measurable kernel $\mu$),   and in Example 10.7.5 of  \cite{bogm}, where also  the essential uniqueness has been verified. The present "extension"  covers  $\mu_Y$-measurable  solutions. The condensed proof is included  only to support the last sentence, which   provides some motivation for  the  main results of this work. Namely, the definition of the conditional expectation may give the impression that we need  to  specify some  random variable   $Y$ among all equivalent random variables  for  determining the conditional expectation of $1_U(X)$  when an observation  $y=Y(\omega_0)\in G$ has occurred. This is not true as a weaker  description of  $Y$ and $X$  suffices.

\begin{theorem}\label{ex}
Let  $(F,\mathcal F)$ and $(G,\mathcal G)$ be two measurable   spaces. Let $X$ be an $F$-valued random variable and $Y$ be a  $G$-valued random variable on a complete probability space $(\Omega,\Sigma,P)$.   If $F$ and $G$ are Souslin spaces equipped with their Borel $\sigma$-algebras, then there exists an essentially unique   solution  $\mu:\mathcal F\times G \rightarrow [0,1]$ of  the statistical inverse problem of estimating the distribution of the unknown $X$ given the observation $Y$. 

The  values $ \mu(U,y)$ are  determined by the joint distribution $\mu_{(X,Y)}$ of $X$ and $Y$ for all $U\in \mathcal F$ and  $\mu_Y$-almost every $y\in G$.
\end{theorem}

\begin{proof} 
First we show that  for each $U\in \mathcal F$  there exists a solution  $\mu_0(U,\cdot): G\rightarrow [0,1]$  such that $y\mapsto \mu_0(U,y)$ is  Borel-measurable and     $\omega\mapsto \mu_0(U,Y(\omega))$ is a conditional expectation of $1_U(X)$ given 
$Y^{-1}(\mathcal G)$.  

   Consider the measure space $(F\times G, \mathcal B(F\times G), \mu_{(X,Y)}) $  and  the sub-$\sigma$-algebra  
    $ \mathcal G_0=\{\emptyset, F\} \otimes \mathcal G$ generated by the   canonical projection $p_2 (x,y)=y$ to the second variable. Recall, that  the direct products of Souslin spaces are Souslin spaces.  Due to the Souslin property of $F\times G$,  there  exists a  conditional measure $\mu_0:\mathcal B(F\times G) \times (F\times G)  \rightarrow [0,1]$ such that 
 $\mu_0(U',\cdot)$ is $\mathcal G_0$-measurable for every $ U'\in\mathcal B(F\times G)$, the measure  $\mu(\cdot, (x,y))$ is a probability distribution  on $\mathcal B(F\times G)$ for every $(x,y)\in G$, and 
  $$
  \mu_{(X,Y)}(U'\cap V')= \int_{V'} \mu_0(U',(x,y)) d\mu_{(X,Y)}(x,y)
   $$
  for every $U'\in  \mathcal B(F\times G) $ and $V'\in \mathcal G_0$  by Corollary 10.4.6 in \cite{bogm}.  
  Let us restrict  $\mu_0 (U',(x,y))$   on sets $U'$ of the form $U\times G$, where $U\in \mathcal F$.   Since $\mathcal G_0$ is trivial with respect to the first variable, we may denote    the restriction with $\mu_0(U,y)$ where $U\in \mathcal F$ and $y\in G$. Especially, 
  $y\mapsto \mu_0(U,y)$ is $\mathcal G$-measurable and 
  $$
P(X\in U \cap Y\in V) =  \mu_{(X,Y)} (U\times V)= \int _{V} \mu_0(U,y) d\mu_Y(y)=\int_{Y^{-1}(V)} \mu_0(U,Y) dP
  $$  
for every $U \in \mathcal F$ and  $ V \in \mathcal G$.  Therefore, $\mu_0:\mathcal F\times G \rightarrow [0,1]$ is 
a solution of the statistical inverse problem of estimating the probabilities of the unknown  $X$ given the observation $Y$.

A  solution $\mu$ is essentially unique since the Borel $\sigma$-algebra of a Souslin space is countably generated (see \cite{bogm}).  Indeed,   suppose that $\mu$ and $\nu$ are two  solutions in the sense of Definition \ref{si}.   For  $U\in \mathcal  F$,   we have that   $ \mu(U,Y(\omega))=  \mathbf E [1_U(X)|  Y^{-1}(\mathcal G)](\omega)=\nu(U,Y(\omega))$  $P$-almost surely.  Then  $\mu(\cdot,y)=\nu(\cdot,y)$  outside some  $\mu_Y$-zero measurable set $N_U\in \mathcal G^{\mu_Y}$,  since
\begin{equation*}  
0=\mathbf E [|\tilde \mu(U,Y)-\tilde\nu(U,Y)|]=   \int | \mu(U,y)- \nu(U,y)| d\mu_Y(y)
\end{equation*}
by Lemma \ref{leu}. Every  countable algebra $\mathcal F_0$  that generates the $\sigma$-algebra $\mathcal F$  is  measure-determining, i.e. measures coinciding on $\mathcal F_0$ coincide on $\mathcal F$ (e.g. Lemma 1.9.4 in  \cite{bogm}). Hence,  the two solutions coincide  except  for $y\in \cup_{U\in \mathcal F_0} N_U$.

Finally, if $\mu$ is any solution then the values $\mu(U,y)$ are determined by the measure 
$\mu_{(X,Y)}$  for all $U\in \mathcal F$ and $\mu_Y$-almost all $y\in G$ since   $\mu$ coincides with 
$\mu_0$ on $\mathcal F\times C$ (by essential uniqueness) and the values of $\mu_0(U,\cdot)$ are actually  versions of the Radon-Nikodym densities of measures  $\mu_{(X,Y)} (U,\cdot) $ with respect to $\mu_Y$ for  $\mu_Y$-almost all $y$. The  distribution $\mu_Y$ is the marginal of 
$\mu_{(X,Y)}$.  
\end{proof}

We have reached  the usual  starting point of nonparametric Bayesian statistics. In a conventional Bayesian experiment, one specifies only   conditional distributions of  $Y$ given  $X=x$ for all values $x\in G$ -- the so-called parametric family of  distributions  or  sampling distributions --  and  the prior distribution  $\mu_X$ on $(F,\mathcal F)$  \cite{flo,schervish}, which together  determine  the joint  distribution of $X$ and $Y$.

\begin{remark}\label{samplespace}
The  choice of the sample  space $(G,\mathcal G)$ of random variable $Y$ is usually not trivial. One might  choose as well a larger  (or sometimes  even a smaller)  space than $G$.  The solutions  of the statistical inverse problem could, in principle, depend  on  the choice of the sample space $(G,\mathcal G)$  since the conditioning $\sigma$-algebra $Y^{-1}(\mathcal G)$ depends on the topology of $G$. But since we are working with the Souslin  spaces this is not the case. Indeed,  if $(G_1,\mathcal G_1)$ and $(G_2,\mathcal G_2)$ are two Souslin  spaces equipped with their Borel  $\sigma$-algebras  and  $i:G_1\mapsto G_2$ is a continuous (or just Borel!)  injection, then, quite remarkably,  $i^{-1}(\mathcal G_2)=\mathcal G_1$. Indeed,  $i^{-1}(\mathcal G_2)\subset \mathcal G_1$ by the continuity of $i$. Moreover,  the image of a Borel set under 
 a Borel mapping between Souslin spaces is a Souslin set  i.e. a Souslin space with respect to the relative topology (see Theorem 6.7.3 in \cite{bogm}). Therefore, $i(G_1)$,  $i(B)$ and $i(G_1\backslash B)$ are all Souslin sets in $G_2$
 for any Borel set $B\in \mathcal G_1$. By injectivity,  $i(G_1\backslash B)= i(G_1)\backslash i(B)$ i.e. the complement of 
 the Souslin set $i(B)$ in the subspace $i(G_1)$ of $G_2$ is  a Souslin set.  By  Corollary 6.6.10 in \cite{bogm},  a Souslin set 
  in a Hausdorff space is a Borel set if also its complement is a Souslin set. Therefore, $i(B)$ is a Borel set in the relative topology of $i(G_1)$.  But each Borel set $i(B)$ in $i(G_1)$ is of the form $i(B)=i(G_1)\cap B'$, where $B'\in \mathcal G_2$. 
 Therefore, $B=i^{-1}(B')$ for some $B'\in \mathcal G_2$ which implies that  $\mathcal G_1= i^{-1}(\mathcal G_2)$.  
      Consequently,  $(i Y)^{-1}(\mathcal G_2)=  Y^{-1}(\mathcal G_1)$  for any  $G_1$-valued random variable $Y$. Therefore,  $\mu_1(\cdot, Y(\omega))=\mu_2(\cdot,  i(Y(\omega)) )$ $P$-almost surely for any solutions  $\mu_1$ and $\mu_2$  of the inverse problems  of estimating the distribution of $X$ given $Y:\Omega\rightarrow  G_1$ and  $i(Y)$, respectively.   If    $\tilde \mu_2$ is a  Borel measurable version of $\mu_2$, then 
      $\tilde \mu_2(U,i(y))$ is $\mathcal G_1$-measurable and 
      $\mu_2(U,y')=\tilde\mu_2(U, y')$,   except possible on some set $N$ such that $\mu_{i(Y)}(N)=0$ which implies that. 
      $\mu_2(U,i(y))=\tilde \mu_2(U,i(y))$  except possibly on the set $i^{-1}(N)$ which has  $\mu_Y (i^{-1}(N))=0$.
      Therefore, $\mu_2 (\cdot, i(\cdot))$ is also a solution of the statistical inverse problem of estimating the distribution of $X$ given  $Y$. We are allowed to mis-specify the Souslin sample space $G_1$ by  Borel injections without altering the essentially unique solution.   In the general case that involves non-Souslin spaces, we only know that $i^{-1}(\mathcal G_2)\subset \mathcal G_1$, where the inclusion may be strict.   As an example, take   $G_1$ and $G_2$ to be the sequence space $\ell^\infty$ where we take $\mathcal G_1$ to be the usual Borel $\sigma$-algebra  with respect to the supremun norm topology (which is not separable) and $\mathcal G_2$ to be the Borel $\sigma$-algebra with respect to the weak topology, and take $i$  to be the identity.  Then $i^{-1}(\mathcal G_2 )\not= \mathcal G_1$ (see  Proposition  2.9 in \cite{vakha}).  
      \end{remark}

\subsection{Partial uniqueness of the solution}\label{sec22}

From practical point of view, the essential uniqueness is not enough since we are given some fixed  observation $y_0\in G$ that might belong to the  set where arbitrariness of  $\mu$ still rules. Our proposal for removing this deficiency of the posterior distributions  is to proceed as in the finite-dimensional case, where $\mu$ is  required  to depend   continuously   on the second variable i.e. the posterior distributions depend continuously on the  observations. The following new concept  turns out to be useful.

\begin{definition}
Let   $(\Omega,\Sigma,P)$ be a complete probability space. Let $F$ and $G$ be two Souslin spaces equipped with their Borel  $\sigma$-algebras $\mathcal F$ and $\mathcal G$, respectively.  Let $X:\Omega\rightarrow F$ and    $Y: \Omega\rightarrow G$ be measurable mappings. Let  $\mu$ be a solution of the statistical inverse problem of estimating the distribution of the unknown $X$  given the observation $Y$. Let $A\subset G $ and let $\mathcal F_0\subset \mathcal F$. We say that  a solution  $\mu$  is {\it $\mathcal F_0$-continuous on $A$} if  the mapping 
$y\mapsto \mu(U, y)$ is continuous on $A$ with respect to the relative topology  for every $U\in \mathcal F_0$.
\end{definition}

Consider a set  $S\subset G$ that contains every point $y\in G$ whose any open neighborhood has positive  $\mu_Y$-measure. On Souslin spaces such a set $S$ is known to coincide with the {\it  topological  support} of $\mu_Y$,  i.e.  the smallest  closed set $S\subset G$   such that $\mu_Y(S)=1$ 
\begin{lemma}\label{support}
Let $\nu$ be a Borel probability measure on a Souslin space $G$. 
The topological support of $\nu$ exists and it  consists of  exactly   those $y\in G$ whose  every open neighborhood   has positive measure.
\end{lemma}
\begin{proof}
See   Theorem 2.1 in \cite{partha}, which generalizes to Souslin spaces, since Souslin spaces are  hereditarily Lindel\"of by Lemma 6.6.4 of \cite{bogm}.
\end{proof}

We obtain   partial uniqueness of the posterior distributions  by using  the  continuity of  solutions on certain subsets of the topological support of $\mu_Y$. We denote with $A^\circ$ the  interior points of $A$.

\begin{theorem}\label{uniq}
Let $F$ and $G$ be Souslin spaces equipped with their Borel $\sigma$-algebras $\mathcal F$ and $\mathcal G$ respectively. Let 
$X$ be an $F$-valued  and $Y$ be a $G$-valued random variable on a complete probability space $(\Omega,\Sigma,P)$. Let  $A \in \mathcal G^{\mu_Y}$  be a subset of the topological support $S$ of $\mu_Y$  such that  either $A\subset  \overline{A^\circ}$  or $\mu_Y(A)=1$. Let $\mathcal F_0\subset \mathcal F$ be a measure-determining class.

All    solutions of the statistical inverse problem  of estimating the probabilities of $X$ given $Y$ that are $\mathcal F_0$- continuous on $A$   coincide on $\mathcal F \times A $.  

\end{theorem}
\begin{proof}
  Assume that   $\mu_1$ and $\mu_2$ are two  solutions that have the described properties.  If $\mu_1\not=\mu_2$  on $\mathcal F\times A$ then there exists  $y_0\in A$ and $U_0\in \mathcal F$  such that $\mu_1(U_0,y_0)\not=\mu_2(U_0,y_0)$,  say  $\mu_1(U_0,y_0)-\mu_2(U_0,y_0)>\varepsilon $.  Since  $\mu_i(\cdot,y_0)$, $i=1,2$,  are measures, the set    $U_0$ can be taken to be from the  measure-determining class  $\mathcal F_0$.

   The   function  $f:A \rightarrow \mathbf R$ defined as $f( y):=\mu_1(U_0, y)-\mu_2(U_0,y)$  is   continuous in the relative topology of $A$  and positive at $ y_0$.  The set   $f^{-1}((\varepsilon ,\infty))$ is therefore a non-empty open neighborhood of $ y_0$  in the relative topology of $A$,  and there exists a non-empty open set   $V\subset G$  such that $V\cap A = f^{-1}((\varepsilon ,\infty))$. The set $V\cap A$ has positive $\mu_Y$-measure. Indeed, if $\mu_Y(A)=1$, then $\mu_Y(V\cap A)=\mu_Y(V)>0$ by Lemma \ref{support}, since $y_0\in V\cap A$ belongs also to the support of $\mu_Y$.  On the other hand, if 
   $A\subset\overline {A^\circ}$, the neighborhood $V$ of $y_0$ contains also points from $A^\circ$. It follows that $V\cap A$ contains a non-empty open set $V\cap A^\circ$. By Lemma \ref{support},  $\mu_Y(V\cap A)>0$.  
    This   implies that  $\mu_1(U_0,y)- \mu_2(U_0,y)>\varepsilon $ on a set  $f^{-1}((\varepsilon ,\infty))\in \mathcal G^{\mu_Y}$ of positive $\mu_Y$-measure. Therefore, it is impossible that the  both  mappings $\mu_1$ and $\mu_2$  satisfy the requirements of Definition     \ref{si}, in particular  the property
     $$
    \int_{f^{-1}((\varepsilon ,\infty))}  \mu_i(U_0,y) d\mu_Y(y) =  P(X \in U_0  \cap  Y\in f^{-1}((\varepsilon ,\infty))),  \; i=1,2, 
    $$
     of conditional expectations does not hold.  Hence, the two solutions  necessarily coincide on $\mathcal F \times A$.
         
\end{proof}

\begin{remark} 
Recall, that a Borel measure  is called strictly positive if it is positive on all non-empty open subsets.  Then the topological support of the measure is the whole space.  When  $\mu_Y$ is strictly positive,
the  partial uniqueness holds on $\mathcal F\times G$  for the solutions that are $\mathcal F_0$-continuous  on $G$.  If  $\mu_Y$ is strictly positive and the solution  $\mu$ is $\mathcal F_0$  continuous on some non-empty open subset $A$   of $G$, we get similarly  the uniqueness   on $\mathcal F \times A$. This situation is often encountered  in finite-dimensional statistical inverse problems, where one usually excludes those $y\in G$ for which the continuous probability density function of   $Y$  vanishes. 
\end{remark}

\begin{remark} \label{samplespa}
The  partial  uniqueness of the  solution is obtained  by fixing the topology of the  space of observations.  However, the topology of a Souslin space is a slightly ambigiuos concept  in  measure theoretical sense.  Namely,  it is well-known that different topologies can generate the same Borel sets.    For example. any  Borel measurable function on a Souslin space is continuous with respect to some stronger topology  that makes  the space Souslin  and generates the same  Borel sets as the original topology (see Exercise 6.10.62 in \cite{bogm} for the proof).  If $\mu$ is a $\mathcal F$-continuous on a Souslin space and $\tilde \mu$ is its Borel-measurable version that is not continuous, then 
the both are continuous  with respect to some stronger topology that makes also $\tilde \mu$ continuous. 
We remark that although the essentially unique solutions are invariant under injective continuous mappings between Souslin  spaces (see Remark \ref{samplespace}),     the strengthening  of the  topology can affect  the partial uniqueness  of the solution  e.g. by  diminishing  the topological support.  
\end{remark}

Due to the properties of the conditional expectation,  the prior distribution  $\mu_X(U)$ is the mixture  $\int  \mu(U,y) d\mu_Y(y)$ of all  posterior distributions so that the prior probability  of $U$ vanishes exactly when $\mu_Y$-almost all posterior probabilities of $U$  vanish.  When  $\mu$ is regular enough, we get the following converse result, which contrasts nicely with the  well-known representation theorem   considered in the  next section. 

\begin{theorem}\label{support2}
Let $F$ and $G$ be Souslin spaces equipped with their Borel $\sigma$-algebras $\mathcal F$ and $\mathcal G$ respectively. Let 
$X$ be an $F$-valued  and $Y$ be a $G$-valued random variable on a complete probability space 
$(\Omega,\Sigma,P)$.
Let $A\in \mathcal G^{\mu_Y}$ be any subset of the topological support $S$ of $\mu_Y$ such  that either  $A\subset \overline{A^\circ}$ or  $\mu_Y(A)=1$.  If $\mu$ is  a   solution of     the statistical inverse problem of estimating probabilities of $X$ given $Y$  that is $\mathcal F$-continuous on $A$, then the posterior distribution  $\mu(\cdot, y)$ at any $y\in A$  is absolutely continuous with respect to the prior distribution.  
\end{theorem}
\begin{proof}
Assume that $\mu_X(U)=0$ for some $U\subset \mathcal F$.  According to the definition of  conditional expectation, 
\begin{equation}\label{cond}
\int \mu(U,y)d\mu_Y(y)= P( X\in U )= \mu_X(U),
\end{equation}
which now vanishes. Since the solution is non-negative, we get that 
$ \mu(U,y)=0$   $\mu_Y$-almost surely on $G$. Since $y\mapsto \mu(U,y)$ is  continuous on $A$,  the set $\widetilde V=\{ y\in A: \mu(U,y)>0\}$ is   a relatively open  set i.e. there exist an open set $ V\subset G$ such that 
 $ V\cap A = \widetilde V$.  Suppose $\widetilde V$ is non-empty.   Similarly as in the proof of Theorem \ref{uniq},  $\mu_Y(\widetilde V)=  \mu_Y(V)>0$ when $A$  has full measure and $\mu_Y(\widetilde V)>0$ when $A\subset \overline {A^\circ}$.  But this contradicts  \eqref{cond} because  $\mu_X(U)=0$.  Thus  the set $\widetilde V$ is empty and $\mu(U,y)=0$ for all $y$ from $A$. 
\end{proof}

When Theorem  \ref{support2} holds for $A=S$,   every   Borel set  $B$ with full prior  probability  has  also full posterior probability  $\mu(B,y)=1$ for all $y\in S$. Our posterior  perception of   the  unknown   appears to be  inline with our prior insight in this aspect.

\begin{remark} \label{divi}
According to a result of Macci \cite{macci}, the absolute continuity of the posterior distributions $\mu(\cdot,y)$  with respect to the prior distribution for  $\mu_Y$-almost every $y$    implies that the conditional distribution of $Y$ given $X=x$ is absolutely continuous  with respect to (=dominated by)
$\mu_{Y}$ for  $\mu_X$-a.e. $x\in F$.  By Theorem \ref{support2}, the posterior probabilities of Borel sets may depend continuously on the observations  $y\in G$ only in  the dominated cases i.e the conditional distribution of $Y$ given $X=x$ has to be  absolutely continuous with respect to some $\sigma$-finite measure for $\mu_X$-a.e. $x\in F$.  The same conclusion holds even if the  space $G$  is replaced with some  subset   $A\in \mathcal G^{\mu_Y}$  having  full $\mu_Y$-measure.   In  the undominated cases, the posterior distribution    has necessarily a  large amount of discontinuities -- the set of all  discontinuity points must have positive  $\mu_Y$-measure. 
 \end{remark}

A partial converse to Theorem \ref{support2} holds in complete separable metric spaces. 

\begin{theorem}\label{lusin}
Let $F$ and $G$ be complete separable metric spaces equipped with their Borel $\sigma$-algebras $\mathcal F$ and 
$\mathcal G$, respectively.  Let $X$ be an $F$-valued and $Y$ be a $G$ random variable on a complete probability space 
$(\Omega,\Sigma,P)$. Let   $\mu$  be a solution of the statistical inverse problem of estimating the distribution of $X$ given $Y$. If the family of  the posterior distributions $\{ \mu(\cdot,y): y\in G\}$ is dominated by a Borel measure $\nu$ on $\mathcal F$, then 
for every $\epsilon>0$ there exists a compact set $K=K(\epsilon,\mu)  \in \mathcal G$ such that 
$\mu_Y(G\backslash K)<\epsilon$ and $\mu$ is $\mathcal F_0$-continuous on $K$ for some family $\mathcal F_0$ of measure-determining
 sets.
 \end{theorem}
\begin{proof}
Equip the space  $M$ of all  probability measures on $(F,\mathcal F)$ with the topology of the weak convergence (i.e. convergence of  integrals of all bounded continuous functions).  It is well-known that this space is a complete separable metric  space whenever   $F$ is a complete separable metric space (see Theorem 8.9.5 in \cite{bogm}).  Equip $M$ with the Borel $\sigma$-algebra  $\mathcal M$ with respect to the weak topology i.e. the cylinder set $\sigma$-algebra  of  the sets of the type 
$$
\{\nu\in \mathcal M: (\nu(f_1),\nu(f_2),...) \in  B\},
$$ 
where $f_i, i\in \mathbf N $ are  continuous bounded functions on $F$ and $B \in \mathcal B(\mathbf R^\infty)$ (with respect to the coordinate-wise convergence).   By Condition 2 of Definition \ref{si},  the solutions $y\mapsto \mu(\cdot,y)$  are  $\mu_Y$-measurable mappings from $G$ to $M$ since $y\mapsto \mu(f,y)$ is  $\mu_Y$-measurable as a pointwise limit of  integrals of simple functions.  By the Lusin  theorem (see Theorem 7.1.13 in \cite{bog}), there exists a family of  compact sets $\mathcal K\subset \mathcal G$  such that given any $\epsilon>0$, the probability   $\mu_Y(K^C)<\epsilon$   for some $K\in \mathcal K$ and  the measure-valued random variable  $y\mapsto \mu(\cdot,y)$ is continuous on $K$ in   the weak topology, of measures implying that $\lim_{i\rightarrow \infty} \mu(f,y_i)= \mu(f,y)$ whenever 
$\lim_{i\rightarrow \infty} y_i=y$ in $K$. Especially, the mappings $y\mapsto \mu(\cdot,y)$  are  $\mathcal F_0$-continuous on $K$, where $\mathcal F_0$  consists of  all Borel sets $U$ whose boundary  satisfies $\mu(\partial U,y)=0$ for all $y\in K$. This follows from the fact that  $\lim_{i\rightarrow \infty} \mu(U,y_i)= 
\mu (U, y)$ whenever $\lim_{i\rightarrow \infty} y_i=y$ by the weak convergence (see Corollary 8.2.10 in \cite{bogm}).
If the family of  posterior distributions $\{\mu(\cdot,y):y \in K\}$ is  dominated by some  Borel  measure, then 
   $\mathcal F_0$ is a measure-determining set (see Lemma 1.9.4 and Proposition 8.2.8 in \cite{bogm}).
      \end{proof}

 \section{The representation of posterior distributions}\label{sec3}

  In this section, we consider a known representation formula  (see   Section 1.2.2 in \cite{flo},  Theorem 1.31 in \cite{schervish}, or  pp. 231--232 in \cite{shir}) for   solution of the statistical inverse problem of estimating the distribution of $X$ given the observations $Y$ that generalizes the finite-dimensional formula 
$$
   D(x|y)= CD(y|x)D_{pr}(x).
$$  
For readers convenience, the proofs  of Lemma  
\ref{yeke}, Lemma \ref{eka2} and Theorem \ref{toka}  are included, 
although they  are  special cases of more general known results.

Throughout the section, we assume that $F$ and $G$  are locally convex Souslin topological  vector spaces equipped with their Borel 
$\sigma$-algebras $\mathcal F$ and $\mathcal G$,  respectively,  and $X$ is taken to be an $F$-valued random variable and $\varepsilon $ is taken to be a $G$-valued random variable statistically independent from  $X$. All the random variables are defined on the same 
complete probability space $(\Omega,\Sigma,P)$.   The mapping $L: F\rightarrow G$ is assumed to be continuous.  We denote $Y=L(X)+\varepsilon $.   

 First, we check that  $Y$ is indeed a random  variable as a combination of Borel measurable mappings.   The product space $F\times G$  is equipped with the usual product $\sigma$-algebra  $\mathcal F\otimes \mathcal G$ generated by rectangles  $U\times V$, where  $U\in \mathcal F$ and $V\in \mathcal G$. 
\begin{lemma} \label{yeke}
The mapping $T:(x,z)\mapsto  L(x)+ z$ is  Borel measurable from $F\times G$ to $G$. 
\end{lemma}
\begin{proof} 
 As the addition is just $\mathcal B(G\times G)$-measurable by continuity,  there is the question whether the Borel 
 $\sigma$-algebra  $\mathcal B(G\times G)$  of the topological product space coincides with the product $\sigma$-algebra   $\mathcal G\otimes \mathcal G$ generated by the rectangles $V\times W$ where $ V,W\in \mathcal G$. 
 
 Certainly, $\mathcal G\otimes \mathcal G\subset \mathcal B(G\times G)$, since the products of open sets   $V,W\subset G$  form  a basis of topology   for $G\times G$.  
 
  Due to the Souslin property, the space $G\times G$  is hereditarily Lindel\"of (\cite{bogm},  Lemma 6.6.4 and Lemma 6.6.5).   Any open set in $G\times G$ can therefore be expressed as a countable union of  sets of the form $V\times W$,  where $V,W \in \mathcal  G$ are open.   Hence $ \mathcal B(G\times G)\subset \mathcal G\otimes \mathcal G$. 
 \end{proof}

 We verify now that 
for any  $\mu_Y$-integrable  $f:G\rightarrow \mathbf R$,  the conditional expectation of 
$f(Y)$ given $X$ is the random variable
$$
\mathbf E[f(Y)| X](\omega) = \int _G f(z) d\mu_{\varepsilon +L(X(\omega))}(z). 
$$
Here the  measure $\mu_{\varepsilon +L(X(\omega))}$ is the image measure of the random variable  $\omega'\mapsto \varepsilon (\omega')+L(X(\omega))$, where $X(\omega)$ is treated as a constant. 
We apply the  following more general claim, for which we failed to find a reference. 
\begin{lemma}\label{eka2}
 Let $Z_1$ be an  $F$-valued   and $Z_2$ be a  $G$-valued random variable  that are statistically independent.  Denote $Z_3=T(Z_1,Z_1)$, where  $T:F\times G\rightarrow G$ is a Borel measurable mapping.    For any 
 $\mu_{Z_3}$-integrable function $f:G\rightarrow \mathbf R$, it holds  that
 $$
 \mathbf E[f(Z_3 )| Z_1](\omega) = \int _G f(z) d\mu_{T(Z_1(\omega),Z_2)}(z) 
 $$
 $P$-almost surely, and $\int_G f(z) d\mu_{T(Z_1 (\omega) ,Z_2)}(z)$ 
 is  a  version of the conditional expectation of $ f(Z_3)$ given $\sigma(Z_1)$.  
\end{lemma}
\begin{proof}
We show  that the claim holds for a  Borel measurable version of  $f$, which exists by Proposition 
2.1.11 in \cite{bogm}. The generalization  for $\mu_{Z_3}$-measurable functions follows then  from Lemma \ref{leu}. 

Remark that   $f\circ T: F\times G \rightarrow \mathbf R$ is then a  Borel measurable function. We will show  that 
 $
 \mathbf E[g (Z_1,Z_2) | Z_2](\omega) = \int _G g(Z_1(\omega), z_2)  d\mu_{Z_2}(z_2) 
 $
holds  for all   Borel measurable  simple functions  $g$  on $F\times G$. The   usual approximation of  Borel  measurable functions  with   simple functions  implies then  for $g=f\circ T$ that 
\begin{eqnarray*}
\mathbf E[ f (T(Z_1,Z_2)) | Z_1](\omega) &=& \int _G f(T(Z_1(\omega), z_2) ) d\mu_{Z_2}(z_2)= 
\mathbf E [ f(T(Z_1(\omega),Z_2))] \\
 &=& \int _G  f(z) d\mu_{T(Z_1(\omega),Z_2)}(z). 
\end{eqnarray*}
Take now  $g=1_C$, where $C\in \mathcal B(F\times G)$.  We need to determine 
 the conditional expectation
$
\mathbf E [ 1_C(Z_1,Z_2)|Z_2]
$
i.e.  the conditional distribution of $(Z_1,Z_2)$ given $\sigma(Z_2)$.  Since $F$ and $G$ are  Souslin spaces,  a regular conditional measure exists  (by Corollary 10.4.6 in \cite{bogm})
and is determined by values on any measure-determining sets. In Souslin spaces,  
the rectangular sets $C=B_1\times B_2$, where $B_1\in \mathcal F$ and $B_2\in \mathcal G$ are measure-determining sets, since $\mathcal B(F\times G)=\mathcal F\otimes \mathcal G$ (see the proof of Lemma \ref{yeke} and  Lemma 1.9.4 in \cite{bogm})).  By the properties of the conditional expectation,
\begin{eqnarray*}
\mathbf E [ 1_{B_1\times B_2}(Z_1,Z_2)|Z_1](\omega)&=& 1_{B_1}(Z_1(\omega)) \int 1_{B_2}(z_2) d\mu_{Z_2}(z_2)\\
&=&
 \int  1_{B_1\times B_2 }(Z_1(\omega),z_2) d\mu_{Z_2}(z_2).
\end{eqnarray*}
\end{proof}

Here is the description of the   solutions $\mu(U,z)$  modulo 
$\mu_Y$-zero measurable sets.   The result is a special case of   Kallianpur-Striebel formula  \cite{kallianpur2}. 

\begin{theorem}\label{toka}
   Let $\mu_{\varepsilon +L(x)}$ be absolutely continuous with respect to a $\sigma$-finite measure $\nu$ for $\mu_X$-a.e.  $x\in F$.  
Set 
\begin{equation*}
\rho(x,z)=\begin{cases} 
 \frac{d\mu_{\varepsilon +L(x)}}{d\nu}(z)  \; {\rm when }\; \mu_{\varepsilon +L(x)}\ll \nu\\
 0 {\rm \; otherwise}.
\end{cases}
\end{equation*}
If   $\rho(x,z) $ is a  non-negative  $\mu_X \times \nu$-measurable  function  on $F\times G$, then there is an  essentially unique   solution $\mu$  of the statistical inverse problem of estimating the distribution of  $X$ given $Y=L(X)+\varepsilon $ 
 such that
\begin{equation} \label{ii}
\mu(U,z)=\frac{\int 1_U(x) \rho(x,z) d\mu_X(x)}{\int \rho(x,z)d\mu_X(x)}  
\end{equation}
for all $z\in G\backslash N_0$, where the set   
   $$ N_0=\{ z\in G: \int \rho(x,z)d\mu_X(x)= 0{\rm \; or \;}\infty\}$$ 
   has   $\mu_Y$-measure zero. 

If $\mu_Y(N)=0$  then $N$ is  also $\mu_{\varepsilon + L(x_0)}$-zero measurable for  $\mu_X$-almost every  $x_0\in F$. If  additionally $\mu_{\varepsilon+L(x)}<<\nu$ 
   for all $x\in F$ and $\rho$ is  positive   $ \mu_X \times \nu$-almost everywhere, then  
$\mu_{\varepsilon + L(x_0)}(N)=0$ for all $x_0\in F$.
 \end{theorem}
\begin{proof}
Let $\mu(U,z)$ be defined by \eqref{ii}.
  If $z\in N_0$, we set $  \mu(U,z)= 1_{U}(x_0)$ for some fixed $x_0\in F$.    We prove that $\mu$ is  a   solution.
      
Let  $U\in \mathcal F$ and $V\in \mathcal G$.  By Theorem \ref{ex} there exists an essentially unique  solution, which we denote here with  $\tilde \mu$. We write two expressions for  $P(X\in U\cap Y\in V)$  using  Lemma \ref{eka2}.  The first is 
\begin{equation}
\begin{split}\label{veka}
\mathbf E[1_U(X)\mathbf E[1_V(Y)|X]]=& \int 1_U(x) \left( \int 1_V(z) d\mu_{\varepsilon +L(x)}(z) \right) d\mu_X(x)\\
=&\int 1_U(x)  \left( \int 1_V(z)  \rho(x,z)d\nu(z) \right) d\mu_X(x)\\
=&\int 1_V(z)  \left( \int 1_U(x)  \rho(x,z) d\mu_X(x) \right) d\nu(z)
\end{split}
\end{equation}
and the second  expression is 
\begin{equation}
\begin{split} \label{seka}
\mathbf E[1_V(Y)\mathbf E[1_U(X)|Y]]=& \mathbf E[1_V(Y)\tilde \mu(U,Y)]\\
=&\mathbf E[\mathbf E[1_V(Y)\tilde \mu(U,Y)|X] ]
\\
=&\int  \int 1_V(z) \rho(x,z) \tilde \mu(U,z) d\nu(z) d\mu_X(x)\\
=&   \int 1_V(z) \left( \int  \rho(x,z) d\mu_X(x) \right) \tilde \mu(U,z)  d\nu(z). 
\end{split}
\end{equation}
 The measurability of $\rho$ is used in  changing the order of integrations by the Fubini theorem.  The  integrability of $\rho$ follows automatically from the finiteness of the left-hand side of \eqref{seka} for $U=F$ and $V=G$.  

Since the equivalence of \eqref{veka} and \eqref{seka}  holds for all $V\in \mathcal G$, we obtain $$
\tilde \mu(U,z){\int \rho(x,z)d\mu_X(x)} = \int 1_U(x) \rho(x,z) d\mu_X(x)$$   for  $\nu$-almost every $z$. Hence, $\tilde \mu(U,z)=\mu(U,z)$ for  $\nu$-almost every $y$ such that 
$0<\int \rho(x,z)d\mu_X(x)<\infty$. 
  
          The denominator in (\ref{ii}) may vanish only on a set $A$  of $\mu_Y$-measure zero since  the choice $U=F$, $V=A$   gives $\mu_Y(A)=0$ in \eqref{veka}.       The same consideration implies that also the measure $\mu_Y$ is absolutely continuous with respect to $\nu$.  Similarly, the denominator is finite $\nu$-almost surely, which implies $\mu_Y$-almost surely.  
          We conclude that  $N_0$ has $\mu_Y$-measure zero and $\tilde \mu(U,z) =\mu(U,z) $ $\mu_Y$-almost surely.  
          Then $\mu(U,y)$  satisfies Condition 2 of Definition \ref{si}.  By Lemma \ref{leu},    $\mu$ satisfies Condition 1. 
           By the integrability of $\rho$, $\mu$ satisfies Condition 3 of Definition \ref{si}.

               We proceed to the last claim.     Taking $V=N$  and $U\in \mathcal F$ in   \eqref{veka} implies that  $\mu_{\varepsilon +L(x)}(N)=\int 1_N(z) \rho(x,z)d\nu(z) $ vanishes  for $\mu_X$-almost all $x\in F$. When $\rho$ is a.e.  positive, also  $\nu(N)$ has to  vanish.  We obtain $\mu_{L(x_0)+\varepsilon }(N)=0$ for all $x_0\in F$  by using the absolute continuity.
     
    \end{proof}

The  last statement of the above theorem is added to show how small the zero  measurable set for a given unknown is. The representation formula does not  improve the  essential uniqueness of   solutions, because   the Radon-Nikodym density $z\mapsto d{\mu_{L(x)+\varepsilon }}/d\nu(z)$ is  only determined up to $\nu$-equivalence.  
It should be noted that under the  domination assumptions  on $\mu_{\varepsilon+L(x)}$ in Theorem \ref{toka}, there always exists versions of   the  Radon-Nikodym densities that are jointly measurable. In  \cite{kallianpur2}, this claim is proved assuming  that  $Y^{-1}(\mathcal G)$ is countably generated. 
In Souslin spaces, the Borel $\sigma$-algebras are countably 
generated  by Corollary 6.7.5 in \cite{bogm}.

It is easy to see, that the prior distribution $\mu_X$ and the posterior distribution $\mu(\cdot,z)$ are equivalent if  \eqref{ii} holds   and $\rho(\cdot, z)>0$  $ \mu_X$-almost everywhere.  

\begin{remark}
The existence of $\nu$ is a delicate matter. For example, the measure $\mu_{\varepsilon+L(x)}$ may not be almost surely absolutely continuous with respect to $\mu_Y$, although
 \begin{eqnarray*}
\mu_Y(U)= \mathbf E [1_U(Y)]= \mathbf E [ \mathbf E [1_U(Y)|X]] = \int  \mu_{\varepsilon +L(x)}(U) d\mu_X(x)
 \end{eqnarray*}
 by Lemma \ref{eka2}.  We can only  conclude that $\mu_{\varepsilon +L(x)}(U)$ vanishes $\mu_X$-a.s. whenever $\mu_Y(U)$ vanishes and the  $\mu_X$-zero measurable set may depend on $U$.  A Gaussian example in  Remark \ref{gah} of Section \ref{sec5} shows that this is indeed the case.      In general,   the Halmos-Savage theorem (Lemma 7 in    \cite{halmos}),  states  that   from a dominated family   of finite measures, which in our case is   $\{ \mu_{\varepsilon+L(x)}: x\in M\}$ where $\mu_X(M)=1$,  one can pick out countably many measures $\mu_{\varepsilon+L(x_i)}$  in such a way that   the measure  $\nu:=\sum_{i} a_i \mu_{\varepsilon+L(x_i)}$, where $\sum_{i}a_i=1$ and all $a_i>0$, is not only  a dominating measure but also equivalent to the family  $\{\mu_{\varepsilon+L(x)}: x \in M\}$ (i.e. the measures in the  family vanish on the same subsets as $\nu$).     Especially, this gives   a necessary and sufficient condition for the domination of the   probability measures $\mu_{\varepsilon+L(x)}$.  In Section \ref{sec5} we concentrate on  special cases where $\mu_\varepsilon$  can be taken as $\nu$.   In these examples, we  require that  $\mu_{\varepsilon +L(x)} (U)=0$ whenever $\mu_{\varepsilon }(U)=0$ i.e. $\mu_\varepsilon $ is quasi-invariant with respect to translations with $L(x)$,  where $x\in F$. This allows the use of any prior distribution on $F$. 
However, Remark \ref{siirto} in Section \ref{sec5} demonstrates that in dominated cases it is not always possible to choose   $ \mu_\varepsilon$ as $\nu$.    
     \end{remark}

We return to the question  of partial uniqueness  (Theorem  \ref{uniq}). 
  The conditions in the next theorems allow  easier validation of the measurability and guarantee  some continuity for the solutions. However,   under  the stronger assumption that the function  $(x,z)\mapsto \rho(x,z)$ is jointly continuous and bounded,   the solution $\mu$ is always $\mathcal F$-continuous on $G$ (see Theorem 7.14.8  in \cite{bogm}).        Recall,  that the class of all Souslin sets is quite large since all Borel subsets of a Souslin space are Souslin sets  by Corollary 6.6.7 in \cite{bogm}.

\begin{theorem} \label{uni}
Let $\mu_{\varepsilon+L(x)}$  be  absolutely continuous with respect  to  a  $\sigma$-finite measure $\nu$   for $\mu_X$-almost every  $x\in F$.      If  
  \begin{equation*}
\rho(x,z)=\begin{cases} 
 \frac{d\mu_{\varepsilon +L(x)}}{d\nu}(z)  \; {\rm when }\; \mu_{\varepsilon +L(x)}\ll \nu\\
 0 {\rm \; otherwise}.
\end{cases}
\end{equation*}
   is a separately continuous  function  on some $F_0\times A$, where $F_0$ is a Souslin subset of $F$ with full $\mu_X$-measure and  $A$ is a Souslin subset of  $G$ such that $\nu(A^C)=0$, then $\rho$ is   $\mu_X\times \nu$-measurable.  

  If additionally   $\sup_{z\in K} \rho(x,z)\in L^1(\mu_X)$ for all compact sets $K\subset G$  then 
  \begin{equation} \label{rocky1} 
z\mapsto \mu(U,z)=\frac{\int 1_U(x) \rho(x,z) d\mu_X(x)}{\int \rho(x,z)d\mu_X(x)}  
\end{equation} 
is  $\mathcal F$-continuous  on  $ K\cap \{ z\in A: 0<\int \rho(x,z)d\mu_X(x)<\infty\}$ for every compact set $K\subset G$.
\end{theorem}
\begin{proof}
 Assume that $\rho$ is separately continuous. Since $F_0$ is a Souslin space, there exists a continuous surjection $R$ from some complete separable metric space $M$ onto $F_0$.  We consider first  the function  
 $(m,z)\rightarrow \rho(R(m),z)$ on $M\times A$. This function is a pointwise limit of continuous functions due to  a theorem of W. Rudin \cite{rud}. Hence the function $(m,z)\mapsto \rho(R(m),z)$   is $\mathcal B(M\times A)=\mathcal B(M)\otimes \mathcal B(A) $ -measurable. We compose it with a $\mu_X\times\nu$-measurable mapping  $(R^{-1},I)$ where the inverse comes from the measurable choice theorem    (see Theorem 6.9.1 in \cite{bogm}, note that Souslin sets are universally measurable i.e. measurable with respect to any finite Radon measure by Theorem  7.4.1 in \cite{bogm}). Then we see that $1_{A}(z)1_{F_0}(x)\rho(x,z)$, together with its equivalent mapping $\rho(x,z)$,  is $\mu_X\times \nu$-measurable. 
  
   By  the  Lebesgue dominated convergence theorem, we obtain the sequential continuity of the marginals. On Souslin spaces, the compact sets are metrizable (see Corollary 6.7.8 in \cite{bogm}). In metrizable spaces sequential continuity coincides with 
continuity.
  \end{proof}

The above Theorem \ref{uni} shows  ${\mathcal F}$-continuity  of the solution $\mu$ on $\{z \in  A :  0<\int \rho(x,z)d\mu_X(x)<\infty\}$  when   $G$ is e.g. a $k$-space i.e.   a subset $C$ of $G$ is closed if and only if $C\cap K$ is closed for every compact $K\subset G$ (see Definition 43.8 in   \cite{willard}).   Indeed, it is well-known that a function $f$ is continuous on a $k$-space if and only if it is continuous on every compact subset. In particularly,   this holds for all first-countable spaces, like metric spaces.   Note,  that the  space of tempered  distributions  $\mathcal S'(\mathbf R^n)$ is a $k$-space when equipped with its  strong topology but not with its weak topology, while the distribution space  $\mathcal D'(U)$ is not a $k$-space with respect to either topology \cite{harada}. But $\mathcal D'(U)$ is a   Lusin space \cite{schwarz} -- i.e. a Hausdorff space that is a continuous injective image of  a  complete metric space -- and can be  equipped with a  stronger   metrizable topology inherited from the metric space.   However, this topology depends on the chosen metric space and has all the   drawbacks indicated in Remark \ref{samplespa}.

We combine Theorem \ref{uniq} and Theorem \ref{uni} in  a simple case.
 \begin{corollary}\label{uni2}
Let $G$  be a $k$-space. Let $\mu_{\varepsilon+L(x)}$  be  equivalent with  a  probability  measure $\nu$   for every  $x\in F$.  Denote $S_\nu$ the topological support of $\nu$.     If  
  \begin{equation*}
\rho(x,z)= \frac{d\mu_{\varepsilon +L(x)}}{d\nu}(z)
\end{equation*}
   is a separately continuous  function  on $F\times S_{\nu}$  and if    $\sup_{z\in K} \rho(x,z)\in L^1(\mu_X)$ for all compact subsets $K\subset G$  then all   solutions of the statistical inverse problem of estimating the distribution of $X$ given $Y$  
   that are $\mathcal F$-continuous on $ \{ z\in  S_{\nu}:  0< \int \rho(x,z)d\mu_X(x)<\infty\} $ coincide with 
     \begin{equation} \label{rocky} 
z\mapsto \mu(U,z)=\frac{\int 1_U(x) \rho(x,z) d\mu_X(x)}{\int \rho(x,z)d\mu_X(x)}  
\end{equation} 
 on   $\mathcal F \times \{ z\in  S_{\nu}:  0< \int \rho(x,z)d\mu_X(x)<\infty\}$.
\end{corollary}
 The proof is an  immediate consequence of  the following  lemma, where  we characterize the topological support of $\mu_Y$ in more convenient terms.
\begin{lemma}\label{konti}
Let $Y=L(X)+\varepsilon$, where $X$ and $\varepsilon $ are statistically independent. 
 The topological support of $\mu_Y$  is the smallest closed set $S\subset G$ such that $\mu_{\varepsilon+L(x)}(S)=1$ for $\mu_X$-almost every   $x\in F$.
Moreover, if $\mu_{\varepsilon+L(x)}$ is equivalent with a probability measure $\nu$ for  every $x\in F$, then 
the topological supports of $\mu_Y$ and $\nu$ coincide. 
\end{lemma}
\begin{proof}
The first claim follows from the convolution $\mu_Y(S)= \int_G \mu_{\varepsilon+L(x)}(S) d\mu_X(x)$. 
For the second claim,  we note that $\mu_{\varepsilon+L(x)}(S)= \nu(S)$ for every closed set $S\subset G$
 with full $\nu$-measure  (or full $\mu_{\varepsilon+L(x)}$-measure). 
\end{proof}

\section{Converging approximations}\label{sec4}

  Throughout this section, we use the following assumptions
\begin{definition} We say that  Assumption A holds,  if the following four conditions are satisfied.
\begin{enumerate}
\item{  Topological spaces  $F$ and $G$  are   locally convex   Souslin topological vector    spaces equipped with their Borel $\sigma$-algebras $\mathcal F$ and $\mathcal G$, respectively.  }
\item {The triple $(\Omega,\Sigma,P)$ is a complete probability space,  $X$ and $X_n$  are 
  $F$-valued random variables on $\Omega$  and   $\varepsilon$ is  $G$-valued random variable on $\Omega$.  The random variables  $X$ and $\varepsilon$ are  independent.   The random variables  $X_n$ and $\varepsilon$ are  independent.} 
\item The mapping $L:F\rightarrow G$ is continuous,  and we denote   $Y=L(X)+\varepsilon$ and  $Y_n=L(X_n)+\varepsilon$.
\item{ The measure  $\mu_{\varepsilon +L(x)}$ is  absolutely continuous with respect to some  $\sigma$-finite measure $\nu$ on $(G,\mathcal G)$ for any $x\in F$, and   its density 
 $$
 \rho(x,y):= \frac{d\mu_{\varepsilon +L(x)}}{d\nu}(y)
 $$ is a $\mu_Z \times \nu$-measurable  function  on $F\times G$ for  random variables $Z=X$ and $Z=X_n$, $n\in \mathbf N$.} 
 \end{enumerate}
 \end{definition}

 When Assumptions A holds, we can use  Theorem \ref{toka}  to  represent  the approximated posterior distribution of $X_n$ given 
 $y=Y_n(\omega_0)$ as
\begin{equation}\label{muuan}
\mu_n(U,y):=\frac{\int 1_U(x) \rho(x,y)d\mu_{X_n}(x)}{\int \rho (x,y) d\mu_{X_n}(x)}
\end{equation}
and the posterior distribution of $X$ given $y=Y(\omega_0)$ as  
\begin{equation}\label{muu}
\mu(U,y):=\frac{\int 1_U(x)\rho (x,y) d\mu_{X}(x)}{\int \rho(x,y) d\mu_{X}(x)}
\end{equation}
for all $U\in \mathcal F$ and  $y\in M_0$, where 
\begin{eqnarray}
M_0 =   \{ y\in G:   0 < \int  \rho (x,y) d\mu_{Z }(x)<\infty ,  \text{for } Z=X, X_n  \text{ where }\;  n\in \mathbf N   \} \label{mnolla}
\end{eqnarray}
 has full $\mu_Y$-measure.
 
We recall  some definitions on the convergence of measures.
\begin{definition}
Let  $m$ and $m_n$, where $n\in \mathbf N$,   be   $\sigma$-finite measures on a topological space $F$ equipped with the Borel $\sigma$-algebra $\mathcal F$.

 $(i)$  The measures $m_n$ {\it converge weakly} to  $m$ if 
$$
\lim_{n\rightarrow \infty }\int f(x) dm_n(x)= \int f(x) dm(x)
$$
for all bounded continuous functions $f$ on $F$.

 $(ii)$ The measures $m_n$  {\it converge setwise} to $m$ if 
$$\lim_{n\rightarrow \infty }m_n(U)=m(U)$$ for every $U\in \mathcal F$.

 $(iii)$    The measures $m_n$    {\it converge in variation} to $m$ if 
\begin{equation*}
\lim_{n\rightarrow \infty }\sup_{U\in \mathcal F} |m_n(U)-m(U)|=0.\end{equation*}
 \end{definition}

 It is well-known that the weak convergence of the probability measures implies  the  convergence of certain  expectations on  regular enough spaces. The  following theorem generalizes slightly Lemma 8.4.3 in \cite{bogm} by requiring that   the  discontinuities of  $f$  belong to some  $m$-zero measurable set.   The proof for the present case seems no to be readily available in the literature.

\begin{lemma}\label{lele}
Let $F$ be a locally convex    Souslin topological vector     space and  $m,m_n$, where $n\in \mathbf N$,  be finite measures on $(F,\mathcal F)$. Let $f$ be an $m$-integrable Borel function on $F$ whose discontinuities are contained in an $m$-zero measurable set.
  If 
  \begin{equation*}%\label{unif}
  \lim_{C\rightarrow \infty} \sup_n \int _{|f|>C} |f|(x)dm_{n}(x)=0,
  \end{equation*}
 then
 $m_{n}(f)$ converge to $m(f)$ whenever 
 $m_{n}$ converge weakly  to $m$  
 \end{lemma}
\begin{proof}
   Consider first  a bounded  Borel measurable function $g$, say  $|g| \leq c$,  whose points of discontinuity belong to an  $m$-zero measurable set $N_g$.  The integral of $g$ can be written as 
   $$
  \int g d m_n=  \int_{-c} ^{c} t  d(m_n\circ g^{-1})(t)  $$
  where the integrand is bounded and continuous on  $[-c,c]$.  We show that 
 the measures  $m_n\circ g^{-1}$ on $[-c,c]$ converge weakly to $m\circ g^{-1}$, which immediately implies the  convergence of  $m_n(g)$  to $m(g)$ as $n$ grows.   We apply a well-known property of completely regular spaces  (see Corollary 8.2.4 in   \cite{bogm}), according to which the weak convergence  of  $m_n \circ g^{-1}$ to the Radon measure  $m\circ g^{-1}$ is equivalent to 
    $$ 
   \limsup_{n\rightarrow \infty}  m_n\circ g^{-1}(A)\leq m\circ g^{-1}(A)
   $$ 
   for all  closed sets $A$. Note that all  locally convex spaces are completely regular.   If $A\subset [-c,c]$ is closed, the
   closure   $\overline {g^{-1}(A)}\subset g^{-1}(A)\cup N_g$ because 
   $g$ is continuous in the relative topology of $G\backslash N_g$.
       Since $N_g$ has zero $m$-measure, 
   $$
 m(g^{-1}(A)) =m(\overline {g^{-1}(A)})\geq 
 \limsup_{n\rightarrow \infty }m_n(\overline {g^{-1}(A)})\geq
  \limsup_{n\rightarrow \infty} m_n (g^{-1}(A))  
 $$ 
 by  the weak convergence of measures $m_n$.

 Let $\wedge$ denote the binary operation of taking the minimum of two real numbers. For the general case,   we approximate $f$ with bounded functions ${ \sgn}(f) (|f|\wedge C)$  in the difference
\begin{eqnarray}
|(m_{n}-m)(f)|&=& |(m_{n}-m)(f- {\sgn}(f) (|f|\wedge C) + {\sgn}(f) (|f|\wedge C))| \label{kaava1} \\
&\leq &   \sup_n \int_{|f|>C}    |f| (x)d(m_n+m)  + \left|\int {\sgn}(f)(|f|\wedge C) d(m_{n}-m)\right|.\nonumber
\end{eqnarray}
By the assumption, the first term in the sum \eqref{kaava1} gets arbitrarily small  when $C$  is chosen large enough.  
Since $ {\sgn}(f)(|f|\wedge C)=:g$  is bounded and $m$-a.e. continuous, the second  term in the sum 
\eqref{kaava1} converge to zero for fixed $C$ when $n$ grows by  the weak convergence of the measures $m_n$.  
\end{proof}

Lemma \ref{lele} can be applied  for $f=g \rho(\cdot,y)$, $m_n=\mu_{X_n}$ and 
$m=\mu_X$, where $g$ is any  continuous  bounded function on $F$.  

\begin{theorem}\label{coco2}
Let Assumption A hold and let $\mu_n$,  $\mu$ and $M_0$ be defined by equations \eqref{muuan}, \eqref{muu} and \eqref{mnolla}, respectively. Let $y\in M_0$ and let the  discontinuities of $x\mapsto \rho(x,y)$ belong to a $\mu_X$-zero  measurable set. If  the functions  $x\mapsto \rho(x,y)$  satisfy   \begin{equation*}
  \lim_{C\rightarrow \infty} \sup_n \int _{|\rho(\cdot,y)|>C} |\rho(x,y)|d\mu_{X_n}(x)=0.
  \end{equation*}
 then  the approximated  posterior distributions  $\mu_n(\cdot, y)$ converge weakly to  the posterior distribution $\mu(\cdot ,y)$  whenever  the approximated prior distributions $\mu_{X_n}$ converge weakly to the prior distribution $\mu_X$.   
\end{theorem}

The conditional  mean is a common estimate for the unknown. 
The convergence  of  conditional mean estimates in  the weak topology of       $F$ is considered next.
\begin{theorem}\label{kass0}
Let Assumption A hold,   let  $y\in M_0$ and  let the discontinuities of $x\mapsto \rho(x,y)$ belong to  a $\mu_X$-zero  measurable set.  If functions  $ x\rightarrow \langle x,\alpha\rangle^k \rho(x,y)$ belong to $L^1(\mu_{X})$ and  satisfy 
  \begin{equation*}
  \lim_{C\rightarrow \infty} \sup_n \int _{|\langle \cdot, \alpha\rangle|^k \rho(\cdot,y)>C} |\langle x,\alpha\rangle|^k\rho(x,y) d\mu_{X_n}(x)=0,
  \end{equation*}
  for $k=0,1$ and  $\alpha \in F'$,  then  the  approximated weak conditional mean estimates  $\mu_n(\langle \cdot,\alpha\rangle, y)$ converge to  the weak  conditional mean estimate $\mu(\langle \cdot ,\alpha\rangle ,y)$  whenever the approximated   prior distributions   $\mu_{X_n}$ converge weakly to the prior distribution $\mu_X$. 
\end{theorem}
\begin{proof}
 The nominators and the denominators of 
$$
\mu_n(\langle \cdot,\alpha \rangle, y)= \frac{\int \langle x, \alpha \rangle \rho(x,y) d\mu_{X_n}(x)}{\int \rho(x,y) d\mu_{X_n}(x) }
$$
converge  as $n$ grows   by Lemma \ref{lele}, and the limit of their quotients is $\mu(\langle \cdot,\alpha\rangle, y)$.
\end{proof}

When $F$ is a separable Banach space, we can state conditions  for the norm convergence of the conditional mean 
estimates that  are defined as Bochner integrals, that is,  
$\int x dm(x)$  is taken to be   the limit in $F$  of  integrals of simple functions of the form  $  x\mapsto \sum_{i=1}^n  x_{k_i} 1_{U_{k_i}}(x)$, where $x_{k_i}\in F$ and $U_{k_i}\in \mathcal F$. 

\begin{theorem}\label{kass}
Let Assumption A hold and let $\mu_n$,  $\mu$ and $M_0$ be defined by equations \eqref{muuan}, \eqref{muu} and \eqref{mnolla}, respectively. Let  $y\in M_0$ and   let the discontinuities of $x\mapsto \rho(x,y)$ belong to  a $\mu_X$-zero  measurable set.   Additionally, let $F$ be  a separable Banach space with norm 
$\Vert \cdot \Vert$.

If  $\Vert  \cdot \Vert \rho(\cdot,y)\in L^1(\mu_X)$ 
satisfies
\begin{eqnarray} \label{ippe}
  \lim_{C\rightarrow \infty} \sup_n \int _{ \{ x :  \Vert x \Vert^k \rho(x,y) > C \} }   \Vert x\Vert ^k\rho(x,y) d\mu_{X_ n}(x)=0
  \end{eqnarray}  
  for $k=0,1$  
       then the conditional mean estimates  $\int x \, \mu_n(dx,y)$ converge in the norm of $F$ to the conditional mean estimate $ \int  x \, \mu(dx ,y)$  whenever  the approximated   prior distributions  $\mu_{X_n}$ converge weakly to the prior distribution $\mu_X$. 
\end{theorem}
\begin{proof}
The assumptions for $ k=0$ guarantee that the denominators 
$\int \rho(x,y)d \mu_{X_n}(x)$ of the posterior distributions  converge to $\int \rho(x,y)d\mu_X(x)$ as $n\rightarrow \infty$ 
by Lemma \ref{lele}.

The function $\Vert \cdot \Vert \rho(\cdot,y) $ has a  finite expectation 
with respect to all measures  $\mu_{X_n}$, $n\in\mathbf N$ and $\mu_X$. Therefore,  the  mapping  $x\mapsto x\rho(x,y)$ is Bochner integrable with respect to all $\mu_{X_n}$ and $\mu_{X}$, and its discontinuities belong to a $\mu_X$-zero measurable set. 
 
 For the moment, let us choose random variables $X_n$ and $X$ on another probability space $(\widetilde \Omega, \widetilde \Sigma, \widetilde P)$  in such a way  that their image measures  are  $\mu_{X_n}$ and $\mu_X$,  respectively, and  the random variables   $X_n$ converge almost surely to $X$ as $n\rightarrow \infty$.  Such a choice is possible by the Skorokhod  representation theorem (see Theorem 8.5.4 in \cite{bogm}).  Especially, $X_n\rho(X_n,y)-X\rho(X,y) $ is Bochner integrable with respect to the probability measure  $\widetilde P$ by the triangle inequality. 
Denote  with  $A_C=\{\Vert X_n\rho(X_n,y)-X \rho(X,y)\Vert >C\}$ for $C>0$. 
Then  the nominators of the posterior distributions satisfy
\begin{eqnarray*}
\left \Vert \int   x \rho(x,y) d\mu_{X_n}(x)-\int x\rho(x,y)d\mu_{X}(x)\right\Vert &=& 
\left \Vert \int X_n \rho(X_n,y)- X\rho(X,y) d\widetilde P \right \Vert \\
&\leq  & \int_{A_C}  \Vert X_n\rho(X_n,y)-X \rho(X,y)\Vert d\widetilde P +\\ 
&&\int C\wedge \Vert X_n\rho(X_n,y)-X \rho(X,y)\Vert d\widetilde P\\
&=:&I_1(n;C)+I_2(n;C).
\end{eqnarray*}
 The integrals $I_1(n;C)$ vanish when $C\rightarrow 
 \infty$ since their arguments are uniformly integrable. Indeed,
both $ \Vert X_n\rho(X_n,y) \Vert$ and $\Vert X \rho(X,y)\Vert$ are uniformly integrable,  as is also their 
sum. Any sequence of non-negative functions  that has   uniformly integrable upper bound  with respect to a finite measure   is again uniformly integrable  with respect to the finite measure. These facts are direct consequences  from the characterization of the uniformly integrable function through  uniformly absolutely continuous integrals (see Proposition 4.5.3 in \cite{bogm}). 

The integrals $I_2(n;C)$ for a fixed $C$  converge to zero as $n\rightarrow \infty $ by  the  Lebesgue dominated convergence theorem and continuity properties  of $\rho$.
\end{proof}

\begin{remark}\label{tapsa}
Theorem \ref{kass}  generalizes the  similar convergence results of  \cite{helin,siltanen2},   in which  the spaces  $F$ and $G$  are separable Banach  spaces, $\varepsilon$ is Gaussian   for non-Gaussian $\varepsilon$ on more general spaces.  In \cite{helin,siltanen2}, it is assumed that   $$ \sup_n \mathbf E [\exp(a \Vert X_n \Vert )]< \infty$$ for all $a>0$. This  attractive condition  is stronger than \eqref{ippe} for the given $\rho$, which has the form $\rho(x,y)=\exp(\langle y, Lx\rangle_G - \frac{1}{2}\Vert Lx\Vert ^2_H) $,  where  $H$ is a certain Hilbert space.    Indeed,  by  the 
de la Vall\'ee Poussin theorem (e.g. Theorem  4.5.9 in \cite{bogm}) the condition \eqref{ippe} 
is equivalent to the  existence of a nonnegative increasing  functions 
$g_k$ on $\mathbf R $ such that $\lim_{t\rightarrow \infty}  t^{-1}g_k(t)=+\infty$ and    $\sup_n  \int g_k(|x|^k \rho(x,y)) d\mu_{X_n}(x)<\infty$.
Moreover,
   $$ 
   \Vert x\Vert ^k \rho(x,y) =\Vert x\Vert^k \exp(\langle y, Lx\rangle_G - \frac{1}{2}\Vert Lx\Vert ^2_H)\leq  \exp (a \Vert x \Vert ), $$
  where $a=1+\Vert y\Vert_G  \Vert L\Vert_{F\rightarrow G}$.  The choice  $g(t)=t^2$  guarantees that the   condition \eqref{ippe} holds when $\sup_n \mathbf E [\exp(a \Vert X_n \Vert )]< \infty$.    
\end{remark}

Next, we pursue after a stronger convergence of the posteriors. 

\begin{theorem}\label{coco}
Let Assumption A hold and let $\mu_n$,  $\mu$ and $M_0$ be defined by equations \eqref{muuan}, \eqref{muu} and \eqref{mnolla}, respectively.  Let $y\in M_0$.
 
  If  the measures $U\mapsto \int_ U \rho(x,y) d\mu_n(x)$ on $(F, \mathcal F)$   are  uniformly bounded, and  equicontinuous at zero in the sense that for  every decreasing sequence $\{U_i\}\subset \mathcal F$ with empty intersection,   $$\lim_{i\rightarrow \infty } \sup_n \int_{U_i} \rho(x,y)d\mu_n(x)=0,$$     then  the  approximated posterior distributions   $\mu_n(\cdot, y)$ converge  setvice (or in variation) to  the posterior distribution $\mu(\cdot,y)$  whenever   the approximated   prior distributions  $\mu_{X_n}$ converge setvice (or in variation) to the prior distribution $\mu_X$.
\end{theorem}
\begin{proof}
Assume first, that the approximated prior distributions converge setwise.
Define a finite  measure $\nu:=\mu_X+\sum_{n=1}^\infty 2^{-n}\mu_{X_n}$ on $(F,\mathcal F)$. Each  $\mu_{X_n}$ is absolutely continuous with respect to $\nu $ and has  Radon-Nikodym density $f_n:= \frac{d\mu_{X_n}}{d\nu}$. 

The measurable function $x\mapsto \rho(x,y)$ is an increasing limit of some simple functions $\phi_y^{(i)}(x)$ and, by   Egorov's theorem, the convergence is almost uniform with respect to  the measure $\nu$.  That is, for every $\varepsilon >0$, there exist a set $A_\varepsilon $ such that $\phi_y^{(i)}$ converge uniformly to $\rho(\cdot,y)$ on $A_\varepsilon $  and $\nu(A_\varepsilon ^C  )<\varepsilon $. One may choose a sequence $\varepsilon _j\rightarrow 0$ and get increasing sets $A_{\varepsilon _j}$ such that $\nu(\cap_j A_{j}^C  )=0$ and the simple functions $\phi_y^{(i)}$ converge uniformly on each $A_{\epsilon_j}$. But then 
\begin{equation*}
\begin{split}
\left \vert \int_U \rho(x,y)d\mu_X(x)-\int_U \rho(x,y)d\mu_{X_n}(x) \right\vert
\leq \left \vert\int_U \phi_y^{(i)}(x)d \mu_X(x) -\int _U  \phi_y^{(i)}(x)d\mu_{X_n}\right \vert
\\  + \left \vert\int_U  \rho(x,y)-\phi_y^{(i)}(x)d \mu_X(x) -\int _U  \rho(x,y)- \phi_y(x)^{(i)}d\mu_{X_n}\right \vert,
\end{split}
\end{equation*}
where the last term is bounded by
\begin{equation*}
\begin{split}
  \int_{U\cap A_{\varepsilon _j}} |\rho(x,y)-\phi_y^{(i)}(x)| (f_n+f)(x)d\nu(x) 
+&\int_{A_{\varepsilon _j}^C  }   |\rho(x,y)-\phi_y^{(i)}(x)| (f_n+f)(x)d\nu(x)\\
 =:&
 I_1+I_2 
 \end{split}
\end{equation*}
In the integral $I_2$,  the estimate $|\rho(x,y)-\phi_y^{(i)}(x)|\leq \rho(x,y)$ gives  \begin{eqnarray}\label{hau}
I_2\leq  \sup_n \int_{A_{\varepsilon _j}^C  } \rho(x,y) d(\mu_{X_n}+\mu_{X})(x).
\end{eqnarray}
If the intersection of  the sets $A_{\varepsilon _j}^C  $   is not empty, we subtract the $\nu$-zero measurable intersection from each $A_{\varepsilon _j}^C  $. Then the equicontinuity at zero of  measures $U\mapsto \int_U\rho(x,y)d\mu_{X_n}(x)$ implies that the integrals \eqref{hau} are bounded by any given  positive number when $j$ is large enough. 
The final thing is to choose the  simple function $\phi_y^{(i)}$ so that 
$|\rho(\cdot,y)-\phi_y^{(i)}|$ is small enough on chosen $A_{\varepsilon _j}$  and  then  choose large enough $n$ so that $|\mu_{X}(1_U \phi_y^{(i)})-\mu_{X_n}(1_U \phi_y^{(i)})|$  gets small enough.
This is possible since the integrand is a bounded simple function and 
 $\mu_{X_n}$ converge setwise to $\mu_X$.

In order to prove convergence in variation, just add $\sup_{U\in \mathcal F}$ in front of the above estimates.  
\end{proof}

 Equivalent conditions for the equicontinuity at zero  of   a bounded family of  measures $m_n$ on $(F,\mathcal F)$ are presented in Lemma 4.6.5  in \cite{bogm}.  The  setwise convergence  of measures $\mu_{X_n}$  actually implies that they are  equicontinuous at zero by Theorem 4.6.3 in \cite{bogm}.

\section{Examples of noise}\label{sec5}

Below, some  cases are presented, where the Radon-Nikodym derivatives 
$$
\frac{d\mu_{\varepsilon+L(x)}}{d \nu}
$$
exist with respect to some $\sigma$-finite measure 
$\nu$. Two first cases,  where  the noise term is  finite-dimensional  or Gaussian, are well-known.
For these cases, we apply the results of previous sections. The  next four cases demonstrate that the 
approach taken in this paper applies also for more general noise models. 

\subsection{Finite-dimensional noise with a probability density}

 This example extends  the  convergence results   in  \cite{fitz} to locally convex Souslin space-valued unknowns. 
Let $G$ be the Euclidian space $\mathbf R^k $,  let $F$ be a locally convex Souslin space, and let  $L:F\rightarrow G$ be a continuous mapping.  Consider the statistical inverse problem of estimating the distribution of  an $F$-valued random variable $X$ given a sample $y$ of a $G$-valued random variable $Y=L(X)+\varepsilon$, where   the  $G$-valued random variable $\varepsilon$ is statistically independent from $X$.    In order to use the representation formula of Theorem \ref{toka} for  the essentially unique posterior distribution of $X$ given a sample $y_0$ of  $Y$, we need  the required  $\sigma$-finite measure $\nu$.  A natural choice is to take the Lebesgue    measure as $\nu$, when possible.

 Assume that  the noise $\varepsilon $ is  $\mathbf R^k$-valued random vector whose image measure $\mu_\varepsilon $ is absolutely continuos with respect to the Lebesgue    measure , say $\mu_\varepsilon (dx)= D_\varepsilon (x)dx$,  with  the property  that  $D_\varepsilon>0$ almost everywhere. Especially, $\mu_\varepsilon$ is then equivalent to the Lebesgue   measure.   
 
 In Theorem \ref{toka},   the Radon-Nikodym derivative of $\mu_{\varepsilon +L(x)} $ with respect to the Lebesgue   measure    , i.e.   $(x,y)\mapsto D_\varepsilon (y-L(x))$, is  required to be jointly measurable.   Since $D_\varepsilon $ is measurable, and the addition is measurable, the continuity  of $L$ suffices here.  We obtain  an essentially unique  solution $\mu$ of the statistical inverse problem of estimating the distribution of $X$ given a sample $y_0$ of $Y$ that satisfies
\begin{equation*} 
\mu(U,y_0)=\frac{\int_U D_\varepsilon (y_0-L(x)) d\mu_X(x)  }{\int  D_\varepsilon  (y_0-L(x))d\mu_X(x)}  
\end{equation*}
for all $U\in \mathcal F$ and  all  $y_0$ such that $0<\int  D_\varepsilon  (y_0-L(x))d\mu_X(x)<\infty$.
 Here $D_\varepsilon (y_0-L(x))$ is often called the likelihood function.  If  $D_\varepsilon $ is continuous and bounded,  we may drop out the word "essentially", as the   solution is the unique continuous solution  by Corollary \ref{uni2} (the topological support of 
 $\mu_Y$ is the whole space by  Lemma \ref{konti} since $\mu_{\varepsilon+L(x)}$ is   equivalent with the Lebesgue   measure).

When  $X$ is an $\mathbf R^m$-valued random variable with a density $D_{pr}(x)$ with respect to  the Lebesgue   measure, we get the familiar expression
\begin{equation*} 
\mu(U,y)=\frac{\int_U D_\varepsilon  (y-L(x)) D_{pr}(x)dx  }{\int   D_\varepsilon (y-L(x))  D_{pr}(x)dx}
\end{equation*}
 for all  $y$  such that  $0<\int   D_\varepsilon (y-L(x))  D_{pr}(x)dx<\infty$.

\begin{remark}\label{siirto}
When $D_\varepsilon \geq 0$  almost everywhere,  $\mu_\varepsilon$ need not be equivalent to the 
Lebesgue   measure.   Moreover, the translated measure $\mu_{\varepsilon+L(x)}$ need not be absolutely continuous with respect to  $\mu_{\varepsilon}$. 
 \end{remark}

We consider next the convergence of posterior distributions. Let  $\mu_{X_n}$ be the finite-dimensional 
distributions that approximate $\mu_X$ and denote with  $X_n$ the corresponding  $F$-valued random variables that are statistically independent from $\varepsilon$.    Denote \begin{equation*}
\mu_n(U,y)=\frac{\int_{ U} D_\varepsilon  (y-L( x)) d\mu_{X_n}(x)  }{\int   D_\varepsilon (y-L(x)) d\mu_{X_n}(x) }
\end{equation*}
the  corresponding    solutions of estimating the probabilities 
of $X_n$ given  $Y_n=L(X_n)+\varepsilon  $.  When $D_{\varepsilon }$ is continuous and bounded,  the probabilities   $\mu_n(\cdot  ,y)$ converge weakly to 
\begin{equation*} 
\mu(\cdot ,y)=\frac{\int_{\cdot} D_\varepsilon (y-L( x)) d\mu_X(x)  }{\int  D_\varepsilon  (y-L(x))d\mu_X(x)}  
\end{equation*}
 for all $y$ such that 
 $$  \inf _n \int  D_\varepsilon  (y-L( x)) d\mu_{X_n}(x) >0 \, \text{and } \,    \sup _n \int  D_\varepsilon  (y-L( x)) d\mu_{X_n}(x)  <\infty
 $$ 
  whenever   $\mu_{X_n} $ converge weakly to $\mu_{X}$ by Theorem \ref{coco2}. Also Theorem \ref{kass} and  Theorem  \ref{coco}  are available, provided the assumptions hold.
 
In practical applications one often  takes such approximations of $\mu_{X}$  that can be  identified with a probability  distribution   $D_{pr}^{(n)}dx $ on $\mathbf R^n$  by  some linear isomorphism   $\mathcal I_n$ defined on a subspace  of full measure  i.e. $\mu_{X_n}\circ \mathcal I_n^{-1}= D_{pr}^{(n)} dx$.

 \subsection{Infinite-dimensional Gaussian noise}

The finite-dimensional Gaussian noise model is often chosen because of  its relatively straightforward  justification -- if  the total noise is produced by  many identical independent noise sources,  the sum is nearly Gaussian by the central limit theorem.   For instance, this applies to the origin of  thermal noise in electrical circuits, where heat motion of the charge carriers disturbs the analog signal. The usual model  of  thermal noise  is  white Gaussian noise, which  is  an acceptable approximation on usual frequencies.

We first recall  a method for constructing infinite-dimensional Gaussian random vectors by a  procedure linked to abstract Wiener spaces \cite{bog}. 

\subsubsection{Basics of Hilbert space-valued Gaussian random variables}

Let $H$ be a separable Hilbert space. We define  $Z$ as a  random sum 
\begin{equation}\label{summa}
Z= \sum_{i=1}^\infty Z_i e_i,
\end{equation}
where  $Z_i$ are independent standard normal random variables on $(\Omega,\Sigma,P)$  and  $\{e_i\}$ is an orthonormal basis of $H$.  Clearly,  the sum does not converge  a.s. in $H$.  Instead,  we take   a  larger Hilbert  space $G$   into which $H$ can be  imbedded with an injective Hilbert-Schmidt operator $j$.  When the range of the imbedding is  dense, the triple $(j,H,G)$ is a special case of an abstract Wiener space \cite{bog}. However, we do not  require the range to be dense.  Let $G'$ denote the dual space of $G$ and  $\langle \cdot,\cdot \rangle$ the duality between $G$ and $G'$. 

A sufficient condition for the a.s. convergence of  the random sums $\sum_{i=1}^n Z_i e_i$  in $G$  is that the series 
 \begin{equation*}
 \sum_{i=1}^\infty \mathbf E[\Vert Z_ie_i \Vert^2_G]=
 \sum_{i=1}^\infty  \Vert e_i\Vert_G^2.
 \end{equation*} 
is convergent \cite{kahane}.  But this follows from the Hilbert-Schmidt property of  the inclusion map $j$.  

Since $G$ is a separable Fr\'echet space (more generally, a locally convex Souslin  space \cite{schwarz}),  its  Borel 
$\sigma$-algebras with respect to  the weak and  the original topology  coincide.   The benefit of the weak topology is that the  measurability  of  the limit  $Z=\lim_{n\rightarrow \infty} \sum_{i=1}^n Z_i e_i$  can be checked similarly as in the case of real-valued  functions with sets of the type  $\cap_{i=1}^k \{ |\langle Z,\phi_i\rangle -a_i| <c_i\}$.  We conclude that  the a.s. limit $Z$ of the random sums defines a measurable mapping from $(\Omega,\Sigma,P)$ to $(G,\mathcal G)$. Its image measure  $\mu_Z=P\circ Z^{-1}$ can be viewed also as a countably additive cylinder set measure.

 In general, the mean of a   random variable $Z$ is the vector $m\in G''=G$ such that 
$\mathbf E[\langle Z,\phi\rangle]=\langle m,\phi\rangle$ for all  $\phi\in G'$ and the covariance operator of $Z$ is the mapping $C: G'\rightarrow G $ such that  
$$
\langle C\phi,\psi\rangle=  \mathbf  E[(\langle Z,\phi\rangle-\langle m,\phi\rangle)( \langle Z, \psi \rangle -\langle    m,\psi\rangle)]
$$ for all  $\phi,\psi\in G'$ \cite{bog}.   

Since limits of Gaussian random variables are Gaussian, the random variable 
\begin{equation*}
\langle Z+m,\phi\rangle= \langle  m,\phi\rangle +\sum_{i=1}^\infty Z_i \langle e_i, \phi\rangle,
\end{equation*} 
where $m\in G$,  has a characteristic function 
\begin{equation}\label{kara}
e^{i\langle m,\phi\rangle -\frac{1}{2}\langle C\phi,\phi\rangle  } = \mathbf E[e^{i\langle Z,\phi\rangle }]
\end{equation}
for all $\phi\in G'$. In our case, the random variable $Z$ has mean
$m=0$ and  covariance  $\langle C \phi,\phi\rangle = \sum_{i=1}^\infty \langle e_i, \phi \rangle^2$. The covariance $\langle C \phi,\phi\rangle$ is the squared norm of  $\phi$ in  the strong dual space $H'$ of $H$.  Indeed, the  linear form $\langle  j\cdot, \phi\rangle_{G,G'}$ is continuous on $H$ so it belongs to $H'$ and its norm is 
$\langle C\phi,\phi\rangle$. For short, we denote $\phi \in H'$.    The  covariance  $\langle C\phi,\phi\rangle $ for  any $\phi \in G'$ is finite, since $H\hookrightarrow G$  implies that  $G'\hookrightarrow H'$ continuously. 
The dual space $G'$                 is actually dense in $H'$  as a consequence of the Hahn-Banach theorem.  Indeed,  if $h'_0\in H'\backslash \overline{j'(G')}\not=\emptyset$ then there would exist     $h\in H''=H$ such that $\langle  h, h'_0\rangle=1$ and $\langle h, h'\rangle=0$ for every $h'\in j'(G')$. But $j'(G')$ separates the points   in $H$ because of the  injectivity of $j$.     Therefore, $h=0$ and  hence $j'(G')$ is dense in $H'$.

The mapping  $ G'\ni \phi\mapsto \langle C \phi,\phi\rangle $ has an extension $H'\ni g\mapsto \langle \bar C g,g\rangle:=\Vert g\Vert_{H'}^2$.  By the polarization equality, $\bar C$ is the isometric isomorphism   between $H'$ and $H$  defined by the Riesz representation theorem.  We continue to  denote $\bar C$ with $C$.

\begin{remark}\label{imb}
It is well-known that the sample space $G$ of $Z$  can be replaced with any bigger locally convex  Souslin  vector space $G_0$ into which   $G$  can be continuously and  injectively embedded. For example, $G_0$ may be  the distribution space $\mathcal D'(U)$, where $U\subset \mathbf R^n$ is open,    equipped with the usual weak topology. 
\end{remark}

 Measures  having   characteristic functions of the above form  \eqref{kara} are called Gaussian measures. Especially, the image measure $\mu_Z=P\circ Z^{-1}$ is Gaussian. Random variables, whose image measures are Gaussian, are   called Gaussian random variables. The space $H$ is  the so-called  {\it Cameron-Martin space} of $\mu_Z$.    
 
 By Theorems 3.2.3,  3.2.7 and  3.5.1 in \cite{bog} any zero-mean Gaussian random variable on a locally convex Souslin space  is equivalent with   a random variable of the form  \eqref{summa}.      More details on Gaussian measures  can be found in 
\cite{bog,gih,kuo}.

\subsubsection{Inverse problems  with Gaussian noise}\label{secgau}

We  consider the statistical inverse  problem  of estimating the distribution of $X$     given a sample of $Y=L(X)+\varepsilon$, where   $\varepsilon $  is   a zero mean Gaussian   random variable  that    has values  in a separable Hilbert space $G$.

  We denote  with  $H_{\mu_\varepsilon}$   the  Cameron-Martin space of $\mu_\varepsilon $ and with  $C_\varepsilon :H_{\mu_\varepsilon} '\rightarrow H_{\mu_\varepsilon} $ the covariance operator of $\varepsilon$.    The unknown random variable $X$ has values in some locally convex   Souslin topological vector    space $F$. The random variables   $\varepsilon $ and $X$ are taken to  be independent.   The direct theory $L:F\rightarrow G$ is  a continuous mapping that satisfies  the folloging  additional restrictive conditions:  $L:F\rightarrow H_{\mu_\varepsilon}$ is continuous,   the range of the  combined mapping  $C_\varepsilon^{-1} L$ belongs to  $G'$ where $G'$ is the  strong dual of $G$, and   the mapping $C_{\varepsilon }^{-1}L : F\rightarrow G'$ is continuous.

As an approximated   model, we take a sequence of  $F$-valued random variables $X_n$ that satisfy the same conditions as $X$.  We denote $Y_n:=L(X_n)+\varepsilon$.

 Recalling Remark \ref{tapsa},   we require that
\begin{equation}\label{alas}
\mathbf E \left[e^{a \Vert L(X) \Vert_{G'} }\right]  \wedge  \sup_n \mathbf E\left [e^{ a \Vert L( X_n) \Vert_{G'} }\right]< \infty 
\end{equation}
for all $a>0$.  The  condition holds  especially when  the range of $C_\varepsilon^{-1} L$ is bounded in $G'$.

According to  the  famous Cameron-Martin formula (see  Corollary 2.4.3 and Theorem 3.2.3 in  \cite{bog}), the Gaussian measures 
$\mu_\varepsilon $ and $\mu_{\varepsilon +L(x)}$  are equivalent 
when  $L(x)\in H_{\mu_\varepsilon} $. The corresponding   Radon-Nikodym density is
\begin{equation*}
\rho(x,z):=\frac{d\mu_{\varepsilon +L(x)}}{d\mu_\varepsilon }(z)= 
\exp\left( \langle z,  C^{-1}_\varepsilon  L(x) \rangle - \frac{1}{2} \Vert L(x)\Vert^2_{H_{\mu _\varepsilon} }\right),  \; z\in G.
\end{equation*}
\begin{remark}    In the Cameron-Martin formula,  the notation $\langle \cdot, \cdot \rangle$ is, in general,  a measurable extension  of the duality.  Namely,  the vector $C^{-1}_\varepsilon  L(x)$ need not belong to the space $G'$ but in the larger space $H'_{\mu_\varepsilon}$. But   $G'$ is dense in $H'_{\mu_\varepsilon}$.  Following  Lemma 2.2.8. in \cite{bog}, we may define  $\langle z, C^{-1}_\varepsilon  L(x)\rangle$ as the 
limit  of $\langle z,  \phi_n\rangle$ in $L^2(\mu_\varepsilon)$
 where $\phi_n\in G'$ converge to $C^{-1}_\varepsilon  L(x)$ in $H'_{\mu_\varepsilon}$ as $n\rightarrow \infty$. Especially, $\langle z, C^{-1}_\varepsilon  L(x)\rangle $ is a Gaussian random variable on $(G,\mathcal G, \mu_\varepsilon)$.  Different approximating sequences  lead to equivalent random variables, since the limits coincide in $L^2(\mu_\varepsilon)$.
\end{remark}

When  the range   $C^{-1}_{\varepsilon} L \subset  G'$,   we have  $\langle z,C^{-1}_\varepsilon L(x) \rangle = \langle z, C^{-1} \varepsilon\rangle _{G,G'}$, and, consequently,   the Radon-Nikodym density 
is separately continuous with respect to  $z$ on $G$  and 
with respect to $x$ on $F$.  By Theorem \ref{uni}, $\rho$ is  
$\mu_X\times \mu_\varepsilon$-measurable.   In  Theorem \ref{toka}, we may choose $\nu=\mu_\varepsilon $ and take
\begin{equation*}
\mu(U,y):=\frac{\int_U \exp\left( \langle y,  C^{-1}_\varepsilon  L(x) \rangle - \frac{1}{2} \Vert L(x)\Vert^2_{H_\varepsilon }\right)d\mu_{X}(x)}{\int \exp\left( \langle y,  C^{-1}_\varepsilon  L(x) \rangle - \frac{1}{2} \Vert L(x)\Vert^2_{H_\varepsilon }\right)d\mu_{X}(x)}
\end{equation*}
as  an essentially unique solution for all $y\in G$.   Note, that our assumptions guarantee that 
$$
0< \exp\left( \langle y,  C^{-1}_\varepsilon  L(x) \rangle - \frac{1}{2} \Vert L(x)\Vert^2_{H_\varepsilon } \right) \leq 
   \exp\left(  \Vert y \Vert_ {G}  \Vert  C^{-1}_\varepsilon  L(x)  \Vert _ {G'}\right) \in L^1(\mu_X)
$$ 
so that the set $M_0$ in \eqref{mnolla} is empty.
Similarly, when  $X_n$ satisfies the same conditions as $X$, we obtain
\begin{equation*}
\mu_n(U,y):=\frac{\int_U \exp\left( \langle y,  C^{-1}_\varepsilon  L(x) \rangle - \frac{1}{2} \Vert L(x)\Vert^2_{H_\varepsilon }\right)d\mu_{X_n}(x)}{\int \exp\left( \langle y,  C^{-1}_\varepsilon  L(x) \rangle - \frac{1}{2} \Vert L(x)\Vert^2_{H_\varepsilon }\right)d\mu_{X_n}(x)}
\end{equation*}
for all $y\in G$.

We  consider  next the partial uniqueness of the solutions $\mu$ and   $\mu_n$ on $\mathcal F \otimes G$. 
Denote with $S_{\mu_\varepsilon}$ the support  of  $\mu_\varepsilon $ on $G$, which coincides with  the closure of the Cameron-Martin space  $H_{\mu_\varepsilon}$ in $G$  by Theorem 3.6.1  in \cite{bog}.   The measure $\mu_{\varepsilon+L(x)}$ is equivalent with $\mu_{\varepsilon}$ by the Cameron-Martin formula. Hence, the measures $\mu_Y$ and   $\mu_{Y_n}$ have the same topological  support as  the measure $\mu_{\varepsilon}$ by Lemma \ref{konti}.   We conclude that   $S_{\mu_Y}=S_{\mu_{Y_n}}= \overline{H_{\mu_\varepsilon}}$.  Since $\sup_{z\in K } \rho(x,z) \leq  \exp (C \Vert C^{-1}_\varepsilon L(x) \Vert_{G'})$, the solutions $\mu$ and  $\mu_n$  are $\mathcal F$-continuous  on $G\cap  \overline{H_{\mu_\varepsilon}}$  by Theorem \ref{uni}. Hence, $\mu$  and $\mu_n$ are the only $\mathcal F$-continuous   solutions  on $G\cap  \overline{H_{\mu_\varepsilon}}$  by Corollary \ref{uni2}.    In the light of  Corollary \ref{uni2} and the discussion preceding it,   the partial uniqueness is not  so simple 
   in the situation  described in  Remark \ref{imb}.

In order to apply Theorem   \ref{coco2}, we use the  continuity of $x\mapsto \rho(x,y)$ and the uniform integrability that follows from  the assumption  \eqref{alas}.   Consequently, Theorem \ref{coco2} holds.   If, for example, the  range of $C_{\varepsilon} L$ is a bounded set in $G'$, also Theorem \ref{coco} is available.

\begin{remark}\label{gah}
In general, the measure $\mu_Y=\mu_{L(X)+\varepsilon}$ does not  satisfy $\mu_{\varepsilon+L(x)}<< \mu_Y$ for $\mu_X$-almost every $x$.  Indeed,  take $X$ and $\varepsilon $  to be independent Gaussian  random variables with the  same Cameron-Martin space $L^2(I)$, where $I$ is the unit interval $(0,1)$.  Let $L$ be the identity.   If $\mu_Y(U)=0$, then 
$\mu_{\varepsilon +x}(U)=0$ for $\mu_X$-almost every $x$ by the formula 
\begin{equation} \label{gaga}
\mu_Y(U)= \mathbf E[ \mathbf E 1_Y(U)|X]]= \int \mu_{\varepsilon +x}(U) d\mu_X(x).
\end{equation}
Suppose that  $\mu_{\varepsilon +x}\ll \mu_Y$ for $\mu_X$-a.e. $x$, say for all $x\in M$ such that $\mu_X(M)=1$. 
 The random variable $Y=X+\varepsilon $ is also Gaussian, and any  two Gaussian measures on the same locally convex space are either equivalent or singular. But then  $\mu_{\varepsilon +x_1}$ is equivalent to $\mu_Y$ and $\mu_{\varepsilon +x_2}$ is equivalent to $\mu_Y$ for any $x_1,x_2\in M$ so also the two measures
   $\mu_{\varepsilon +x_1}$ and   $\mu_{\varepsilon +x_2}$ are equivalent.   But $P(X\in L^2(I))=0$ so equivalence should hold also for some $x_1,x_2\notin L^2(I)$, which is impossible  by the Cameron-Martin theorem. 
   The $\mu_X$-zero measurable set in (\ref{gaga}) necessarily depends    on $U$ in this case. 
  \end{remark}

\subsection{Gaussian dominated noise} \label{ridi}

We consider a simple  modification of   Gaussian noise. Suppose that the assumptions in 
Section \ref{secgau} hold except that  instead of   $Y=L(X)+\varepsilon$ we are  observing  $\widetilde Y= L(X)+ \widetilde \varepsilon$, where 
 $\mu_{\tilde \varepsilon}$ is dominated by the Gaussian measure $\mu_\varepsilon$ i.e. 
 $$
\frac{d\mu_{\widetilde \varepsilon}}{d\mu_\varepsilon} (y) = f(y)
 $$
 for some $f\in L^1(\mu_\varepsilon)$.   The translation of $\mu_{\widetilde \varepsilon} $ by $L(x)$  has the form 
\begin{eqnarray*}
\mu_{\widetilde \varepsilon+L(x)}(V)&= &\int   1_{ V } (y+L(x)) d\mu_{\widetilde \varepsilon}(y) \\
&=&  \int   1_{ V } (y+L(x))f(y)  d\mu_{ \varepsilon}(y) \\
 &= &  \int_V  f(y-L(x)) d\mu_{\varepsilon+L(x)}(y)\\
&=&   \int_{V} f(y-L(x)) \exp\left( \langle y,  C^{-1}_\varepsilon  L(x) \rangle - \frac{1}{2} \Vert L(x)\Vert^2_{H_ {\mu_\varepsilon}}\right) d\mu_{\varepsilon}(y).
\end{eqnarray*}
The integrand is a $\mu_X\times \mu_{\varepsilon}$-measurable functions as a product of two  $\mu_X\times\mu_ \varepsilon$- measurable functions.
By Theorem \ref{toka}, the posterior distribution of $X$ given a sample $y$ of   $\widetilde Y=L(X)+\widetilde \varepsilon$ can be taken to be 
\begin{eqnarray}\label{donoise}
\mu(U,y)=\frac{\int_U  f(y-L(x))\exp\left( \langle y,  C^{-1}_\varepsilon  L(x) \rangle - \frac{1}{2} \Vert L(x)\Vert^2_{H_ {\mu_\varepsilon}}\right) d\mu_X(x)}{\int  f(y-L(x)) 
\exp\left( \langle y,  C^{-1}_\varepsilon  L(x) \rangle - \frac{1}{2} \Vert L(x)\Vert^2_{H_ {\mu_\varepsilon}}\right) 
d\mu_X(x)}
\end{eqnarray}
whenever the denominator is positive. 

 For instance,  let  $\widetilde \varepsilon$  to be  a restriction of $\varepsilon$ to some open set $K\in \mathcal G$  that has positive   $\mu_\varepsilon$-measure.  This means that  the noise  $\widetilde \varepsilon  =\varepsilon|_K$ has the  distribution  
\begin{equation}\label{ehto}
\mu_{\widetilde \varepsilon} (V)= \frac{P(\varepsilon\in K\cap V)}{P(\varepsilon\in K)}
\end{equation}  
for all $V\in \mathcal G$  i.e.  we consider  conditional probabilities 
$$
\mu_{\widetilde \varepsilon} (V)=  \mu_\varepsilon(V|K).
$$
Note that as a Borel set,  $K$ is of  the form $
K=\{ y\in G: (\langle y,\phi_1\rangle, \langle y,\phi_2\rangle,\cdots  ) \in  E\},
$
where $\phi_i\in G'$ separate the points in $G$  and  $E\in \mathcal B( \mathbf R^\infty)$. 
The Radon-Nikodym density of $\mu_{\widetilde \varepsilon}$ with respect to 
$\mu_{\varepsilon}$ is by \eqref{ehto}
   $$
 f(y) = \frac{d\mu_{\widetilde \varepsilon}}{d\mu_\varepsilon}(y)= \frac{1}{ \mu_\varepsilon(K)}1_K(y).
   $$
By Theorem \ref{toka},  an essentially unique  posterior distribution of $X$ given  a sample  $y_0$  of   $\widetilde Y=L(X)+\widetilde \varepsilon$ can be represented as \begin{eqnarray*}
\mu(U,y_0)=\frac{\int_U 1_K(y_0-L(x))\exp\left( \langle y_0,  C^{-1}_\varepsilon  L(x) \rangle - \frac{1}{2} \Vert L(x)\Vert^2_{H_ {\mu_\varepsilon}}\right) d\mu_X(x)}{\int 1_K(y_0-L(x)) 
\exp\left( \langle y_0,  C^{-1}_\varepsilon  L(x) \rangle - \frac{1}{2} \Vert L(x)\Vert^2_{H_ {\mu_\varepsilon}}\right) 
d\mu_X(x)}
\end{eqnarray*}
whenever the denominator is positive.  We see that when we can  exclude noise patterns, the  posterior distribution will  concentrate more on the true  value $x_0$  (when $L$ is  injective). 
When $\mu_X(  L^{-1} (\partial (\{y_0\}-K)))=0$,  the mapping $x\mapsto 1_K(y-L(x))$ is continuos on a set of full  $\mu_X$-measure and the convergence results are  hence available. 

For example,  take $G=F=H^{-1}(a,b)$, where  $-\infty <a<b<\infty$, and  set    $Lx= \sum_{i=1}^\infty c_i  \langle x, e_i\rangle  e_i$ for all $x\in L^2(a,b)$, where  $\{e_i\}_{i=1}^\infty $ is an orthonormal basis of $L^2(a,b)$ and the constants $c_i>0$ satisfy  $ \{(1+i) c_i\}_{i=1}^\infty \in \ell ^2$. Then $L: L^2(a,b) \rightarrow  H^1(a,b)$ is continuous.   Set  $X= \sum_{i=1}^\infty X_i e_i $   and $\varepsilon = \sum_{i=1}^\infty \varepsilon_i e_i $, where $\varepsilon_i$ and $X_i$, $i\in \mathbf N$,   are independent standard normal random variables.  Set $$K=\{ y\in G: | \sum_{i=1}^k \langle y , e_i\rangle|\leq C \}.$$  Then $\mu_{\varepsilon}(K)>0$ and 
\begin{eqnarray*}\mu_X(L^{-1} (\partial  (\{y_0\}-K))) &=&  
\mu_X(  L^{-1} (\{y \in G:  | \sum_{i=1}^k  \langle y,e_i\rangle + \langle y_0,e_i\rangle   | = C  \})\\
&=&\mu_X( \{ y \in L^2(a,b) :  |  \sum_{i=1}^k   c_i( \langle y,e_i\rangle +   \langle y_0,e_i\rangle)  | = C \})\\
&=&  \prod_{i=1}^k P (   |Z | = C)=0,
\end{eqnarray*}
where $Z=\sum_{i=1}^k   c_i(  X_i +   \langle y_0,e_i\rangle)$ is a Gaussian random variable.  
 
  The partial uniqueness  with respect to the  topology of $G$ remains  an open question.

Another example arises from  the Girsanov formula.  We equip $G=C([0,T])$, where $T>0$,  with  the usual supremum norm. The space $G$ is  then complete separable Banach space and its dual space $G'$ is the space of Radon measures on $[0,T]$. We assume that the observation is of the form $Y_t=L(X)_t+\widetilde \varepsilon_t$ for $0\leq t\leq T$, where $F$-valued $X$ and $C([0,T])$-valued $\widetilde\varepsilon$ 
are statistically independent and $L:F\rightarrow C([0,T])$ is a continuos mapping.  More precisely, we assume a stronger condition that $L:F\rightarrow C^2_0(0,T)$ is continuous.

   Suppose that the noise $\widetilde \varepsilon\in G$ is of the form
$$
\widetilde \varepsilon _t  =  \varepsilon_t + \int_0 ^t  a(s; \varepsilon_s ) ds  
$$ 
where $\varepsilon_t$ is an ordinary Brownian motion on $[0,T]$  and $a: [0,T] \times \mathbf R \rightarrow \mathbf R$ is continuous.  Note that $\widetilde \varepsilon_t $ indeed is  a $C([0,T])$-valued  random variable since  the continuous functionals $\{\delta_t: t\in \mathbf Q \cap [0,T]\}$ separate the points in $G$ and, therefore, 
  also generate the $\sigma$-algebra of $G$.

  It is  well-known that the Cameron-Martin space of the Brownian motion  on $[0,T]$ is the separable  Hilbert space
$\{ f\in H^1(0,T): f(0)=0\}$ equipped with the norm $\Vert f' \Vert _{L^2}$,  the covariance operator $C_\varepsilon$  has kernel $\min(t,s)$ and 
$C^{-1}_\varepsilon = \frac{d^2}{dt^2}$ on  $\{f\in H^2(0,T): f(0)=0, f'(T)=0\}$ (see \cite{bog}). By the Cameron-Martin theorem
$$
\frac{d \mu_{\varepsilon+L(x)}}{d\mu_\varepsilon} (y) =  \exp\left(  \int_0^T  y_s  \frac{d^2 L(x)_s }{ds^2} ds  - \frac{1}{2} \left \Vert\frac {d L(x)_s}{ds}\right\Vert^2_{L^2(0,T)}\right).
$$
  The Girsanov formula 
$$
\frac{d \mu_{\widetilde \varepsilon}}{d\mu_\varepsilon} (y) =
   \exp \left( \int_0^T a(s,y_s) dy_s - \frac{1}{2}  \int_0^T |a(s,y_s)|^2 ds \ \right),  
 $$
where  the first integral is a sample  of the  corresponding stochastic integral, holds when the  Novikov's condition
\begin{equation}\label{novikov}
\mathbf E \left[\exp\left( \int_0^T |a(s; \varepsilon_s )|^2 ds \right) \right]<\infty
\end{equation}
is satisfied (see  \cite{oek}). For example, if $|a(s,x)|\leq  C(1+|x|)$ for some $C>0$, then \eqref{novikov} holds  since
 $$
\mathbf E [e^{C \int_0^T  \varepsilon_s^2 ds }]\leq  e^{\frac{C^2}{4a}}\mathbf    E [ e^{ a  \Vert\varepsilon  \Vert^2_{L^2(0,T)}}]<\infty
 $$ 
 by the Fernique  theorem (see Corollary 2.8.6 in \cite{bog}).   For instance, take $a(s,x)=\frac{2x}{1+x^2}$. By the It\=o formula,  we see that  the mapping 
 $$
\varepsilon \mapsto  \int_0^T a(\varepsilon_s) d\varepsilon_s=  \ln (1+\varepsilon_T^2) - \int_0 ^T   \frac{1- \varepsilon_s^2}{(1+\varepsilon_s^2)^2 }ds  
 $$
 extends to a    continuous functional  on $C([0,T])$. Thus $y\mapsto \frac{d\mu_{\widetilde \varepsilon}}{ d\mu_{\varepsilon}}(y)$
 has a continuous version
 $$
  \frac{d\mu_{\widetilde \varepsilon}}{ d\mu_{\varepsilon}}(y) =  (1+y_T^2)  \exp\left( - \int_0 ^T   \frac{1- y_s^2}{(1+ y_s^2)^2 }ds  - \frac{1}{2}  \int_0^T \left | \frac{2y_s}{1+y_ 2^2}\right|^2 ds \right) 
 $$
   on $C([0,T])$.  As in \eqref{donoise}, we obtain an explicit solution
$$
\mu(U,y)= \frac{\int_U  \frac{d \mu_{\widetilde \varepsilon}}{d\mu_\varepsilon} (y-L(x))     \frac{d \mu_{\varepsilon+L(x)}}{d\mu_\varepsilon} (y)  d\mu_X(x) } {\int   \frac{d \mu_{\widetilde \varepsilon}}{d\mu_\varepsilon} (y-L(x))        \frac{d \mu_{\varepsilon+L(x)}}{d\mu_\varepsilon} (y)   d\mu_X(x) }
$$
 of the statistical inverse problem of estimating the distribution of $X$  given the observation 
 $Y_t=L(X)_t+\varepsilon_t+ \int_0^t a(s, \varepsilon_s) ds$ on $[0,T]$.
 The posterior convergence results are available for approximated prior distribution. 

In general,  any $G$-valued random variable  $\widetilde \varepsilon$ whose image measure is  absolutely continuous with respect to  a zero mean  Gaussian measure $\mu_\varepsilon$ satisfies  
 $$
 \widetilde  \varepsilon =  \varepsilon + T(\varepsilon) $$
  in distribution for some mapping  $T: G \rightarrow H_{\mu_\varepsilon}$  
(see Corollary 4.2  in \cite{bogata}).  

\subsection{Spherically invariant noise}\label{spheri}

Let $F$ and $G$ be locally convex Souslin topological   vector spaces. We say that $\varepsilon $ is   {\it a spherically invariant $G$-valued random variable} if  $\varepsilon =\gamma Z$, where $Z$ is a zero-mean Gaussian $G$-valued random variable whose Cameron-Martin space is  infinite-dimensional, and $Z$ is statistically independent from   
a non-negative real-valued random variable $\gamma$  whose distribution has no atom at zero.

The expression "spherically invariant random process (SIRP)" is used in  the engineering literature \cite{sph} while the more descriptive but little used  expression "$H_{\mu_Z}$-spherically symmetric measure"  appears in the mathematical literature (see Definition 7.4.1 in \cite{bog}).  The latter has emphasis on the fact that  the measure is  only invariant   with respect to  orthogonal operators  on  $H_{\mu_Z}$ (see Theorem 7.4.2 in  \cite{bog}).

In order to study the posterior  measure of    $X$ given 
$Y=L(X)+ \gamma Z$, we apply an averaging principle  together with  the following lemma.

\begin{lemma}\label{olk}
Let $F$ and $G$ be locally convex Souslin topological  vector  spaces.
Let $Z$ be a zero-mean Gaussian $G$-valued random variable whose Cameron-Martin space is infinite-dimensional. Let  $X$ be an $F$-valued random variable, and
 let  $\gamma$ be a non-negative  random variable whose distribution has no atom at zero. Suppose that 
 $\gamma$, $X$ and $Z$ are statistically independent.

Let   $L:F\rightarrow G$   be a continuous mapping such that
$ L(F)\subset H_{\mu_Z}$, where $H_{\mu_Z}$ is the Cameron-Martin space of $\mu_Z$.  Let   $\{e_i\}_{i=1}^\infty$  be an orthonormal basis of $H_{\mu_Z}$ such that  $C_Z^{-1} e_i \in G'$, where   $C_Z$  is  the covariance operator of $Z$.
Set $Y=L(X)+ \gamma Z$.

  For  any $f\in L^1(\mu_{(Y,\gamma)})$,  the conditional expectation $$
\mathbf E[f(Y,\gamma)| \sigma(Y)] (\omega)= f(Y(\omega),\gamma_{Y(\omega)})
$$  
for $P$-almost every $\omega\in \Omega$, where $ y\mapsto \gamma_y$ is a $\mathcal G$-measurable  function on $G$ that satisfies
\begin{equation}\label{pre} 
\gamma_y= \left(\lim _{n\rightarrow \infty}  \frac{1}{n} \sum_{i=1}^n \langle y, C_Z ^{-1}e_i \rangle^2 \right) ^\frac{1}{2}
 \end{equation}
 whenever a finite limit exists and $\gamma_y=0$ otherwise.
 \end{lemma} 
\begin{proof}
The mapping $y\mapsto \gamma_{y}$ is 
indeed measurable since the set
$$
N=\{ y\in G: \lim _{n\rightarrow \infty}  \frac{1}{n} \sum_{i=1}^n \langle y, C_Z ^{-1}e_i \rangle^2 \not\exists\}.
 $$ 
  is a Borel set (see Lemma 2.1.7 in \cite{bogm}). We show in a moment  that $\gamma=\gamma_Y$ $P$-almost surely.   
  Then  the conditional expectations of $f(Y,\gamma)$ and $f(Y,\gamma_Y)$ coincide  since the two random variables coincide almost surely.  In order to conclude the claim, we note that  $\gamma_{Y(\omega)}$ is  $Y^{-1}(\mathcal G)$-measurable as a combination of two measurable functions.

  The random variables $\gamma $, $X$ and  $Z$ are statistically independent, which implies that  their image measure $\mu_{(\gamma,X,Z)}$ is a product measure
 on the product space  $\mathbf R_+\times F\times G$.  
 
 Since  $Z$ has a Gaussian distribution, the random variables $\langle Z, C_{Z}^{-1} e_i\rangle$ are 
 statistically independent standard normal random variables.  The same holds for the random variables 
 $(t,x,z)\mapsto \langle z, C_{Z}^{-1} e_i\rangle$  on the measure space 
$ (\mathbf R_+\times F\times G, \mathcal B(\mathbf R_+\times F\times G)   ,\mu_{\gamma}\otimes \mu_X\otimes\mu_Z)$.
The random variable 
$$
(t,x,z)\mapsto  L(x)+tz
$$
has the following property.  The law of large numbers implies that
 $$
\lim_{n\rightarrow \infty } \frac{1}{n}\sum_{i=1}^n \langle L(x)+ t  z, C_{Z}^{-1}e_i\rangle ^2 =  \lim_{n\rightarrow \infty } \frac{t^2 }{n}\sum_{i=1}^n \langle   z, C_{Z}^{-1} e_i\rangle ^2   = t^2 
 $$ 
 for any $t\in \mathbf R_+$, $x\in F$, and  $\mu_Z$-a.e. $z\in G$.  Since the image measure has the product  structure, this  also holds  
for  $\mu_{(\gamma,X,Z)}$-almost every $(t,x,z)$.  Hence,
$$
(\gamma,X, Z)^{-1} \{ (t,x, z): \lim_{n\rightarrow \infty } \frac{1}{n}\sum_{i=1}^n \langle L(x)+ t  z, C_{Z}^{-1} e_i\rangle ^2  =t^2   \}
$$
has full $P$-measure i.e.
$$
\lim_{n\rightarrow \infty } \frac{1}{n}\sum_{i=1}^n \langle  L(X)+\gamma  Z, C_{Z}^{-1} e_i\rangle ^2 =  \gamma^2  
$$ 
$P$-almost surely. 
 \end{proof}

The averaging principle for  the posterior distributions is given in the  following lemma.
Note, that also topological products of Souslin spaces are   Souslin spaces.
\begin{lemma}\label{aver}
Let the assumptions of the Lemma \ref{olk} hold. 

 A  solution   $\mu(\cdot,  y)$ of the statistical inverse problem of estimating the distribution  of $X$ given  a sample $y$ of 
$Y= L(X)+\gamma Z$ coincides  $\mu_Y$-almost surely with a Borel measurable solution  
$\tilde \mu(\cdot, (y,\gamma_y))$ of the statistical inverse problem of  estimating the distribution   of $X$ given  $(Y,\gamma)=(y,\gamma_y)$, where $\gamma_y$ is defined by \eqref{pre}.
 \end{lemma}

\begin{proof}
The $\sigma$-algebra  $\sigma(Y)$ generated by $Y=L(X)+\gamma Z$ is a sub-$\sigma$-algebra of 
the $\sigma$-algebra $\sigma( (Y,\gamma)$ generated by the $G\times \mathbf R_+$-valued 
random variable $(Y, \gamma) = (\gamma Z+ L(X),\gamma)$.  By Lemma \ref{olk} and a property of conditional expectations, the 
 solutions satisfy
\begin{eqnarray*}
\mu(U, Y)&=&\mathbf E[1_U(X)| \sigma(Y)] = \mathbf E[  \mathbf E[1_U(X)| \sigma(Y,\gamma)] |\sigma(Y)] \\
&=& =\mathbf E [\tilde  \mu(U, (Y,\gamma)) | \sigma(Y)]=  \tilde \mu(U, (Y, \gamma_Y))
\end{eqnarray*}
almost surely  for a fixed $U\in \mathcal  F$.   It is easy to see that $y\mapsto \tilde \mu(U, (y, \gamma_y))$ is  Borel-measurable.  
By the Souslin property, it is enough to consider only countably many 
$U\in \mathcal F$ in order to identify the two  measures. Hence,   $(U,y)\mapsto \tilde \mu(U, (y, \gamma_y))$
 is a solution of the  statistical inverse problem of estimating the distribution  of $X$ given  a sample $y$ of 
$Y= L(X)+\gamma Z$. 
\end{proof}

\begin{theorem}\label{subi}
Let the assumptions of the Lemma \ref{olk} hold. The essentially unique  posterior distribution    
of $X$ given a sample $y$ of $Y=L(X)+ \gamma Z$ has a version
\begin{equation}\label{kaava}
\mu(A, y) =   \frac{\int_ A \exp \left( \langle  y,  \gamma_y^{-2}C^{-1}_Z L(x) \rangle - \frac{1}{2\gamma_y^2} \Vert L(x) \Vert_{H_{\mu_Z}}^2  \right) d\mu_X (x)   }{\int_ F \exp \left( \langle  y,  \gamma_y^{-2}C^{-1}_Z L(x) \rangle - \frac{1}{2\gamma_y^2} \Vert L(x) \Vert_{H_{\mu_Z}}^2  \right) d\mu_X (x)   }, 
 \end{equation}
for all $y\in G$ such that the limit 
 $$
 \gamma_y=\left(\lim _{n\rightarrow \infty}  \frac{1}{n} \sum_{i=1}^n \langle y, C_Z ^{-1}e_i \rangle^2 \right) ^\frac{1}{2}
 $$
 exists and does not vanish.
\end{theorem}
 \begin{proof}     
 Let us calculate the posterior distribution    of $X$ given $(Y,\gamma)$. 

 The conditional distribution of $(Y,\gamma)$ given a sample $x$ of $X$ is 
  $\mu_{(\gamma Z+ L(x), \gamma)} (C \times B)$, 
 where $C\in \mathcal G$ and $B\in \mathcal B(\mathbf R_+)$ by Lemma \ref{eka2}.  Furthermore, the conditional distribution  of $L(x)+\gamma Z$ given $\sigma((\gamma,X))$  is  $\mu_{\gamma(\omega_0) Z +L(x)}$. 
 Taking  conditional expectations inside the integral gives
 \begin{eqnarray*}
  \mu_{(\gamma Z+ L(x), \gamma)} (C \times B)&=& P( \gamma Z+L(x) \in C \cap \gamma \in B) \\
  &=& \mathbf E [1_C(\gamma Z+L(x)) 1_B(\gamma)]\\
  &=& \mathbf E [\mathbf E[ 1_C(\gamma Z+L(x))  | \sigma(\gamma)] 1_B(\gamma)] \\
  &=&  \int \mu_{a Z +L(x)} (C) 1_B(a) d\mu_{\gamma}(a).
  \end{eqnarray*}
  We may now use the absolute continuity of the translated measures $\mu_{aZ+L(x)}$ with respect to 
  $\mu_{aZ}$ which follows from the Cameron-Martin theorem. We obtain 
  \begin{eqnarray*}
 \mu_{(\gamma Z+ L(x), \gamma)} (C \times B) &=& \int_B  \int_ C e^{ \langle  y,  a^{-2}C^{-1}_Z L(x) \rangle  - \frac{1}{2a^2} \Vert L(x) \Vert_{H_{\mu_Z}}^2  } d\mu_{aZ}(y)   d\mu_\gamma(a) \\
  &=& \int_ {C\times B}  e^ { \langle  y,  a^{-2}C^{-1}_Z L(x) \rangle - \frac{1}{2a^2} \Vert L(x) \Vert_{H_{\mu_Z}}^2  } d\mu_{aZ}(y) d\mu_ \gamma(a)\\
  &=& \mathbf E [  \mathbf E[ 1_ {C\times B}(\gamma  Z,\gamma ) e^{ \langle  \gamma Z,  \gamma ^{-2}C^{-1}_Z L(x) \rangle - \frac{1}{2\gamma^2} \Vert L(x) \Vert_{H_{\mu_Z}}^2  } | \sigma(\gamma) ] ]\\
  &=& \int_ {C\times B}     e^{ \langle  y,  a^{-2}C^{-1}_Z L(x) \rangle - \frac{1}{2a^2} \Vert L(x) \Vert_{H_{\mu_Z}}^2} d\mu_{(\gamma Z, \gamma )} (y,a). 
    \end{eqnarray*}
  Hence, the Radon-Nikodym derivative of $\mu_{(\gamma Z+L(x),\gamma)}$ with respect to $\mu_{(\gamma Z,\gamma)}$ is 
   $$
   \frac{d\mu_{(\gamma Z + L(x),\gamma)}}{ d\mu_{(\gamma Z, \gamma)} }  (y,a)= \exp \left( \langle  y,  a^{-2}C^{-1}_Z L(x) \rangle - \frac{1}{2a^2} \Vert L(x) \Vert_{H_{\mu_Z}}^2  \right). 
   $$ 
 The   posterior distribution    of $X$ given a sample $(y,a)$ of $(Y, \gamma)$ has a version
 $$
 \mu (U, (y,a))= \frac{\int_ U \exp \left( \langle  y,  a^{-2}C^{-1} _ZL(x) \rangle - \frac{1}{2a^2} \Vert L(x) \Vert_{H_{\mu_Z}}^2  \right) d\mu_X (x)   }{\int_ F \exp \left( \langle  y,  a^{-2}C^{-1} _ZL(x) \rangle - \frac{1}{2a^2} \Vert L(x) \Vert_{H_{\mu_Z}}^2  \right) d\mu_X (x)   } 
 $$
 for all $y\in G$ and $a\not=0$.  We obtain the required result  by  Lemma  \ref{aver}.
   \end{proof}

Posterior convergence holds under the same  conditions as in the Gaussian case. 

\begin{remark}
The posterior distribution  \eqref{kaava}  does not depend on the distribution of $\gamma$.  Especially, $\gamma$ does not necessarily have finite moments.
  \end{remark}
\begin{remark}\label{epa}
If the sample $y\in H_{\mu_Z}$, then  the  estimated random number $\gamma_y=0$. Consequently,  we can not apply   Theorem \ref{uniq}    for the solution \eqref{kaava}  on   any measurable linear  subspace  of $G$ of full $\mu_{Z}$-measure, since it contains the Cameron-Martin space  $H_{\mu_Z}$. Besides the Lusin theorem, nothing seems to be known about the continuity of the measurable function $y\mapsto \gamma_y$.  Even though the continuity of the posterior distribution as a function of observations remains an open question, we can anticipate from the form of the posterior distribution  that the prior distribution  will have a good regularizing effect on  the corresponding  ill-posed inverse problem. 
 \end{remark}

Following \cite{taqqu},   we call  $\varepsilon =\gamma Z $    {\it  a symmetric  $\alpha$-stable  sub-Gaussian $G$-valued random variable} if 
  $\gamma= \sqrt{\Gamma}$, where  the  non-negative random variable $\Gamma $ satisfies
  $$
  \mathbf E [e^{-t \Gamma}] = e^{- t^{\alpha/2}},  \; t>0,
  $$ 
  for some $0<\alpha <2$, and $Z$ is a zero mean $G$-valued Gaussian random variable.

For instance,   $\alpha$-stable   random variables  are used as approximative    models for    ambient noise. 
   An example of ambient noise is the acoustic noise in oceans originating 
from  e.g. shipping, rain fall, waves, animal activity, bubbles, cracking of ice and geological processes \cite{mar,ur}. It disturbs  acoustic communication and  active acoustic remote sensing  in underwater environments \cite{sea,ray}.  
 The  finite-dimensional distributions of ambient  noise are thought to originate from many disturbances occurring in natural environments: typically few strong and a large number of weak disturbances of different orders.   The variances of   individual disturbances  are often  such that  Lindeberg's condition, which is a sufficient condition  (and in some cases also necessary)  for the applicability of  the classical central limit theorem,  does not hold \cite{ambient}.  A generalized central limit theorem states that a.s. converging sums of independent random variables necessarily have stable distributions (see Definition 1.1.5 in \cite{taqqu}).     Non-Gaussian stable distributions  exhibit heavy tails, which explains why  the Gaussian  distributions are  not the best ones  for  modeling ambient noise. 
 Symmetric $\alpha$-stable sub-Gaussian random variables are perhaps the most simple subclass of stable distributions.     
 
 Sub-Gaussian noise is encountered also in fMRI (functional magnetic resonance imaging), where it models physiological noise, e.g.  disturbances originating from  breathing and  heartbeat \cite{adali}.

  Spherically symmetric noise models are used also  as approximative   models  in  high resolution radar imaging for describing  the ground-clutter (i.e. unwanted echoes of the transmitted radar signal from the ground), and also sea-clutter (i.e. echoes from the surface of the sea) \cite{conte,conte1,conte2}.  It should be noted that the modeling of radar clutter and underwater noise is not yet a mature field of science. Beside of  spherically symmetric models also other models  have been developed and better models are pursued after.
 
Noise is usually   rougher than the signal by rule of thumb. In  the above applications,  it is not verified whether this holds for the noise $\varepsilon$ and  signals $L(x)$, where $x\in F$.   For radar imaging  this is not  a critical point since 
the reflected signal acquires some regularity from the transmitted signal.

\subsection{Subordinated  noise}\label{sasse}

 We consider  another  generalization of  Gaussian noise  that is similar to 
  spherically symmetric noise.    
  
  Let  $B_t$  be a Brownian motion on $\mathbf R_+$ satisfying $B_0=0$ almost surely.   Subordinated  noise is here  defined as a time-changed process  $$\varepsilon_t= B_ {\alpha_t}, $$  where    $\alpha_t$ is a strictly increasing  stochastic process
    that is  statistically independent from the Brownian motion $B_t$. We assume that   $\alpha_t$ has bi-Lipschitz-continuous sample paths and  satisfies $\alpha_0=0$. For example, $\alpha_t$ can be an integral function of some  statistically independent Gamma process starting from a non-zero value.  Such a distribution of $\alpha$  can reflect inaccuracies that are  believed to be  present in 
    the covariance operator of the noise $\varepsilon$.

   \begin{lemma}
   The random function $\varepsilon_\cdot $ on $[0,1]$ is a $C([0,1])$-valued random variable.
     \end{lemma}
 \begin{proof}
The sample paths of $\varepsilon$ are continuous functions as  compositions of continuous functions.  Moreover, the space $C([0,1])$ is a separable Fr\'echet space, which implies that its Borel $\sigma$-algebra is  generated 
by the cylinder sets 
$$
A=\{ f\in C([0,1]):  f(t_i) \in U_i \, \forall \, i \in I\} 
$$
where $U_i\in \mathcal B(\mathbf R)$,  $I\subset \mathbf N$ are finite sets, and  $\cup_{i=1}^\infty t_i $ is a dense subset of $[0,1]$.   (see Theorem A.3.7 in \cite{bog}).  It is enough to check that  the mapping 
 $$
\omega\mapsto   B_{\alpha_{t_i}} 
 $$
 is a random variable for any $t_i\in [0,1]$.  But this follows from the joint measurability of the Brownian motion
   from $[0,1]\times \Omega$ into $\mathbf R$.
   \end{proof}  
   
  We take $G= C([0,1])$, $\mathcal G= \mathcal B(C([0,1]))$, and 
     denote  $\mu_\varepsilon(A)= P( \varepsilon _\cdot \in A)$ for 
   any  Borel set $A\subset C([0,1])$.

 \begin{lemma}\label{soso}
  The  Gaussian measure $ \mu_{B_{\alpha (\omega)}+L(x)}$ that  has mean $L(x)$ and  the  covariance  operator with kernel $C_{\alpha(\omega)}(t,s)=\min (\alpha_t (\omega),\alpha_s(\omega))$ on $ [0,1]\times [0,1]$ is a version of the   conditional probability  $V\mapsto \mathbf E [ 1_ V(B_{\alpha}+L(x) )  | \sigma(\alpha)] (\omega)$ on $\mathcal G$. 
  \end{lemma}
\begin{proof}
 By defining $B_t=1_{t\geq 0}B_t$ on $\mathbf R$,  the Brownian motion extends to a $C( \mathbf R)$-valued random variable, where $C(\mathbf R)$ is equipped 
with the Borel $\sigma$algebra with respect to the locally convex topology given by the family of seminorms 
$$
\rho_i(f)= \sup_{t\in K_i} |f(t)|,
$$
where  $K_i=[-i,i]$ and $i\in \mathbf N$ (i.e. the topology of  uniform convergence on compact sets). The space $C({\mathbf R})$ is then a locally convex Souslin space, since its topology  is metrizable by  a  complete metric  
$$
d(f_1,f_2)= \sum_{i=1}^\infty  2^{-i} \frac{\rho_i(f_1-f_2)}{1+\rho_i(f_1-f_2)}
$$
and  the polynomials  with rational coefficients form a dense set by the Stone-Weierstrass theorem.
Moreover, $\alpha$ is a $C([0,1])$-valued random variable. 

Recalling Lemma \ref{eka2}, we need to check that the composition mapping 
$(f,g)\mapsto f\circ g+L(x)$ is Borel measurable from  $C(\mathbf R)\times C([0,1])$ into 
$C([0,1])$. Since point evaluations generate the Borel $\sigma$-algebra of $C([0,1])$, it is enough to show that functionals 
$ (f,g)\mapsto  f\circ g(t)+L(x)_t$ are Borel measurable for a fixed $t\in[0,1]$.  We show that this function is actually continuous. 
Since the both spaces are metric spaces, it is enough to check  the  sequential continuity on the product space 
$C(\mathbf R)\times C([0,1])$, which is metrizable.

 Let $\lim_{i\rightarrow \infty }(f_i,g_i)=  (f,g)$ in 
  $C(\mathbf R)\times C([0,1])$, which implies that  $\lim_{i\rightarrow \infty} f_i =f$ and 
$\lim_{i\rightarrow \infty} g_i=g$ in corresponding spaces.  Then  $K=\{g_i(t)\in \mathbf R:  i\in \mathbf N\}$ is compact for the fixed $t\in [0,1]$  and  
\begin{eqnarray*}
|f_i(g_i(t))- f(g(t))| &=&| f_i(g_i(t))-f(g_i(t))+f(g_i(t))-f(g(t))|\\
&<&  \sup_{t\in K} |f_i(t)-f(t)|+  |f(g_i(t))-f(g(t))|\rightarrow 0
\end{eqnarray*}
as $i\rightarrow \infty$ by the convergence of $(f_i,g_i)$ and the continuity of $f$.  
\end{proof}

\begin{theorem}
Let $F$ be a locally convex Souslin topological  vector space equipped with its Borel $\sigma$-algebra  $\mathcal F$ and let $L: F \rightarrow H^1([0,1])$ be a continuous mapping that satisfies 
$L(x)|_{t=0} =0$ for all $x\in F$.  Let  $B_t$ be a Brownian motion on $[0,1]$ starting from zero. Let $\alpha_ t$ be  a strictly increasing stochastic process that is statistically independent from  the Brownian motion $B_t$ and that has  bi-Lipschitz continuous sample paths satisfying $\alpha(0)=0$ almost surely.  Let $X$ be an $F$-valued random variable that is statistically independent 
from the  Brownian motion $B_t$ and  the stochastic process $\alpha_t$.

The essentiaaly unique  solution of     estimating the distribution of $X$ given a sample path $y:[0,1]\rightarrow \mathbf R$   of 
$Y_t=L(X)_t +B_{\alpha_t} $ has a version $\mu$ such that 
$$
\mu(U,y)= \frac{\int_U  \exp \left(  \langle y, C^{-1}_{B_{[y]}}L(x) \rangle -\frac{1}{2} \Vert L(x) \Vert_ {H_{\mu_{B_{[y]}}}}^2 \right) d\mu_X(x)}  {\int  \exp \left( \langle f, C^{-1}_{B_{[y]}} L(x) \rangle -\frac{1}{2} \Vert L(x) \Vert_ {H_{\mu_{B_{[y]}}}}^2 \right) d\mu_X(x) } 
$$ 
for any $U\in \mathcal F$ and for  any $y\in C([0,1])$ such that its quadratic variation $[y]$ satisfies   $0<[y]_t<\infty$   for all  $t\in (0,1]$.
  \end{theorem}
\begin{proof}
Let $g$ be some sample of $\alpha$ on $[0,1]$.  The mapping    
$$
T_g:   f \mapsto f \circ g
$$
is   linear and measurable from   $C(\mathbf R_+)$  to $C([0,1])$. 
 Hence, the Cameron-Martin space of $ T_g B=B_g$ coincides with 
 $T(H^1_0(\mathbf R_+))$ as a  vector space (see Theorems 3.7.3 and 3.7.6  in \cite{bog}; choose $X=C(\mathbf R_+)\times C([0,1])$ in order to generalize the claim to  the present situation).  Since $g$ is bi-Lipschitz continuous, the mapping 
$ f\circ g $ is in $H^1(0,1)$ whenever $f\in H^1 (g(0,1))$ (e.g.Theorem 2.2.2 in \cite{zie}), and the mapping is actually onto the subspace $H=\{f\in H^1(0,1): f(0)=0\}$.    Especially, the vector $L(x)\in H_{\mu_{B_g}}$ by the assumption, so by Lemma \ref{soso}
  $$
   \frac{d\mu_{(B_{\alpha} +L(x),\alpha)}}{ d\mu_{(B_\alpha,\alpha)}} (z,\alpha)= 
    \exp\left( \langle z, C_\alpha^{-1} L(x)\rangle - \frac{1}{2}\Vert  L(x)\Vert^2_{H_{\mu_{B_{\alpha}}}} \right)   
    $$
as in Theorem \ref{subi}.  It is well-known that for continuous time-changes $\alpha_t$ the quadratic variation of  $B_{\alpha_t}$ coincides with $\alpha_t$ (see Chapter 5:  Proposition 1.5    in  \cite{revuz}).  Therefore, 
$\alpha_t$ is a measurable function of  the sample pathsof $B_{\alpha_t}$ (the quadratic variation is obtained by taking a  limit in probability and we need to pick up  a subsequence in order to get the  a.s. convergence). 
  Since $L(X)\in H^1(0.1)$,  it has finite variation, which implies that its quadratic variation vanishes. 
 Also $\mu_{Y}$-almost every sample path of $ L(X)_t+B_{\alpha_t}$ has $\alpha_t$ as its 
 quadratic variation. We obtain the claim similarly as in   Lemma \ref{aver}. 
\end{proof}

Posterior convergence holds similarly as in the Gaussian case. The assumptions that guarantee   the continuity of the solution are not known.

\subsection{Decomposable additive noise}\label{sparsec}

Let  $F$ and $G$ be  locally convex Souslin topological  vector spaces.   We say that  $G$-valued random noise $\varepsilon$ is {\it decomposable} if it is  of  the form 
$$
\varepsilon= \sum_{i=1}^\infty \varepsilon_i f_i,
$$ where  $\varepsilon _i$ are independent random variables with  a.e. positive probability density functions $\rho_i$ with respect to the Lebesgue  measure  and $f_i\in G$ are some non-zero vectors. 
\begin{remark}
If $\varepsilon_i$ are random variables and 
$\varepsilon:= \sum_{i=1}^\infty \varepsilon_i f_i$  a.s. for some vectors  $f_i\in G$, then 
$\varepsilon$ is a $G$-valued random variable. Indeed, since $G$ is  a Souslin topological  vector space, the mapping 
$\mathbf R \times G \ni (a,f)\mapsto af =:T(a,f)$ is continuous, therefore also $\mathcal B(R\times G)= \mathcal B(R)\otimes \mathcal G$ measurable.  The composition of the measurable mapping $(\omega,f)\mapsto (\varepsilon _i (\omega),f)$ with  $T$ gives a  $G$-valued random variable  
  $T(\varepsilon _i, f)= \varepsilon _i f$.  Also the sum of two $G$-valued random variables is  a $G$-valued random variable  and   limits of locally convex Souslin space-valued  random  variables are   random variables (since the cylinder sets generate the Borel $\sigma$-algebra by Theorem 6.8.9 in \cite{bogm}).
  \end{remark}

If all possible signals $L(x)$, $x\in F$ are sparse in the sense that they belong to the linear span of  $\{f_i: i\in \mathbf N\}$ and  the noise $\varepsilon$ is decomposable, then the measures  $\mu_{\varepsilon+L(x)}$ are absolutely continuous with respect to $\mu_\varepsilon$ \cite{sato}. 

 Moreover,  if   $\{f_i\}_{i=1}^\infty$  is a basis of 
the closed subspace $\overline{{\rm span}(\{ f_i: i\in N\})}$,   the proof in \cite{sato}  gives, with minor additional work,  an explicit  formula for the Radon-Nikodym density.  For simplicity, we take $G=\overline{{\rm span}(\{ f_i: i\in N\})}$.

\begin{theorem}\label{spar}
Let  $G$ be a locally convex   Souslin topological vector    space  equipped with the Borel $\sigma$-algebra $\mathcal G$ and  a  basis $\{ f_i \}_{i=1}^\infty$ such that  the unique coefficients $y_i$  in  $y= \sum_{i=1}^\infty y_i f_i$ depend  measurably  on $y\in G$.  Let  a  
$G$-valued random variable $\varepsilon$ be of the form $\varepsilon = \sum_{i=1}^\infty \varepsilon_i f_i$, where  the random variables $\varepsilon_i$ are statistically independent and have
probability density functions $\rho_i$ that are a.e. positive.  If $L_n(x)=\sum_{i=1}^n  a_i(x) f_{k_i}$, then 
\begin{equation}\label{sparse}
\frac{d\mu_{\varepsilon+L_n(x)}}{d\mu_\varepsilon}(y)=\frac{\prod_{i=1}^{n} \rho_{k_i}(y_{k_i}-a_i(x)) }{\prod_{i=1}^{n} \rho_{k_i}(y_{k_i})}
\end{equation}
for  almost every  $y=\sum_{i=1}^\infty y_i f_i$.  
\end{theorem}
\begin{proof}
Let $A\in \mathcal G$. By possibly rearranging  finitely many vectors, we may suppose that  $L_n(x)= \sum_{i=1}^n a_i f_i$. We consider the probability
\begin{eqnarray*}
\mu_{\varepsilon+L_n(x)} (A) = P(\varepsilon+L_n(x)\in A)
= \mathbf E \left[ 1_A\left ( \sum_{i=1}^\infty \varepsilon_i f_{i} + \sum_{i=1}^n a_i f_i  \right)\right].
\end{eqnarray*}
Denote $Z=\sum_{i=n+1}^\infty \varepsilon _i f_i$. 
Following \cite{sato}, we calculate the conditional expectation of $1_A(\varepsilon+L(X))$ given 
$\sum_{i=1}^n   (\varepsilon_i+ a_i ) f_i  $ and, by  Lemma \ref{eka2},  obtain with straightforward calculations
\begin{eqnarray*}
\mathbf E\left [ 1_A\left( \sum_{i=1}^\infty \varepsilon_i f_{i} + \sum_{i=1}^n a_i f_i  \right)\right]
&=& \mathbf E \left[ \mathbf E \left[1_A\left( Z + \sum_{i=1}^n (\varepsilon_i+a_i) f_i  \right) |\sum_{i=1}^n   (\varepsilon_i+ a_i ) f_i   \right]\right]\\
&=&
\int     \mu_{ Z+\sum_{i=1}^n (\varepsilon_i(\omega)+a_i) f_i  } \left (A \right) dP(\omega)   \\
&=&   \int_{\mathbf R^n}   \mu_{Z+\sum_{i=1}^n (y_i+a_i) f_i  } \left (A \right) 
 \left(  \prod_{i=1}^n \rho_i(y_i)  \right)dy_1\cdots dy_ n\\
&=&   \int_{\mathbf R^n}   \mu_{Z+\sum_{i=1}^n  y_i f_i  } \left (A \right) 
 \left(  \prod_{i=1}^n \rho_i(y_i-a_i)  \right)dy_1\cdots dy_ n.
 \end{eqnarray*}
At this point, the proof   differs from \cite{sato}. Namely, we multiply and divide with
the positive densities of $\varepsilon_i$, and obtain 
   \begin{eqnarray*}
\mathbf E\left [ 1_A\left( \sum_{i=1}^\infty \varepsilon_i f_{i} + \sum_{i=1}^n a_i f_i  \right)\right]
 &=&   \int    \mu_{ Z+ \sum_{i=1}^n y_i f_i  }  (A)    \prod_{i=1}^n  \frac{\rho_i(y_i-a_i)}{ \rho_i( y_i)}   \rho_i(y_i)dy_1\cdots dy_n\\
 &=&   \int    \mu_{ Z+ \sum_{i=1}^n \varepsilon_i (\omega) f_i  }  (A)    \prod_{i=1}^n  \frac{\rho_i(\varepsilon_i(\omega)-a_i)}{ \rho_i( \varepsilon_i (\omega))} dP(\omega)\\
&=&   \mathbf E \left[  \mathbf  E \left[ 1_A\left(\sum_{i=1}^\infty \varepsilon_i f_i \right) | \sum_{i=1}^n \varepsilon_i f_i    \right]
 \prod_{i=1}^n  \frac{\rho_i(\varepsilon_i-a_i)}{ \rho_i( \varepsilon_i)} \right ] \\
 &=&  \mathbf E \left[  1_A  \left (\varepsilon \right)    \prod_{i=1}^n  \frac{\rho_i(\varepsilon_i-a_i)}{ \rho_i( \varepsilon_i)} \right ]. 
\end{eqnarray*}
 Since the unique coefficients $\varepsilon_i$  depend  measurably on  $\varepsilon $, we may write 
\begin{eqnarray*}
\mu_{\varepsilon+L_n(x)} (A) &=&  \mathbf E \left[ 1_A(\varepsilon)  \prod_{i=1}^n \frac{ \rho_i (\varepsilon_i-a_i)}{\rho_i(\varepsilon_i) }\right]\\
&=&\int 1_A(y)\prod_{i=1}^n \frac{ \rho_i (y_i-a_i)}{\rho_i(y_i) } d\mu_{\varepsilon} (y). 
\end{eqnarray*}
\end{proof}

The above theorem verifies the intuitive  picture that for sparse signals  we may as well study the posterior  of $X$ given  the finite-dimensional data 
$$
Y_n=\sum_{i=1}^n  \left( a_i(X) f_{k_i} + \varepsilon_i f_{k_i}\right).
$$

The following theorem gives a significant enlargement of applicable noise models in statistical inverse problems.

   \begin{corollary}[Generalized Cameron-Martin formula]\label{cama}
   Let  $G$ be a locally convex   Souslin topological vector    space  equipped with the Borel $\sigma$-algebra $\mathcal G$ and  a  basis $\{ f_i \}_{i=1}^\infty$ such that  the unique coefficients $y_i$  in  $y= \sum_{i=1}^\infty y_i f_i$ depend  measurably  on $y\in G$.  Let  a  
$G$-valued random variable $\varepsilon$ be of the form $\varepsilon = \sum_{i=1}^\infty \varepsilon_i f_i$, where  the random variables $\varepsilon_i$ are statistically independent and have
probability density functions $\rho_i$ that are a.e. positive.  If $L(x)=\sum_{i=1}^\infty  a_i(x) f_{i}$ for all $x\in F$, and densities  
\begin{equation*}
\frac{d\mu_{\varepsilon+L_n (x)}}{d\mu_\varepsilon}(y)=\frac{\prod_{i=1}^{n} \rho_{i}(y_{i}-a_i(x)) }{\prod_{i=1}^{n} \rho_{i}(y_{i})}
\end{equation*}
are uniformly integrable with respect to $\mu_{\varepsilon}$ and convergent  $\mu_{\varepsilon}$-almost everywhere, then 
\begin{equation*}
\frac{d\mu_{\varepsilon+L (x)}}{d\mu_\varepsilon}(y)=\prod_{i=1}^{\infty} \frac{ \rho_{i}(y_{i}-a_i(x)) }{\rho_{i}(y_{i})}
\end{equation*}
for  $\mu_\varepsilon$-almost every  $y=\sum_{i=1}^\infty y_i f_i$.  
\end{corollary}
\begin{proof}
See Proposition 9.9.10 in \cite{bogm}, which says that    if  $\lim_ n T_n =T$, where $T_n$ and $T$ are measurable mappings on a completely regular space,  and  the   distributions of  all   $T_n$    have uniformly integrable Radon-Nikodym densities $\rho_n$ with respect to 
the  same Radon  probability  measure $\nu$,    then the distribution of  $T$   has Radon-Nikodym density $\rho$  with respect to the same Radon probability measure as well, and  $\rho$ is  the limit of $\rho_n$ in the weak topology of 
 $L^1(\nu)$.    This result is especially applicable to the  random variables $ T_n (x,z)=  L_n (x)+ z$ and $T(x,z)=
   L(x)+z$  on $(F\times G, \mathcal F\otimes \mathcal G, \mu_{X}\otimes \mu_{\varepsilon})$ and the measure $\nu=\mu_{\varepsilon}$.    The  integrals of the densities  over   any  Borel set converge. By Theorem 4.5.6 and Corollary 4.5.7 in \cite{bogm} the weak limit coincides with the almost sure limit.
\end{proof}
In Corollary \ref{cama},   the Radon-Nikodym density  $\frac{d\mu_{\varepsilon+L (x)}}{d\mu_\varepsilon}(y)$   has  a form similar  to  Radon-Nikodym densities appearing  in  the Kakutani dichotomy theorem, which addresses the  equivalence  and singularity of  infinite product measures on $\mathbf R^\infty$  \cite{kakutani}.   Also Umemura \cite{umemura}  has given   conditions  for the    absolute continuity of  measures  on abstract spaces  when the corresponding finite-dimensional distributions are absolutely continuous.  In our case,   Umemura's conditions ask     $\left(\frac{d\mu_{\sum_{i=1}^n (\varepsilon_i+ L(x)_i)}}{d\mu_{ \sum_{i=1}^n \varepsilon_i}} (y)\right)^\frac{1}{2}$ to be  a Cauchy sequence in $L^2(\mu_{\varepsilon})$. 
 We feel that  the uniform integrability of the Radon-Nikodym densities is easier to validate than Umemura's conditions.

  \begin{corollary}
  Let the assumptions of Corollary \ref{cama} hold.  The essentially unique posterior distribution of $X$ given a sample $y$ of 
  $Y=L(X)+\varepsilon$  has a version $\mu(\cdot, y)$  such that 
  $$
  \mu(U,y)= \frac{\int _U  \prod_{i=1}^{\infty}  \frac{\rho_{i}(y_{i}-L(x)_i) }{ \rho_{i}(y_{i})}  d\mu_X(x)}
{\int  \prod_{i=1}^{\infty}  \frac{  \rho_{i}(y_{i}-L(x)_i(x)) }{ \rho_{i}(y_{i})}  d\mu_X(x)}
  $$
  whenever $0<\int  \prod_{i=1}^{\infty}  \frac{  \rho_{i}(y_{i}-L(x)_i) }{ \rho_{i}(y_{i})}  d\mu_X(x)< \infty$. 
  \end{corollary}

\begin{remark}\label{model}
Signals in non-Gaussian noise may appear in model approximations.
Let the true model be  
$$
Y= L(X)+\varepsilon
$$
where $\varepsilon=\sum_{i=1}^\infty  \varepsilon_i f_i$  and  all $\varepsilon_i$ are statistically independent.  When the model  $L$ is numerically very complicated,   the common practice is to replace $L$ with some simpler approximation $ L_n$.  For example,  $L_n$ may have the form $ L_n(X) =\sum_{i=1}^n   a_i(X) f_i$.
Though the true model $L $ and the approximated   model $L_n$ are known,  the model error $\widetilde  \varepsilon =   L(X)-  L_n(X)$   is sometimes  replaced with  a $G$-valued random variable  $\varepsilon'$ that has the same distribution as   $\widetilde \varepsilon$ but is statistically independent from $X$  \cite{tanja}. 
We note that the  observation model is then 
 $$
 Y=L_n(X)  +\varepsilon'+ \varepsilon,
 $$
where $\varepsilon$ represents the uncertainties in the forward model $ L_n$.  Beside of physical noise, the distribution of the noise may  represent  our prior beliefs about the  uncertainties in the forward model, which  do not  necessarily
have  Gaussian distributions. 
\end{remark}

\subsection{Periodic signals in  decomposable Laplace noise}\label{peri}

In this section, we study an example case of the generalized Cameron-Martin formula for a non-Gaussian noise distribution.   
A similar distribution  has been constructed before by Shimomura \cite{shimomura}  who  gave conditions under which certain translates of the  distribution were equivalent to the original distribution. However, we use the methods  of Section 
\ref{sparsec}.

One  class of inverse problems  that involves periodic signals  are the inverse  scattering problems --    the far-field  pattern of  the scattered wave in the 2D fixed energy  inverse  acoustic or potential scattering problem  is a  function on the  torus. The  measured far-field pattern is possibly  contaminated by  instrumental noise,  far-fields of  other unknown incoming fields, contributions from other scatterers,   and the near-field  and plane wave approximation errors.     Although the random model below is  oversimplified to fully cover this case,  it  shows how  periodicity can be utilized in Bayesian inverse problems. 
 
 Suppose that $L(x) \in C ^\alpha (S^{1)}$  for all $x\in F$ and some $\alpha>1$. Then 
the  Fourier coefficients 
$$
\widehat{L(x)}_k =  \frac{1}{2\pi}\int_0^{2\pi}  L(x;t) e^{-i k t} dt
$$
 are 
 $\ell_1$-summable and the corresponding Fourier series converges to the limit
 $$
 L(x;t)= \sum_{k=-\infty}^\infty \widehat{L(x)}_k e^{i k t}
 $$ 
 in $C(S^1)$ (equipped with the usual supremum norm).
 
   Let   $\varepsilon_k$ be  mutually statistically  independent random variables 
   whose  probability density functions with respect to the Lebesgue measure are
 $$
 \rho_k (t)= \frac{}{2b}e^{-|t|/b}
 $$
 for all $k\in\mathbf Z$ and some common $b>0$ i.e. they are  zero mean  Laplace random variables.  The relation of the normal distribution to  the Laplace distribution  is that a conditionally normal random variable     $\tilde \varepsilon_k | \sigma  \sim N(0,\sigma^2)$ with  a  Rayleigh distributed  variance  has a Laplace distribution. In statistical inverse problems,  one  interpretation of the Laplace distribution is that we do not know  the  error variance exactly and  are lead to describe  our lack of knowledge in the form of a probability distribution. 
 
 \begin{lemma}\label{spar1}
  The random sum
  $$
  \varepsilon = \sum_{k=-\infty}^\infty   \varepsilon_k e^{ik  t}
  $$
 converges in $H^{-1} (S^1)$ and 
 $$
 \langle \varepsilon, e^{-ikt}\rangle _{H^{-1}(S^1), H^1(S^1)} =  \varepsilon_k
 $$ 
 \end{lemma}
 \begin{proof}
 The space $H^{-1}(S^1)$ is a Hilbert space,  and the random variables $\varepsilon _k$ have zero mean so it suffices to 
 prove that $ \sum_{k}  E[  \Vert \varepsilon _k e^{ik t }\Vert_ {H^{-1}(S^1)}^2 ]<\infty$  (see Theorem  2 in Chapter  3.2 in
 \cite{kahane}).   The sequence $\{ e^{ik t}\}_{k=-\infty}^\infty$ forms an orthonormal 
basis of $L^2(S^1)$. The imbedding  of $L^2(S^1)$ into $H^{-1}(S^1)$ is 
Hilbert-Schmidt by Maurin's theorem, which implies that
$$
 \sum_{k}  E[  \Vert  e^{ik t }\Vert_ {H^{-1}(S^1)}^2 ]< \infty
$$
and therefore
$$
\mathbf E   \sum_{k=-\infty}^\infty | \varepsilon_k|^2   \Vert e^{ik t}\Vert_{H^{-1}(S^1)}^2  \leq  C \sum_{k} \Vert e^{ik t}\Vert_{H^{-1}}^2 <\infty . 
$$
Here we used the fact that the variance of  the  Laplace random variable $\varepsilon_i$ is  $2b^2$.
 \end{proof}
 
 Next, we quickly check that $\varepsilon$ is a non-Gaussian  random variable.
 The characteristic function of the random variable 
 $\varepsilon$  is
 $$
\widehat\mu_{\varepsilon}(\phi) = \mathbf E[e^{i \langle \varepsilon, \phi\rangle_{H^{-1}(S^1), H^1(S^1)}}] = 
 \prod_{k=-\infty}^\infty  \frac{1}{1+ b^{2}({\widehat \phi}_{-k})^2},
 $$
 where $\widehat{\phi}_k $ is  the  Fourier coefficient  $\frac{1}{2\pi}\int_0^{2\pi}  \phi(t) e^{-i k t} dt$  of $\phi\in H^1(S^1)$.
 Note, that when $\phi_j \rightarrow \psi \in L^2(S^1)$ as $j\rightarrow \infty$, then 
 $$
 \lim_{j\rightarrow \infty} \widehat\mu_{\varepsilon}(  \phi_j ) =  \prod_{k=-\infty}^\infty  \frac{1}{1+ b^{2} ({\widehat \psi}_k)^2}=:\widehat\mu_{\varepsilon}(\psi )
 $$
 i.e. their distributions converge weakly.  Especially, when $\widehat{\psi}_k=\frac{1}{\pi(|k| -\frac{1}{2})}$, $k\not=0$ and 
 $\widehat \psi_{0}=0$, then 
  $$
  \widehat\mu_\varepsilon(t\psi)= \frac{1}{\cosh^2( bt)}.
 $$ 
 Since the weak limits of  zero mean Gaussian distributions are always zero mean Gaussian distributions, this shows that $\varepsilon$ is indeed a  non-Gaussian random variable.

 We wish to study the statistical inverse problem of estimating the probability distribution of 
 $X$ when  a sample of 
 $$
 Y= L(X)+\varepsilon 
 $$
 is known. This means that the inexact observations   of the Fourier coefficients of $L(X)$ are assumed to be similarly  inaccurate and  some components are allowed to have  high inaccuracies. Since Laplace distribution has heavier tails than the Gaussian distribution,  it protects against outliers better than the normal distribution.

  Consider first finite sums 
 $$
 L_n (x;t)=  \sum_{|k|\leq n} \widehat{L(x)}_k e^{i k t}.  
 $$
By  Lemma \ref{spar1} and Theorem \ref{spar}, the Radon-Nikodym derivative of the  translated measure $\mu_{\varepsilon+L_n(x)}$  with respect to 
the measure $\mu_{\varepsilon}$ on $H^{-1}(S^1)$ is
 \begin{eqnarray*}
 \frac{d\mu_{\varepsilon +L_n(x)}}{d\mu_\varepsilon} (y)&=&\frac{\prod_{k=-n}^n \rho_{k}(\widehat y_{k}- \widehat {L(x)}_k) }{\prod   _{k=-n}^{n} \rho_{k}(\widehat y_k)}\\
 &=& \prod_{k=-n}^n  e^ {- b^{-1} |\widehat y_{k}- \widehat {L(x)}_k) | +b^{-1}  | \widehat y_k|}.
 \end{eqnarray*}
  By the triangle inequality, we obtain that
$$
||\widehat y_k|- |\widehat  y_k -\widehat {L(x)}_k|| \leq |\widehat{  L(x)}_k|, 
$$
which are summable.   Therefore, the limit 
 $$
  \sum_{k= -\infty }^\infty  (  |\widehat y_k | - |\widehat y_{k}- \widehat {L(x)}_k |)  
 $$
 exists. 

 Random variables  $\varepsilon+L_n(x)$ converge almost surely to $\varepsilon+L(x)$. Therefore, 
 corresponding measures converge weakly i.e. for all Borel sets  $A$ whose boundary  is $\mu_{\varepsilon+L(x)}$- zero measurable
 it holds that
 \begin{eqnarray*}
\mu_{\varepsilon+L(x)}(A)&=& \lim_{n\rightarrow \infty } \mu_{\varepsilon+L_n(x)} (A) \\
&=&  \lim_{n\rightarrow \infty }\int_ A   e^{  b^{-1} \sum_{k= 1}^n  (  |\widehat y_k | - |\widehat y_{k}- \widehat {L(x)}_k|) } d\mu_\varepsilon(y)\\ 
&=& \int_ { A}    e^{ b^{-1}  \sum_{k= -\infty}^\infty  (  |\widehat y_k | - |\widehat y_{k}- \widehat {L(x)}_k |) } d\mu_\varepsilon(y)
\end{eqnarray*}
by  the Lebesgue dominated convergence theorem. The  exponential  function is 
the Radon-Nikodym derivative  $\frac{d\mu_{\varepsilon +L(x)}}{d\mu_\varepsilon} (y) $. 

We have shown  the following theorem.
\begin{theorem}
Let $F$ be a locally convex Souslin topological  vector space equipped with its Borel $\sigma$-algebra.  
Let $X$ be an $F$-valued random variable and let $L:F \rightarrow  C ^\alpha (S^{1})$ be a continuous 
mapping  for  some $\alpha>1$.  Let  
$$
\varepsilon= \sum_{k} \varepsilon_k e^{ik t}
$$
be  a $H^{-1}(S^1)$-valued random variable such that  all $\varepsilon_k$, where $k\in \mathbf Z$, are  mutually statistically  independent random variables  with probability density functions 
 $$
 \rho_k (t)= \frac{1}{2b}e^{-|t|/b}
 $$
with respect to the Lebesgue   measure     for some $b>0$.

The  solution of the statistical inverse problem of estimating the distribution  of $X$ given  a sample $y \in H^{-1}(S^1)$ of $Y=L(X)+\varepsilon$ is essentially unique and has a version  $\mu$ such that
$$
\mu (U,y)= \frac{\int_ U e^{ b^{-1} \sum_{k= -\infty}^\infty  (  |\widehat y_k | - |\widehat y_{k}- \widehat {L(x)}_k |) } d\mu_X(x)  }{
 \int_ U e^{b^{-1} \sum_{k= -\infty}^\infty  (  |\widehat y_k | - |\widehat y_{k}- \widehat {L(x)}_k |) } d\mu_X(x) }
$$
for all $U\in \mathcal F$ and for all  $y\in H^{-1}(S^1)$ such that the denominator is finite and non-zero.
\end{theorem}

\begin{remark}
The convergence of posterior distributions holds for example if for all $n\in \mathbf N$ 
it holds that
$
\mathbf E [ e^{b^{-1}\Vert  \widehat { L(X_n)} \Vert _{\ell^1}}] \leq C
$
 and 
$
\mathbf E [ e^{b^{-1}\Vert  \widehat{  L(X) }\Vert _{\ell^1}}] \leq C
$
 for  some  $C>0$. Under the same conditions, the posterior distributions depend 
 continuously on the observations. Indeed, the posterior distributions are 
 sequentially continuous by the Lebesgue dominated convergence theorem and  the continuity of $y\mapsto \hat y_k$ , and 
  sequentially continuous functions on Hilbert spaces are continuous.  The topological support of 
  $  \mu_Y$ coincides with the topological support of $\mu_{\varepsilon}$ by Lemma \ref{konti}. The topological support of $\mu_\varepsilon$ coincides with the closure of  the linear span of $\{e^{ikt}\}_{k}$ in $H^{-1}(S^1)$, that is,    $H^{-1}(S^1)$ (see \cite{sato}). By Theorem \ref{uniq},  $\mu$ is the only posterior distribution that depends continuously on the observations. 
\end{remark}

\section{Examples of  prior approximations}\label{sec6}

We present some  methods  for approximating the  unknown.  We take  the unknown $X$ always to be statistically independent from the noise  $\varepsilon $. Especially, when  the Radon-Nikodym   density $\rho(x,z)=\frac{d\mu_{\varepsilon+L(x)}}{d\nu}(z)$ is bounded and continuos, we do not need to ask anything special on $X$  or its converging  approximations  in order to obtain   posterior convergence.

\subsection{Random series}  

As discussed in Section \ref{review}, random  series and wavelet expansions   are important devices in  defining  infinite-dimensional prior models.  If a random variable $X$ in a locally convex Souslin topological vector space  $F$  can be expressed as  an almost surely converging series  $X=\sum_{i=1}^\infty Z_i \phi_i$, where $\phi_i \in F$ and $Z_i$ are ordinary random variables,  we   obtain immediately   finite-dimensional approximations  $X_n:= \sum_{i=1}^n Z_i \phi_i$ by truncation.  The almost sure  convergence of random variables $X_n$  to $X$ implies (by the Lebesgue dominated convergence theorem)  that the distributions $\mu_{X_n}$  converge weakly to  $\mu_X$.    We return to  this topic  in connection with the 
 linear discretizations of $X$ in  Section \ref{rs}. 

\subsection{Gaussian priors}
\subsubsection{Gaussian random series}
All  Gaussian  $F$-valued random variables can be expressed with random series expansions \cite{bog}.  A typical example is the Karhunen-Lo\`eve expansion   of  a zero mean $L^2(a,b)$-valued Gaussian random variables $X$,  where $\phi_i$ are chosen to be normed  eigenfunctions of the covariance operator 
$$
C f(t)= \int_a^b \mathbf E [X_t X_s] f(s)ds
$$ 
on $L^2(a,b)$. Then  $X$ can be expressed as $X=\sum_{i=1}^\infty Z_i e_i$, where  the standard normal random variables $Z_i$ are statistically independent, and the truncated series $X_n$ gives an almost surely converging   approximation of $X$. 
 The almost sure convergence of $X_n$ to $X$ implies the  weak convergence of $\mu_{X_n}$ to $\mu_X$.   

\subsubsection{Converging covariances} 
 
 A well-known sufficient condition for  the weak convergence of probability measures $ m_{n}$ to a probability measure $m$ on  a locally convex  Souslin topological vector space $F$   is that 
  \begin{itemize}
\item[$(i)$]{the measures $m_{n}$ are uniformly tight i.e. for every $\varepsilon >0$ there exist a compact set $K_\varepsilon  \subset F$ such that $\sup_n m_{n}(K_\varepsilon ^C  )<\varepsilon $.}  
\item[$(ii)$]{Characteristic functionals $\widehat m_{n} (\phi)$ converge to 
 $\widehat m(\phi )$ for every $\phi \in F'$.}
  \end{itemize} 
According to  Prohorov's theorem,   $(i)$ implies  that  
each subsequence of $m_n$ has a weakly convergent subsequence (see   Theorem 8.6.7 in \cite{bogm}).  Part $(ii)$ identifies the limits of different subsequences.   For some  spaces,    $(i)$ can be deduced from  $(ii)$ and the known  properties of $\widehat m$ (see   \cite{vakha}).  These spaces include $\mathbf R^n$ and   the distribution spaces $\mathcal S'(\mathbf R^d)$ and $\mathcal D'(U)$, $U\subset \mathbf R^d$ open.

Let $\mu_{n}$ and $\mu$ be Gaussian zero mean measures on the distribution space  $\mathcal D'(U)$, where   $U\subset \mathbf R^n$ is open. If 
the covariance operators $C_n$ of $\mu_n$ converge weakly to the covariance operator $C$ of $\mu$, i.e. 
$$
\lim _{n\rightarrow \infty }\langle C_n \phi,\psi \rangle =\langle C \phi,\psi \rangle 
$$
for every $\phi,\psi \in \mathcal D(U)$, then $(i)$ holds. Moreover,   the characteristic functionals
 $\widehat \mu_n$ are then equicontinuous at zero, which 
 is sufficient for the uniform tightness of the sequence  $\mu_n$ by Corollary 7.13.10 in \cite{bogm}.

\subsubsection{Martingale approximations}
      Another possibility  is to use   a special  martingale approximation of  the unknown $X$.   We discretize a separable Hilbert space-valued  $X$ with finite-dimensional  increasing orthogonal projections $P_n$  on the Cameron-Martin space  $H$ of $\mu_X$  by setting       $$
    X_n= \sum_{j=1}^n  \langle X, f_j\rangle f_j, 
    $$
     where  the vectors $f_j$, $ j=1,...,n$ form an orthonormal basis of the finite-dimensional  subspace  $P_n(H)$ and 
     $\cup_ n P_n(H)$ is dense in $H$.  Such  martingale approximations where introduced to statistical inverse problems  in \cite{sari}. Then 
     $$
\mathbf E[     \langle X_n,\phi\rangle  | X_m]=   \langle X_m,\phi\rangle
     $$
     when $m\leq n$ and $\phi\in F'$.  Hence  $\langle X_n,\phi\rangle$ is a martingale. This makes  $\Vert X-X_n\Vert_F$  a reversed submartingale. It is integrable, since 
     $$
    \mathbf E  [ \Vert X-X_n\Vert^2_F ] \leq  
  \sum_{i=1}^\infty  \Vert (I-P_n) e_i\Vert^2_F=  \Vert I-P_n\Vert_{HS}^2, 
     $$
where $\Vert I-P_n \Vert _{HS}$ denotes the Hilbert-Schmidt norm of   $I-P_n:H\rightarrow F$,      and has limit  0 in $L^2(P)$. Hence, $X-X_n$ converge a.s. to zero (see Theorem 10.6.4 in \cite{dud}).
  The almost sure convergence of  $X_n$ to $X$ implies  the weak convergence of $\mu_{X_n}$ to $\mu$.

\subsection{Mappings of Gaussian variables}

We consider non-linear functions of  continuous Gaussian processes as prior models.   We start by deforming a  Brownian motion $B_t,  t\in  [0,1]$,  with a continuous function  $f:\mathbf  R\rightarrow \mathbf R$ by setting
$$
X_t= f(B_t).
$$

 We check that $X=X_t$ is a random variable having values in a suitable  function space $F$.    Obviously, each $X_t, t\in [0,1]$ is a random variable. The process  $X_t$   also inherits sample-continuity from  Brownian motion. A natural choice for   $F$ is the space $C([0,1])$ of  continuous functions on  compact interval $[0,1]$, equipped with the supremum  norm  $\Vert f\Vert_\infty= \sup_{t\in [0,1]} |f(t)|$. Sample-continuous stochastic processes on $[0,1]$ are  $C([0,1])$-valued random variables, since  the Borel $\sigma$- algebra $\mathcal F$ of $F$ coincides with the smallest $\sigma$-algebra generated by Dirac's delta functions $\delta_t, t\in [0,1]$, which are continuous linear forms on $C([0,1])$ (see  Proposition 12.2.2 in  \cite{dud}).

 The process $X_t$ can be  discretized by replacing the Brownian motion with its piecewise linear interpolation 
 $$
b_n(t)=\sum_{i=1}^n  B_{t_i}  \phi_i(t)
$$
on $[0,T]$.  Here  the functions $\phi_i$ are the usual linear interpolation functions.  The mapping $f\mapsto f_n$ is measurable on $C([0,1])$, and $\lim_{n\rightarrow \infty} f_n = f $ in $C([0,1])$.  Especially, the  approximations  $b_n(t) $ converge a.s. to $B_t$ in $C([0,T])$ (the sample path of $B_t$ may first  be  approximated by a  $C^2$-function). Then  $X_n(t)=f(b_n(t))$ converge almost surely to $X_t=f(B_t)$ as $n\rightarrow \infty $ due to  the continuity  of $f$. Almost sure convergence implies the weak convergence of  the corresponding image measures.

Examples:  1) Take $f(t)=t^2$,i.e. $X_t=B_t^2$ and  $X_n(t)=b_n(t)^2$.  The positive continuous functions form a  measurable set in $C([0,T])$  that  has full measure in this case.   2) Take $f(t)=\min(t,1)$. Then we obtain the approximation 
$X_n(t)=\min(b_n(t),1)$ of the bounded function $X_t=\min(B_t,1)$.

The Brownian motion may be replaced with any other stochastic 
process whose sample paths are continuous.

\subsection{Stochastic integrals}

We consider now prior models defined with stochastic integrals 
$$
X(t)=\int_0^t  f(s,\omega)dB_s,  \; t\in \mathbf [0,T]
$$
where $B_s$ is a Brownian motion and $f:[0,T]\times \Omega \rightarrow \mathbf R$ is   in the class $\mathcal V([0,T])$ that   satisfies the following conditions. A function $f\in \mathcal V([0,T])$ is   $\mathcal B([0,T]\times \mathcal F)$-measurable,  $f(t,\omega)$ is $\mathcal F_t$-adapted and $\mathbf E[\int_0^t f(t,\omega)^2dt ]<\infty$.  Here $\mathcal F_t$ is the $\sigma$- algebra generated by  all  $B_s$, $s\leq t$  (see   \cite{oek}). Furthermore, we assume that $f$ satisfies
$$
\lim_{n\rightarrow \infty }\mathbf E \left [\int_0^{T}\left\vert  \sum_{j=1}^n f(t_{j-1}^{(n)},\omega) 1_{(t_{j-1}^{(n)},t_{j}^{(n)}]} (t) - f(t,\omega)\right\vert ^2 dt\right]=0, 
$$
where  $0=t_0^{(n)}<t_1^{(n)}<...<t_{n}^{(n)}=T$  are such that 
$\max_{i} (t_i^{(n)}-t_{i-1}^{(n)})\rightarrow 0$ as $n\rightarrow \infty$.

 With probability one, $X(t)$ has continuous sample paths (see   \cite{oek}).  As in the previous section, we may interpret $X$ as $C([0,T])$-valued random variable. One discrete approximation is to take
 $$
 X_n(t)= \sum_{i=0}^{n} X_n(t_i^{(n)}) \phi_i(t),
 $$
 where the functions  $\phi_i$ are the linear interpolation functions,  and 
 \begin{eqnarray*}
X_n (t_i^{(n)}; \omega) &=&  \sum_{j=1}^i  f(\omega,t_{j-1}^{(n)}) (B_{t_j^{(n)}}(\omega)-B_{t_{j-1}^{(n)}}(\omega))
\\
&=& 
\int_0^{t_i^{(n)}}  \sum_{j=1}^n f(\omega,t_{j-1}^{(n)}) 1_{[t_{j-1}^{(n)},t_j^{(n)})}(t) dB_t(\omega) 
\end{eqnarray*}
 are  approximations of the stochastic integrals $X(t_i^{(n)})$ for $i\geq 1$ and $X_n(0)=0$. Then $X_n$ has a subsequence that converges  to $X$ on 
 $C([0,T])$. Indeed,
 $$
\sup_t | X(t) -X_{n} (t) |  \leq  \Vert  X (t) - \sum_{i=1}^{n} X(t^{(n)}_i) \phi_i (t) \Vert_\infty   +    \sup_{1\leq i \leq n}  |X(t_i^{(n)})-X_{n}(t_i^{(n)})|,  
  $$
where, by  Doob's  inequality  and the  It\=o isometry   the latter term satisfies  
\begin{eqnarray*}
\mathbf E [ \sup_{1\leq i\leq n} |X_n(t_i^{(n)})-X(t_i^{(n)})|^2]  &\leq  &C   \sup_{1\leq i\leq n } \mathbf E [|X_n(t_i^{(n)})-X(t_i^{(n)})|^2]
\\
 & \leq & C \mathbf  E \left [\int_0^{T}\left| \sum_{j=1}^n f(t_{j-1}^{(n)},\omega) 1_{(t_{j-1}^{(n)},t_{j}^{(n)}]} (t) - f(t,\omega)\right|^2 dt\right], 
\end{eqnarray*}
which converge to zero as $n\rightarrow \infty$.    Especially,  the sequence $\{ \sup_{1\leq i\leq n}|X_n(t_i^{(n)})-X(t_i^{(n)})|\}_{n}$  converges in $L^2(P)$ and has therefore an a.s. convergent  subsequence $\{   |X_{n_j}(t_i^{(n_j)})-X(t_i^{(n_j)})|\}_{j}$. 
Therefore, $\mu_{X_{n_j}}$ converges weakly to $\mu_X$.  Then  any subsequence of measures $\mu_{X_n}$ has   a  weakly converging subsequence with  the same limit $\mu_X$, and  the measures   $\mu_{X_n}$ converge weakly to $\mu_{X}$ on $C([0,T])$.

\subsection{Hyperparametric models}\label{hyyp}

We consider here  simple hyperparametric models.  Approximations and convergence results in  a  more complicated  case concerning   edge-preserving     Gaussian hierarchical models  were obtained in \cite{helin}.

 Let $\lambda $ be a Borel measure on $\mathbf R^d$.  Let  $\nu_{n}^t$, $n\in \mathbf N$ be Borel probability measures on  a locally convex Souslin  topological vector space  $ F$  for all $t\in \mathbf R^d$ and let $t\mapsto \nu^t_n (U)$ be  $\lambda$-measurable 
 for all $U\in \mathcal F$ and $n$. If $\nu_n^t$ converge weakly to  the probability  measure  $\nu^t$ for all $t$ then  the  hierarchical prior model
$\mu_{n}(U)=\int \nu_{n}^t(U)d\lambda(t)$
converge weakly. Indeed, if $f$ is a continuous bounded function on $F$, then 
$$
\lim_{n\rightarrow \infty }\mu_n(f) = \int \lim_{n\rightarrow \infty } \nu_n^t  (f) d\lambda(t) =\mu(f).
$$

For example,  take $X=\alpha Z$, where $Z$  is  zero mean Gaussian with  covariance $C$, and $\alpha$ is an ordinary random variable independent from $Z$ (so-called scale-mixing).  We take $Z_n$ to be the linear discretizations 
$ P_n(Z)$ and $X_n= \alpha Z_n$, where $P_n(Z)\rightarrow Z$ a.s. as $n\rightarrow \infty$. We denote the  distribution of $\alpha $ with $\lambda $. Then the
hierarchical prior distributions 
$$
\mu_{X_n}(U)=\int \mu_{ t Z_n} (U)d\lambda(t)
$$
converge weakly to $\mu_X$ as $n\rightarrow \infty$. This holds especially for  sub-Gaussian processes.  Moreover, all  spherically $H$-symmetric nonatomic measures are  mixtures of Gaussian measure $\mu_{tZ}$, where $Z$ is centered Gaussian with infinite-dimensional  Cameron-Martin space $H$  by Theorem 7.4.2 in \cite{bog}.

If   only the distribution of  the  hyperparameters $\alpha \in \mathbf R^d$  is approximated, we may get stronger convergence. Indeed, let $Z$ be any $F$-valued random variable and let $\alpha_n$ be a sequence of hyperparameters with probability densities $\lambda_n(t)$ on $\mathbf R^d$. We set $X=\alpha Z$ and $X_n= \alpha_n Z$.  If the densities $\lambda_n(t)$     converge almost everywhere to the density  $\lambda(t)$ of  the hyperparameter $\alpha$  and the densities are uniformly bounded, then the  hierarchical  prior distributions
  $$
\mu_{X_n}(U)=\int_{\mathbf R^d} \mu_{ tZ}(U)\lambda_n(t)dt
$$
  converge  in variation to $\mu_X$. Under the conditions of  Theorem \ref{coco}, also the corresponding posterior distributions converge in variation.

\subsection{Linear discretizations of random variables} \label{rs}

 In    \cite{siltanen}, a sequence of  random variables $X_n$ having values in a separable Banach space $F$ (equipped with the Borel $\sigma$-algebra)  is   called {\it a proper linear discretization} of an $F$-valued random variable  $X$ if $X_n$ converge to $X$  weakly  in distribution (i.e.   $X_n=P_n X$  for a sequence of   finite-rank operators $P_n$ on $F$ and $\mu_{ \langle X_n,\phi\rangle }$ converge weakly to $\mu_{\langle X,\phi\rangle} $   for all continuous linear functionals $\phi\in F'$).    The  definition of a proper linear discretization is too weak  for the present convergence results. The reason is that  the weak convergence in distribution  is equivalent to the convergence of  the characteristic functionals $\widehat \mu_{X_n} (\phi)$ to the characteristic functional   $\widehat \mu_X (\phi)$ for all $\phi\in F'$  (see Theorem 7.6 in \cite{bill}) and  the weak convergence of the approximated prior distributions $\mu_{X_n}$ is  guaranteed if the measures $\mu_{X_n}$ are additionally uniformly tight  (see Corollary 3.8.5 in \cite{bog}).  In order to get convergent  CM estimates,    Lassas et al  applied  in \cite{siltanen2} an  enforced  condition  that $P_n x$ converge in norm to $x$  for all  $x\in F$. We follow  the   Gaussian case  \cite{sari} and call $X_n$  {\it a  measurable  linear discretization} of $X$ if there exists measurable  operators $P_n$ having finite-dimensional ranges on $F$ such that   $X_n=P_n(X)$ and    $\mu_{X_n}$ converge weakly   to   the measure  $\mu_X$ on $F$.  Moreover, we call  $X_n$  {\it  a continuous  linear discretization } of $X$ if there exists  finite-rank  operators  (i.e. bounded linear operators with finite-dimensional ranges) $P_n$ on $F$ such that   $X_n=P_n X$ and    $\mu_{X_n}$ converge weakly   to   the measure  $\mu_X$.    We discuss shortly the existence of 
certain linear discretizations.

   The notion of a continuous linear discretization is related  to the  so-called  $\mu$-approximation property.       Let $\mu$ be  a Radon probability measure on a separable  Banach space $F$.  The space $F$ is said to have {\it the $\mu$-approximation property},   if there exists finite-rank operators  $P_n$   converging  $\mu$-a.s. 
  to identity  on $F$.  Moreover, the space $F$ is said to have {\it  the stochastic approximation property}  if  it has $\mu$-approximation property for every Radon probability measure $\mu$ (see   \cite{fonf}).   
       The  stochastic $\mu_X$-approximation property gives finite-rank   operators  $P_n$,  which  define continuous linear discretizations of  $X$ by  $X_n=P_nX$.    In \cite{fonf}, it was  demonstrated that not all  separable Banach spaces have stochastic approximation property.  Hence,  not all    separable Banach space-valued random variables $X$   have  almost everywhere converging continuous linear discretizations. 
   
      In \cite{fonf} it was  shown that on separable Banach spaces the stochastic approximation    property coincides with the existence of a stochastic basis  of  Herer    --  a biorthogonal system $(e_k,f_k)$ in $(F,F')$ such that  $x=\sum_{k=1}^\infty f_k(x)e_k$ for $\mu$-almost every $x$ in $F$ \cite{herer}.   Candidates of the type $X_n=\sum_{k=1}^n f_k(X)e_k$  are therefore plausible for almost everywhere converging continuous linear discretizations of $X$.  Moreover, if the coefficients $f_k(X)$ are   mutually statistically independent and their distributions are equivalent to the Lebesgue measure, then many properties of Gaussian measures hold also for  $\mu_X$ \cite{sato}. For example,  measures $\mu_{X+x_0}$ are then absolutely continuous with respect to  $\mu_X$ for every $x_0$ in the linear span of $\{e_k\}$, the linear span of $\{e_k\}$ has either  $\mu_X$-measure zero or one   and the topological support of $\mu_X$  coincides with the closure of the linear span of $\{e_k\}$ in $F$.
     
 Continuous linear discretizations  of Souslin space-valued unknowns  $X$ with property $\mu_X(H)=1$ for some separable Hilbert space $H$ were  considered in \cite{petteri} (see the discussion in Section \ref{sec13}).

  In some cases,   the prior distribution on  a  separable  Fr\'echet space $F$ may not be quite what we expect.     Okazaki  \cite {oka} proved a  remarkable result  that  for any separable  Fr\'echet  space  $F$ equipped with the Borel $\sigma$-algebra, a   probability measure  $\mu $  and a stochastic basis  ${(e_k,f_k)}_{k\in\mathbf N}$, there   exists a  separable  Banach space $B$ such that  $B\subset F$ continuously, 
  $\mu(B)=1$ and  the  stochastic basis  $(e_k,f_k)_{k\in\mathbf N}$ is actually a  Schauder basis for $B$.   The prior distribution $\mu$  on $F$ could be replaced with a prior distribution on the Banach space with the Schauder  basis.   The  above result refines   Kuelbs' classic  result that any Radon measure on a separable Fr\'echet space has  Banach support (see \cite{kue}, or  e.g.  Theorem  3.6.5  in \cite{bog}).

\subsection{Uniformly distributed sequences}\label{unif}

We suggest an approximation method that has been  historically valued as a competitor to Monte Carlo methods. The best known application of the approximation method  is  the so-called  quasi-Monte Carlo method (see    \cite{nie}).

  Let $F$ be a Hausdorff space  and $\mu$  a  finite Borel measure on $F$.  A sequence            $\{x_i\}_{i=1}^\infty \subset F $ is called $\mu$-uniformly distributed if 
          $$
          \lim_{n\rightarrow \infty }\frac{1}{n} \sum_{i=1}^n f(x_i) = 
          \int f (x)d\mu(x)
           $$
           for all continuous and bounded $f$ on $F$.             That is, the average  $n^{-1} \sum_{i=1}^n  \delta_{x_i}$ of point masses   converges weakly to $\mu$.  Roughly speaking,  it is the   law of large numbers with   a   predetermined  sequence.   
                   
    For any Borel probability measure $\mu_X$     on  a locally convex Souslin space $F$,  there exists some uniformly distributed sequence $\{x_i\}$  (see   Section 8.10 (ix) in \cite{bogm}).   A low dimensional example of such 
         a  sequence is the  Hammersley sequence for the Lebesgue   measure     on the unit square \cite{ham}. 
  This means that the  prior distribution  $\mu_X$ on $(F,\mathcal F)$   may   be approximated by measures $\mu_{X_n}=  n^{-1}\sum_{i=1}^n \delta_{x_i}$, whose prior information states that a realization of  the random variable $X_n$  is one of the values   $x_i$, $i=1,...,n$  and there is no preference between the values $x_i$, $i=1,...,n$.  The uniformly distributed  sequence is also  a possible tool  for interpreting prior information.

\section{Conclusions}

The generalized Bayes formula is an efficient tool for obtaining posterior convergence in the weak topology of measures or in the  stronger topologies of setwise convergence and convergence in variation (cf. Theorem  \ref{coco2} and Theorem \ref{coco}). In the case, when  only the hyperparameters of a hierarchical model are approximated,  we verified  that the posterior  distributions converge in variation when the approximations are refined (cf. Section \ref{hyyp}).    In Section \ref{sec5} we gave examples of  applicable non-Gaussian noise models. The explicit expressions  of   posterior distributions derived in Section \ref{sec5}   for simple  non-Gaussian noise models may serve  as model cases for further studies on the effects of non-Gaussianity of the noise distribution.  In particular,  a Kakutani type generalization of the Cameron-Martin formula  was derived. We anticipate that this  generalization  opens a way for a wide class of non-Gaussian noise models in statistical inverse problems,  especially when used in connection with the  wavelet expansions.   
Another example  demonstrates  the surprising fact that  the posterior  distribution given an infinite-dimensional observation  can have   significantly simpler expression  than the posterior distribution given a corresponding finite-dimensional observation.  This suggest that  in some cases the infinite-dimensional model could provide new numerical approximations schemes.   

It is well-known that the generalized Bayes formula holds  when the measures $\mu_{\varepsilon+L(x)}$ are  $\mu_X$-almost surely dominated i.e.  absolutely continuous with respect to some $\sigma$-finite measure for $\mu_X$-almost every $x\in F$. We showed that there is  a  curious interplay  between the continuity of posterior distributions with respect to the observations  and the $\mu_X$-a.s. domination of $\mu_{\varepsilon+L(x)}$ (cf. Theorem \ref{support2} and Remark \ref{divi}). The continuity of the posterior distribution with respect to observations   is only possible in the dominated case, which means that in the undominated cases some posterior distributions have discontinuities.   The discontinuities of the posterior distributions may enhance the errors caused by  replacing the required sample of $Y_n=L(X_n)+\varepsilon$  in   the posterior distribution of  $X_n$ given $Y_n$   by the actual   observation  of   $Y=L(X)+\varepsilon$.   Moreover, the regularizing effect of the prior distributions on an ill-posed inverse problem  could  be of limited power.

Continuity of the posterior distributions with respect to the observations  has also other roles in statistical inverse problems. It helps  to reduce the nonuniqueness  of posterior distributions in quite general cases  (cf. Theorem \ref{uniq}). 
  
   In Section \ref{rs}, we discussed the linear discretizations of the unknown $X$.  We remarked that on  arbitrary separable Banach spaces there does not always exist  a  continuous linear discretization  by  a result of Fonf \cite{fonf}. The present convergence results, which are written in the same spirit as in  \cite{helin},  are therefore important since they do not require the    pointwise convergence of the discretization operators.   Beside of continuous  linear discretizations,   other approximation methods can  therefore be used.  One of them is a generic method for approximating any prior measure on a locally convex Souslin space with the help of  a  quasirandom sequence (see Section  \ref{unif}).

 Finally, we list some directions for generalizing this work.  We have not studied the speed of convergence of posterior distributions, which is a very natural question when choosing between  different approximation schemes. The generalized Bayes formula gives a good  framework for this study.    Moreover, we  considered only classical noise models i.e.  statistically independent noise and unknowns, which enabled us to write a simple expression for the conditional   probability of the observation $Y$ given the unknown  $X$.  The case of  the statistically depended noise and the unknown is not purely theoretical, since often the unknown is approximated with a  simple expression and the  approximation error is included in the noise term. Convergence of  the corresponding posterior distributions  is  therefore an important topic. 
 Furthermore,   we have not discussed     what kind of approximating  prior distributions could  guarantee  meaningful convergence of  maximum a posteriori (MAP) estimates.  Note, that the question is proper for classical noise models that are statistically independent from the unknown  since  the posterior distribution  of $X$ given a sample of $Y=L(X)+\varepsilon$ depends then on  the prior model $X$  only through its distribution.  The example of Lassas and Siltanen on total variation priors shows that the  weak convergence of any approximating prior distributions is, in general, not sufficient as  MAP estimates   converged  then to zero.

There  is still a wide class of   statistical inverse problems, which  are covered  neither by the present work nor  \cite{sari,petteri},  where the question of posterior convergence remains open.
 
\section*{Acknowledgements}  This work was supported by the Academy of Finland
(application number 213476, Finnish Programme for Centres of
Excellence in Research 2006-2011).


\begin{thebibliography}{200}


 \bibitem{silve2}
 \newblock   F. Abramovich, T. Sapatinas and B. W. Silverman,
 \newblock   \emph {Wavelet thresholding via a Bayesian approach}, \newblock J. R. Stat. Soc. Ser. B Stat. Methodol., \textbf{60} (1998), 725--749. 

\bibitem{silve3}
\newblock F. Abramovich and B. W. Silverman, 
\newblock \emph {Wavelet decomposition approaches to statistical inverse problems}, 
\newblock Biometrika, \textbf{85} (1998), 115--129.
 
\bibitem{silve}
\newblock F. Abramovich, T. Sapatinas and B. W. Silverman, 
\newblock \emph {Stochastic expansions in an overcomplete wavelet dictionary}, 
\newblock Probab. Theory Related Fields, \textbf{117} (2000), 133--144. 
 
 \bibitem{maen}
 \newblock B. D'Ambrogi, S. M\"aenp\"a\"a and M. Markkanen, 
 \newblock \emph{Discretization independent  retrieval of atmospheric ozone profile}, Geophysica, \textbf{35} (1999), 87--99.

 
 \bibitem{backus}
\newblock  G. Backus, 
\newblock \emph {Isotropic probability measures in infinite-dimensional spaces},
\newblock  Proc. Nat. Acad. Sci. U.S.A.,  \textbf{84} (1987),  8755--8757.
 
\bibitem{barron} 
\newblock A. Barron, M. J. Schervish and L. Wasserman,  
\newblock \emph {The consistency of posterior distributions in nonparametric problems}, 
\newblock  Ann. Statist., \textbf{27} (1999), 536--561. 
    

\bibitem{bernardo} 
 \newblock J. M. Bernardo and A. F. M. Smith  
\newblock ``Bayesian theory,"
\newblock John Wiley \& Sons, Chichester, 1994.

\bibitem{berti} 
\newblock  P. Berti, L. Pratelli and  P. Rigo, 
\newblock \emph{Almost sure weak convergence of random probability measures}, 
\newblock  Stochastics, \textbf{78} (2006), 91--97.

\bibitem{bhat}
\newblock A B. Bhatt,  G. Kallianpur,  and R. L.  Karandikar,
\newblock\emph{Robustness of the nonlinear filter},
\newblock  Stochastic Process. Appl., \textbf{81} (1999),  247--254.


\bibitem{bill}
\newblock P. Billingsley,
\newblock ``Convergence of probability measures,"  John Wiley \& Sons, New York-London-Sydney, 1968.


\bibitem{biz}
\newblock N. Bissantz and H. Holzmann, 
\newblock \emph{Statistical inference for inverse problems},
\newblock Inverse Problems, \textbf{24} (2008),  034009--034025. 


\bibitem{bog} 
     \newblock V. I. Bogachev, 
     \newblock ``Gaussian Measures," 
     \newblock American Mathematical Society, Providence, RI, 1998.

\bibitem{bogm} 
     \newblock V. I. Bogachev,
     \newblock `` Measure Theory. Vol. I, II," 
     \newblock Springer-Verlag, Berlin, 2007.
     
       \bibitem{bogata}
\newblock  V. I.  Bogachev,  A. V.  Kolesnikov, and  K. V.  Medvedev, 
  \newblock \emph{Triangular transformations of measures},  
  \newblock Sb. Math., \textbf{196} (2005),  309--335. 
     
     
     
     \bibitem{bul}
\newblock V. V. Buldygin, 
\newblock \emph{On invariant Bayesian estimators for generalized random variables}, 
\newblock Theor. Probability Appl.,  22 (1977),  172--175.
     
     \bibitem{burgess}
\newblock J. P. Burgess and R. D.  Mauldin,
\newblock \emph{Conditional distributions and orthogonal measures},
\newblock Ann. Probab., \textbf{9} (1981),  902--906. 

 \bibitem{adali}
 \newblock  V. D. Calhoun and T.   Adali,  
 \newblock\emph{Unmixing fMRI with independent component analysis},
\newblock IEEE  Engineering in Medicine and Biology Magazine,  
\textbf{25} (2006),  79 -- 90.

\bibitem{cava} 
    \newblock L. Cavalier, 
    \newblock \emph{Inverse problems with non-compact operators},
    \newblock J. Statist. Plann. Inference,  \textbf{136}  (2006), 390--400.  
    
    \bibitem{cava2}
    \newblock L. Cavalier, \emph{Nonparametric statistical inverse problems}, 
    \newblock  Inverse Problems, \textbf{24} (2008), 034004--033025.
    
     \bibitem{sea}
 \newblock M. A.  Chitre, J. R.  Potter,  and  Ong  Sim-Heng,  
 \newblock\emph {Optimal and Near-Optimal Signal Detection in Snapping Shrimp Dominated Ambient Noise},
\newblock IEEE J. Ocean. Eng., \textbf{31} (2006), 497--503.
    
    \bibitem{gos} 
   \newblock N. Choudhuri, S. Ghosal and A.  Roy, 
\newblock \emph{Bayesian methods for function estimation},  in  D. K. Dey  et al (Eds.): Bayesian thinking: modeling and computation, 373--414, Elsevier/North-Holland, Amsterdam, 2005. 
  
  \bibitem{conte}
   \newblock  E. Conte and M. Longo,  
\newblock \emph{Characterisation of radar clutter as a spherically invariant random process},
\newblock IEE Proc. Part F, 
\textbf{134} (1987), 191--197.
 
      \bibitem{conte1}
     \newblock E. Conte, M.    Longo and M.   Lops,  
\newblock \emph{Modelling and simulation of non-Rayleigh radar clutter},
 \newblock IEE Proc. Part F, \textbf{138} (1991), 121--130.
 
   \bibitem{conte2}
 \newblock E. Conte and A. De Maio,
 \newblock \emph{Mitigation techniques for non-Gaussian sea clutter},
\newblock  IEEE J. Ocean. Eng., \textbf{29} (2004), 284--302.

     \bibitem{cotter}
\newblock S. L. Cotter, M. Dashti, J. C. Robinson and A. M. Stuart, 
\newblock \emph {Bayesian inverse problems for functions and applications to fluid mechanics},
\newblock  Inverse Problems, \textbf{25} (2009), 115008--1150051.

\bibitem{cotter2}
\newblock S. L. Cotter, M. Dashti, and A. M. Stuart,
\newblock \emph{Approximations of Bayesian inverse problems for PDEs},
\newblock SIAM J. Numer. Anal., \textbf{48} (2010), 322--345.

  \bibitem{cox} 
 \newblock  D. D. Cox, 
 \newblock \emph {An analysis of Bayesian inference for nonparametric regression},
 \newblock   Ann. Statist.,  \textbf{21} (1993), 903--923.
  
     \bibitem{crimaldi1}
   \newblock
   I. Crimaldi and L. Pratelli,
\newblock \emph{Convergence results for conditional expectations}, 
Bernoulli, \textbf{11} (2005),  737--745. 
 
 \bibitem{crimaldi2}
  \newblock I. Crimaldi and L. Pratelli, 
\newblock  \emph {Two inequalities for conditional expectations and convergence results for filters}, 
\newblock Statist. Probab. Lett., \textbf {74} (2005),  151--162. 



\bibitem{davenport}
 \newblock W. B. Davenport and W. L. Root,
\newblock ``An introduction to the theory of random signals and noise", 
 \newblock McGraw-Hill Book Company, Inc., New York-Toronto-London, 1958. 


\bibitem{diaco}
\newblock P. Diaconis and D. Freedman,
\newblock \emph {On the consistency of Bayes estimates},
\newblock Ann. Statist., \textbf{14} (1986), 1-67. 

\bibitem{dia} 
     \newblock P. Diaconis and D. Freedman,
     \newblock \emph{Consistency of Bayes estimates for nonparametric regression: normal theory},
     \newblock Bernoulli, \textbf{4} (1998), 411--444.


\bibitem{diestel}
\newblock J.  Diestel and J. J.  Uhl, 
 \newblock ``Vector measures",
 \newblock American Mathematical Society, Providence, RI, 1977.


\bibitem{die1} 
\newblock J. Dieudonn\'e, 
\emph{Un exemple d'espace normal non susceptible d'une structure uniforme d'espace complet}, 
\newblock C. R. Acad. Sci. Paris, \textbf{209} (1939), 145--147.

\bibitem{die2} 
\newblock J. Dieudonn\'e,  \emph{Sur le th\'eor\`eme de Lebesgue-Nikodym. III}, 
\newblock Ann. Univ. Grenoble. Sect. Sci. Math. Phys. (N.S.), \textbf{ 23}  (1948), 25--53.

\bibitem{doob}
\newblock J. L. Doob, 
\newblock \emph{Stochastic processes depending on a continuous parameter},
\newblock Trans. Amer. Math. Soc., \textbf{42} (1937), 107--140. 

\bibitem{doob1}
\newblock J. L. Doob, 
\newblock \emph{Stochastic processes with an integral-valued parameter},
\newblock Trans. Amer. Math. Soc., \textbf{44} (1938), 87--150. 

\bibitem{doobm}
\newblock J. L. Doob,
\newblock \emph {Application of the theory of martingales},
\newblock  in Le Calcul des Probabilit\'es et ses Applications, pp. 23--27, Coll. Int. du CNRS Paris, 1949. 

\bibitem{dud} 
     \newblock R. M. Dudley,
     \newblock ``Real Analysis and Probability," 
     \newblock Cambridge University Press, Cambridge, 2002.

\bibitem{evans} 
     \newblock S. N. Evans and P. B. Stark,
     \newblock \emph{Inverse problems as statistics},
     \newblock  Inverse Problems, \textbf{18} (2002), R55--R97.

\bibitem{esco}
\newblock  M. D. Escobar and M. West,
 \newblock\emph{Bayesian density estimation and inference using mixtures},
 \newblock  J. Amer. Statist. Assoc., \textbf{90} (1995),   577--588.


\bibitem{fer2}
 \newblock T. S. Ferguson.
\emph{Prior distributions on spaces of probability measures}.
Ann.  Statist., \textbf{2} (1974), 615--629.
  
 \bibitem{fer}
 \newblock T. S. Ferguson, 
 \newblock \emph{ A Bayesian analysis of some nonparametric problems},
 \newblock   Ann. Statist., \textbf{1} (1973), 209--230.

\bibitem{fitz}
     \newblock B. G. Fitzpatrick,
     \newblock \emph{Bayesian analysis in inverse problems},
     \newblock Inverse Problems, \textbf{7} (1991), 675--702.
 
\bibitem{flo} 
     \newblock  J.-P. Florens, M. Mouchart and  J.-M. Rolin,  
     \newblock ``Elements of Bayesian Statistics,"
         \newblock Marcel Dekker, Inc., New York, 1990.
         
         
         \bibitem{simoni}
         \newblock J.-P. Florens and A. Simoni,
         \newblock \emph{Regularizing priors for linear inverse problems},
          \newblock IDEI Working paper, \textbf{621} (2010). 

\bibitem{fonf} 
     \newblock V. P. Fonf,  W .B. Johnson,  G. Pisier and  D. Preiss, 
          \newblock \emph{Stochastic approximation properties in Banach spaces},
     \newblock Studia Math., \textbf {159} (2003), 103--119.

\bibitem{foster}(0145649)
\newblock M. Foster,
\newblock \emph{An application of the Wiener-Kolmogorov smoothing theory to matrix inversion},
\newblock J. Soc. Indust. Appl. Math., \textbf{9} (1961), 387--392. 

\bibitem{frank} 
     \newblock J. N. Franklin,
     \newblock \emph{Well-posed stochastic extensons of ill-posed linear problems},
     \newblock J. Math. Anal. Appl., \textbf{31} (1970), 682--716. 
  

  
  
    \bibitem{frechet}
\newblock M. Fr\'echet,
\newblock \emph {On two new chapters in the theory of probability},
\newblock Math. Mag., \textbf{22} (1948), 1--12. 
 
 \bibitem{freed}
\newblock D. Freedman,
\newblock \emph {On the asymptotic behavior of Bayes estimates in the discrete case I},
\newblock Ann. Math. Statist., \textbf{34} (1963), 1386-1403.
  
  
   \bibitem{gans} 
\newblock P. G\"anssler and J. Pfanzagl,
\newblock \emph{Convergence of conditional expectations},
\newblock Ann. Math. Statist., \textbf{42} (1971), 315--324. 
  
  
  \bibitem{gel}
 \newblock  I. M. Gelfand and N.Ya. Vilenkin,
 \newblock ``Generalized functions. Vol. 4: Applications of harmonic analysis",
 \newblock  Academic Press, New York - London, 1964. 
     
  \bibitem{gho}
\newblock J. K. Ghosh  and R. V.  Ramamoorthi,   
  \newblock 
 ``Bayesian Nonparametrics,"
   \newblock Springer-Verlag, New York, 2003. 

\bibitem{gih}
\newblock I. I. Gihman and A. V. Skorohod, 
\newblock ``The theory of stochastic processes I", 
\newblock  Springer-Verlag, New York-Heidelberg, 1974. 
 
 \bibitem{goggin}(1288145)
 \newblock   E. Goggin, 
 \newblock \emph{Convergence in distribution of conditional expectations},
 \newblock  Ann. Probab., \textbf{22} (1994), 1097--1114.

\bibitem{ctnoise}
\newblock P. Gravel,  G. Beaudoin and  J. A. De Guise,
\newblock  \emph{A method for modeling noise in medical images},
\newblock  IEEE Trans Med Imaging., \textbf{23} (2004), 1221-32.

 \bibitem{grenander}
\newblock U. Grenander, 
\newblock \emph{Stochastic processes and statistical inference},  \newblock Ark. Mat., \textbf{1} (1950), 195--277.
 


\bibitem{halmos} 
\newblock P. R. Halmos and L. J. Savage, 
\newblock \emph{Application of the Radon-Nikodym theorem to the theory of sufficient statistics},
\newblock Ann. Math. Statist., \textbf{20} (1949), 225--241. 



\bibitem{ham} 
     \newblock J. M. Hammersley,
     \newblock \emph{Monte Carlo methods for solving multivariable problems},
     \newblock  Ann. New York Acad. Sci., \textbf{86} (1960), 844--874.


\bibitem{harada}
\newblock K. Harada and H. Saigo,
\newblock\emph{The space of tempered distributions as a k-space},
\newblock  preprint,  ArXiv {1009.1429}.



\bibitem{hegland}
\newblock M. Hegland, 
\newblock \emph{Approximate maximum a posteriori with Gaussian process priors},
\newblock Constr. Approx., \textbf{26} (2007), 205--224. 

\bibitem{helin}
\newblock T. Helin, 
\newblock \emph{On infinite-dimensional hierarchical probability models in statistical inverse problems},
\newblock  Inverse Probl. Imaging, \textbf{3} (2009), 567--597.   

\bibitem{helin2}
\newblock T. Helin and M. Lassas,
\newblock\emph{Hierarchical models in statistical inverse problems and the Mumford-Shah functional},
\newblock Inverse problems, \textbf{27} (2011), 015008--014039.

\bibitem{herer} 
     \newblock W. Herer,
     \newblock \emph{Stochastic bases in Fr\'echet spaces},
     \newblock Demonstratio Math., \textbf{14} (1981), 719--724.


 \bibitem{mar}
 \newblock J. A.  Hildebrand,
\newblock\emph{Anthropogenic and natural sources of ambient noise in the ocean},
\newblock Mar Ecol Prog Ser., \textbf{295} (2009), 5--20.



\bibitem{pikkarainen3}
\newblock A. Hofinger and H. K. Pikkarainen,
\newblock \emph{Convergence rate for the  Bayesian approach to linear inverse problems},
\newblock Inverse Problems, \textbf{23} (2007), 2469--2484.


\bibitem{pikkarainen2}
\newblock A. Hofinger and H. K. Pikkarainen,
\newblock \emph{Convergence rates for linear inverse problems in the presence of an additive normal noise},
\newblock Stoch. Anal. Appl., \textbf{27} (2009), 240--257.


 
 \bibitem{jessen}
\newblock B. Jessen,
\newblock \emph{The theory of integration in a space of an infinite number of dimensions},
\newblock Acta Math., \textbf{63} (1934), 249--323. 
 
 \bibitem{ji1}
\newblock M. Ji\v{r}ina, 
\newblock\emph{On regular conditional probabilities}, 
\newblock  Czechoslovak Math. J.,  \textbf {9} (1959), 445--451. 

\bibitem{ji2}
 \newblock M. Ji\v{r}ina, 
\newblock\emph{Conditional probabilities on $\sigma $-algebras with countable basis},  in Select. Transl. Math. Statist. and Probability, 
\newblock Vol. 2,  pp. 79--86, American Mathematical Society, Providence, RI, 1962. 
 
 
 
 \bibitem{silver}
\newblock I. M. Johnstone and B. W. Silverman,
\newblock\emph{Speed of estimation in positron emission tomography and related inverse problems},
\newblock Ann. Statist., \textbf{18} (1990), 251--280. 

\bibitem{kahane} 
     \newblock J.-P. Kahane,
     \newblock ``Some Random Series of Functions," 
          \newblock  Cambridge University Press, Cambridge, 1985.

\bibitem{kailath} 
\newblock T. Kailath,
\newblock\emph{A view of three decades of linear filtering theory},
\newblock IEEE Trans. Information Theory,\textbf{IT-20} (1974), 146--181. 

\bibitem{kaip} 
     \newblock J. Kaipio and E. Somersalo,
     \newblock ``Statistical and Computational Inverse Problems,"     
     \newblock Springer-Verlag, New York, 2005.

\bibitem{ks}
\newblock J. Kaipio and E. Somersalo, 
\newblock \emph{Statistical inverse problems: discretization, model reduction and inverse crimes},
\newblock  J. Comput. Appl. Math., \textbf{198} (2007), 493--50.

\bibitem{kakutani}
\newblock S. Kakutani,
\newblock \emph{On equivalence of infinite product measures},
\newblock  Ann. of Math.,  \textbf{49} (1948), 214--224.

 \bibitem{kallianpur}
\newblock G. Kallianpur,
\newblock \emph{Stochastic filtering theory}, 
\newblock Springer-Verlag, New York-Berlin, 1980.

\bibitem{kallianpur2} 
\newblock G. Kallianpur and C. Striebel,
\newblock \emph{Estimation of stochastic systems: Arbitrary system process with additive white noise observation errors},
\newblock Ann. Math. Statist. \textbf{39} (1968), 785--801



\bibitem{karhunen}
\newblock K. Karhunen, 
\newblock\emph{\"Uber lineare Methoden in der Wahrscheinlichkeitsrechnung},
\newblock Ann. Acad. Sci. Fennicae. Ser. A. I. Math.-Phys., No. 37 (1947).

\bibitem{kelly}  
\newblock E. J. Kelly, I. S. Reed and W. L. Root,
\newblock \emph{The detection of radar echoes in noise. I, II},
\newblock J. Soc. Indust. Appl. Math., \textbf{8} (1960), 309--341, 481--507. 


\bibitem{wahba2}
\newblock G. S. Kimeldorf and G.  Wahba,  
\newblock \emph{A  correspondence between Bayesian estimation on stochastic processes and smoothing by splines},
\newblock  Ann. Math. Statist., \textbf{41} (1970), 495--502. 


\bibitem{siltanen3}
\newblock V. Kolehmainen, M. Lassas, K. Niinim\"aki, and S. Siltanen,
\newblock\emph{Sparsity-promoting Bayesian inversion},
\newblock Preprint (2011).


\bibitem{kol}
\newblock A. Kolmogorov, 
\newblock ``Grundbegriffe der Wahrscheinlichkeitsrechnung", \newblock Springer, Berlin, 1933.

\bibitem{kol1}
\newblock A. Kolmogorov.
\newblock \emph{Stationary sequences in Hilbert's space} (Russian),
\newblock Bolletin Moskovskogo Gosudarstvenogo Universiteta. Matematika, \textbf{2} (1941). 


\bibitem{kri}
\newblock K. Krikkeberg,
\newblock \emph {Convergence of conditional expectation operators}, 
\newblock Theory Probab. Appl. \textbf{9} (1964), 538--549.

\bibitem{krein1} 
\newblock M. Krein, 
\newblock\emph{On a generalization of some investigations of G. Szeg\"o, V. Smirnoff and A. Kolmogoroff},
\newblock   C. R. (Doklady) Acad. Sci. URSS (N.S.), \textbf{46} (1945), 91--94. 


\bibitem{krein2}
\newblock M. Krein, 
\newblock\emph{On a problem of extrapolation of A. N. Kolmogoroff},
\newblock  C. R. (Doklady) Acad. Sci. URSS (N. S.), \textbf{46} (1945),  306--309.




\bibitem{krug}
\newblock P. Krug, 
\newblock\emph{
The conditional expectation as estimator of normally distributed random variables with values in infinitely-dimensional Banach spaces},  
\newblock J. Multivariate Anal., \textbf{38} (1991), 1--14. 


\bibitem{kue} 
     \newblock J. Kuelbs,
     \newblock \emph{Some results of probability measures on linear topological vector spaces with an application to Strassen's log log law},
     \newblock  J. Funct. Anal., \textbf{14} (1973), 28--43.

\bibitem{kuo}
\newblock H. H. Kuo,
\newblock ``Gaussian measures in Banach spaces",
\newblock  Springer-Verlag, Berlin-New York, 1975. 


 \bibitem{ray}
 \newblock E. E. Kuruoglu, W. J. Fitzgerald, and P. J. W.  Rayner,  
\newblock \emph{Near optimal detection of signals in impulsive noise modeled with a symmetric $\alpha$-stable distribution},
\newblock IEEE Communications Letters, \textbf{2} (1998), 282--284.



\bibitem{land} 
\newblock D. Landers and  L. Rogge,  
\newblock \emph{ A generalized Martingale theorem},
\newblock Z. Wahrsch. Verw. Gebiete, \textbf {23} (1972), 289--292.

\bibitem{sari}
\newblock  S. Lasanen, 
\newblock \emph{Discretizations of generalized random variables with applications to inverse problems}, Dissertation, University of Oulu,
\newblock   Ann. Acad. Sci. Fenn. Math. Diss.,  No.130 (2002).

\bibitem{roininen}
\newblock S. Lasanen and L. Roininen,
\newblock \emph{Statistical inversion with Green's priors}, \newblock Proceedings of the 5th
International Conference on Inverse Problems in Engineering: Theory and Practice, Cambridge,UK, 11-15th July 2005.

\bibitem{siltanen2} 
     \newblock M. Lassas, E. Saksman and S. Siltanen, 
     \newblock \emph{Discretization-invariant Bayesian inversion and Besov space priors},
     \newblock Inverse Probl. Imaging, \textbf {3} (2009), 87--122.

\bibitem{siltanen} 
     \newblock M. Lassas and S. Siltanen, 
     \newblock \emph{Can one use total variation prior for edge-preserving Bayesian inversion?},
     \newblock  Inverse Problems, \textbf {20} (2004), 1537--1563.
    


\bibitem{ledoux} 
\newblock M. Ledoux and M. Talagrand,
\newblock ``Probability in Banach spaces. Isoperimetry and processes", 
\newblock Springer-Verlag, Berlin, 1991.

\bibitem{koodi}
\newblock M. Lehtinen, B. Damtie, P. Piiroinen and M. Orisp\"a\"a,
\newblock \emph{Perfect and almost perfect pulse compression codes for range spread radar target},
\newblock Inverse probl. Imaging, \textbf{2} (2009), 465--486.  


\bibitem{markku} 
     \newblock M. Lehtinen, L. P\"aiv\"arinta and E. Somersalo,     \newblock \emph{Linear inverse problems for generalised random variables},
     \newblock Inverse Problems, \textbf{5} (1989), 599--612.


\bibitem{lewa}
\newblock M. Lewandowski, M.  Ryznar, and T.  Zak,
\newblock \emph{Anderson inequality is strict for Gaussian and stable measures},
\newblock  Proc. Amer. Math. Soc., \textbf{123} (1995),  3875--3880. 


     \bibitem{lus}
\newblock H. Luschgy,
\newblock \emph{Linear estimators and Radonifying operators}, 
\newblock Theory Probab. Appl., \textbf{40} (1995), 167--175. 
     
 
\bibitem{macci} 
\newblock C. Macci,
\newblock \emph{On the Lebesgue decomposition of the posterior distribution with respect to the prior in regular Bayesian experiments}, \newblock Statist. Probab. Lett., \textbf{26} (1996), 147--152. 

\bibitem{mandal}
\newblock P. K. Mandal and V.  Mandrekar,
\newblock \emph{A Bayes formula for Gaussian noise processes and its applications},
\newblock SIAM J. Control Optim, \textbf{39} (2000),  852--871. 


\bibitem{mandelbaum} 
     \newblock A. Mandelbaum,
     \newblock \emph{Linear estimators and measurable linear transformations on a Hilbert space},
     \newblock  Z. Wahrsch. Verw. Gebiete,  \textbf{65}  (1984),   385--397.


\bibitem{muller}
\newblock  P. M\"uller and  F. A. Quintana,
\newblock  \emph{Nonparametric Bayesian data analysis}, 
\newblock  Statist. Sci., \textbf{19} (2004), 95--110. 





 \bibitem{pikkarainen}
  \newblock A. Neubauer and H. K. Pikkarainen,
  \newblock \emph{Convergence results for the Bayesian inversion theory},
  \newblock J. Inverse Ill-Posed Probl., \textbf{16} (2008), 601--613.

\bibitem{neveu}
\newblock J. Neveu, 
\newblock ``Discrete-parameter martingales",
\newblock  North-Holland Publishing Co., Oxford, 1975.  


\bibitem{nie} 
     \newblock H. Niederreiter,
     \newblock ``Random Number Generation and Quasi-Monte Carlo Methods",
     \newblock SIAM, Philadelphia, PA, 1992.

\bibitem{oka} 
     \newblock  Y. Okazaki,
     \newblock \emph{Stochastic basis in Fr\'echet space},
     \newblock Math. Ann., \textbf{274} (1986), 379--383.


\bibitem{oek}
\newblock B. Oeksendal, 
\newblock \emph{Stochastic differential equations. An introduction with applications}, 
\newblock Springer-Verlag, Berlin, 2003.



\bibitem{osul} 
     \newblock   F. O'Sullivan,   
      \newblock \emph{A statistical perspective on ill-posed inverse problems},
     \newblock Statist. Sci., \textbf{1} (1986), 502--527. 


\bibitem{partha}
 \newblock K. R. Parthasarathy,
 \newblock ``Probability measures on metric spaces",
  \newblock AMS Chelsea Publishing, Providence, RI, 2005.

\bibitem{pet}
\newblock B. J.  Pettis, 
\newblock\emph{On integration in vector spaces},
\newblock  Trans. Amer. Math. Soc., \textbf{44} (1938), 277--304.




\bibitem{phil}
\newblock D. L. Philips,
\newblock \emph{A technique for the numerical solution of certain integral equations of the first kind}, \newblock Journal of the ACM,  \textbf{9}  (1962), 84--97. 

\bibitem{petteri}
\newblock P. Piiroinen, 
\newblock \emph{Statistical measurements, experiments and applications} Dissertation, University of Helsinki,
\newblock  Ann. Acad. Sci. Fenn. Math. Diss., No 143 (2005).


 \bibitem{poin2}
 \newblock  H. Poincar\'e,  
 \newblock ``Science and Hypothesis",  
\newblock Walter Scott Publishing, London, 1905.


\bibitem{poin} 
     \newblock  H. Poincar\'e,
     \newblock ``Calcul des probabilit\'es", 
        \newblock \'Editions Jacques Gabay, Sceaux, 1987.


\bibitem{pren} 
     \newblock P. M. Prenter and  C. R. Vogel, 
     \newblock \emph{Stochastic inversion of linear first kind integral equations. I. Continuous theory and the stochastic generalized inverse}, \newblock J. Math. Anal. Appl., \textbf{106} (1985),  202--218.


\bibitem{regu}
\newblock D. Ramachandran, 
\newblock \emph{A note on regular conditional probabilities in Doob's sense},
\newblock Ann. Probab., \textbf{9} (1981),  907--908. 


\bibitem{revuz}
\newblock D. Revuz and M. Yor,
\newblock ``Continuous martingales and Brownian motion", 
Springer-Verlag, Berlin, 1999.

\bibitem{robert}
 \newblock C. P. Robert, 
 \newblock ``The Bayesian choice. From decision-theoretic foundations to computational implementation",
 \newblock   Springer-Verlag, New York, 2001. 


\bibitem{roh} 
\newblock V. A. Rohlin, 
\newblock\emph{On the fundamental ideas of measure theory} (Russian) Mat. Sbornik N.S. \textbf{ 25(67)} (1949) 107--150.
Translated in  Amer. Math. Soc. Translation  \textbf{71} (1952). 

\bibitem{rud} 
     \newblock W. Rudin, 
     \newblock \emph{Lebesgue's first theorem},      \newblock In L. Nachbin (Ed): Mathematical analysis and applications. Part B, pp. 741--747, \newblock Academic Press, 1981, 

\bibitem{taqqu}
 \newblock G.  Samorodnitsky and M. S.  Taqqu, 
 \newblock  ``Stable non-Gaussian random processes", 
 \newblock  Chapman \and Hall, New York, 1994.




\bibitem{sato}
\newblock  H.  Sato, 
\newblock \emph{ An ergodic measure on a locally convex topological vector space},
\newblock  J. Funct. Anal., \textbf{43} (1981), 149--165.

\bibitem{sa} 
\newblock V. V. Sazonov, 
\newblock \emph{On perfect measures}, 
\newblock Izv. Akad. Nauk SSSR Ser. Mat. \textbf {26} (1962), 391--414. Translated in American Mathematical Society Translations. Series 2. Vol. 48: Fourteen papers on logic, algebra, complex variables and topology. American Mathematical Society, Providence, RI, 1965.




\bibitem{schervish}
\newblock M. J. Schervish,
\newblock ``Theory of statistics",
 \newblock Springer-Verlag, New York, 1995.

\bibitem{schwarz} 
     \newblock L. Schwartz,
     \newblock ``Radon Measures on Arbitrary Topological Spaces and Cylindrical Measures," 
     \newblock Oxford University Press, London, 1973.

\bibitem{lorraine}
\newblock L. Schwartz,
\newblock \emph{On Bayes procedures}, 
\newblock  Z. Wahrsch. Verw. Gebiete, \textbf{4} 1965, 10--26. 


\bibitem{shimomura}
\newblock H. Shimomura, 
\newblock\emph{Some new examples of quasi-invariant measures on a Hilbert space},
\newblock  Publ. Res. Inst. Math. Sci., \textbf{11} (1975/76),  635--649.


\bibitem{shir} 
     \newblock  A. N. Shiryaev,
     \newblock ``Probability," 
     \newblock  Springer-Verlag, New York, 1996.


\bibitem{simonithe}
\newblock A. Simoni,
\newblock ``Bayesian Analysis of Linear Inverse Problems with Applications in Economics and Finance,"
 \newblock Dissertation,  Univ. of Bologna,  2009.
 
 \bibitem{slu}
\newblock E. Slutsky,
\newblock \emph{Quelques propositions sur la th\'eorie des fonctions al\'eatoires} (Russian),
\newblock Acta [Trudy] Univ. Asiae Mediae. Ser. V-a.,  \textbf{31} (1939). 



 \bibitem{stein}
\newblock D. M. Steinberg,
\newblock \emph{A Bayesian approach to flexible modeling of multivariable response functions},
\newblock  J. Multivariate Anal., \textbf{34} (1990), 157--172. 
 
\bibitem{strand}
\newblock O. N. Strand and E. R. Westwater,
\newblock \emph{Statistical estimation of the numerical solution of a Fredholm integral equation of the first kind},
\newblock  J. Assoc. Comput. Mach., \textbf{15} (1968), 100--114.
 
 
  \bibitem{stuart}
  \newblock A. M. Stuart,
  \newblock  \emph{Inverse Problems: A Bayesian Perspective}, 
  \newblock  Acta Numerica, \textbf{19} (2010), 451--559. 



\bibitem{tarantola}
\newblock A. Tarantola, 
\newblock  ``Inverse Problem Theory. Methods for Data Fitting and Model Parameter Estimation",
\newblock  Elsevier Science Publishers,  Amsterdam, 1987.

\bibitem{tarantol}
\newblock  A. Tarantola  and B. Valette, 
\newblock \emph {Inverse Problems = Quest for Information}, 
\newblock  J. Geophys., \textbf{50} (1982), 159-170.

\bibitem{tanja}
\newblock  T. Tarvainen, V. Kolehmainen, A. Pulkkinen, M. Vauhkonen, M. Schweiger, S. R. Arridge, and J. P. Kaipio, 
   \newblock  \emph{An approximation error approach for compensating for modelling errors between the radiative transfer equation and the diffusion approximation in diffuse optical tomography},
   \newblock Inverse Problems, \textbf{26} (2010), 015005--015023.

\bibitem{luis} 
    \newblock L. Tenorio,
    \newblock  \emph{Statistical regularization of inverse problems},
    \newblock  SIAM Rev., \textbf{43} (2001), 347--366.


\bibitem{thomas}
\newblock G. E. F. Thomas,
\newblock \emph{Integration of functions with values in locally convex Suslin spaces},
\newblock Trans. Amer. Math. Soc., \textbf{212} (1975), 61--81. 


 

\bibitem{turchin} 
\newblock V. F. Turchin, \newblock \emph{Statistical regularization}, in H. J. Krappe et al (Eds): Advanced methods in the evaluation of nuclear scattering data, pp. 33-49, 1985, \newblock Springer, Berlin, 1985. 

\bibitem{twomey} 
\newblock S. Twomey, 
\newblock \emph{On the numerical solution of Fredholm integral equations of the first kind by the inversion of the linear system produced by quadrature},
\newblock  J. Assoc. Comput. Mach., \textbf{10} (1963),  97--101.
 
 \bibitem{umemura}
\newblock Y. Umemura,
\newblock \emph{Measures on infinite dimensional vector spaces},
\newblock Publ. Res. Inst. Math. Sci. Ser. A, \textbf{1} (1965),  1--47. 
 
 
 \bibitem{ur}
 \newblock  R. J. Urick,
 \newblock ``Ambient noise in the sea",
 Undersea Warfare Technology Office, Naval Sea Systems Command, Dept. of the Navy, Washington, D.C., 1984.


\bibitem{vakha} 
     \newblock N. N. Vakhania,  V.I. Tarieladze and S. A. Chobanyan,
     \newblock ``Probability Distributions on Banach Spaces," 
     \newblock Reidel Publishing Co., Dordrecht, 1987.

\bibitem{vaart} 
\newblock A. W. van der Vaart and J. H. van Zanten,
\newblock \emph{Rates of contraction of posterior distributions based on Gaussian process priors},
\newblock  Ann. Statist., \textbf{36} (2008), 1435-1463.

\bibitem{vara}
\newblock V. S. Varadarajan,
\newblock ``Measures on topological spaces", 
\newblock  Amer. Math. Soc. Transl., \textbf {2}  (1965), 161--220.

\bibitem {wahba}
\newblock  G. Wahba, 
\newblock \emph{Improper priors, spline smoothing and the problem of guarding against model errors in regression},
\newblock   J. Roy. Statist. Soc. Ser. B, \textbf{40} (1978), 364--372.  



\bibitem{walter}
S. G. Walker, P. Damien, P. W. Laud  and A. F. M. Smith,
\newblock \emph{Bayesian nonparametric inference for random distributions and related functions. With discussion and a reply by the authors},
\newblock  J. R. Stat. Soc. Ser. B Stat. Methodol., \textbf{61} (1999),  485--527.
 
 \bibitem{ambient}
 \newblock R. J. Webster,
 \newblock\emph{Ambient noise statistics},
\newblock IEEE Trans. Signal Proces., \textbf{41} (1993), 2249--2253.


\bibitem{whittle1}  
\newblock P. Whittle, 
\newblock \emph{Curve and periodogram smoothing}, J. \newblock Roy. Statist. Soc. Ser. B, \textbf{19} (1957), 38--47.

\bibitem{whittle2}
\newblock P. Whittle, 
\newblock \emph{On the smoothing of probability density functions},  
\newblock J. Roy. Statist. Soc. Ser. B,  \textbf{20} (1958), 334--343. 
 


\bibitem{wiener2}
\newblock N. Wiener, 
\newblock ``Extrapolation, Interpolation, and Smoothing of Stationary Time Series. With Engineering Applications",
\newblock  Chapman \& Hall, Ltd., London, 1949.
    
\bibitem{wiener} 
\newblock N. Wiener,
\newblock ``Collected works. Vol. I."  Edited by P. Masani.
 MIT Press, Cambridge, Mass.-London, 1976.


\bibitem{willard}
\newblock S.  Willard,
\newblock  ``General topology", 
\newblock  Dover Publications Inc., Mineola NY, 2004. 

     \bibitem{sph}
     \newblock G. Wise and N. Gallagher,
     \newblock\emph  {On spherically invariant random processes}, 
\newblock IEEE Trans. Information theory,
\textbf{24} (1978), 118--120. 



 \bibitem{wol1}
 \newblock R. L. Wolpert and K.  Ickstadt,
\newblock  \emph{Reflecting uncertainty in inverse problems: a Bayesian solution using L\'evy processes},
\newblock  Inverse Problems, \textbf{20} (2004), 1759--1771. 

\bibitem{wol2}
R. L. Wolpert. K. Ickstadt and M. B. Hansen,
\newblock{A nonparametric Bayesian approach to inverse problems}, \newblock in  Bayesian statistics 7, 403--417, Oxford Univ. Press, New York, 2003. 



   \bibitem{xia}
 \newblock D. X. Xia,
 \newblock ``Measure and integration theory on infinite-dimensional spaces",
 \newblock  Academic Press, New York-London, 1972.
      


\bibitem{xi} 
\newblock Y. Xing and B. Ranneby, 
\newblock \emph{Sufficient conditions for Bayesian consistency},
\newblock J. Statist. Plann. Inference, \textbf{139} (2009),  2479--2489.


\bibitem{yamasaki} 
Y. Yamasaki, 
\newblock ``Measures on infinite-dimensional spaces",
 \newblock  World Scientific Publishing Co., Singapore, 1985.



\bibitem{zhao} 
\newblock L. H. Zhao, 
\newblock \emph{Bayesian aspects of some nonparametric problems},
\newblock  Ann. Statist., \textbf{28} (2000), 532--552. 

  \bibitem{zie}
  \newblock W. P.  Ziemer, 
  \newblock ``Weakly differentiable functions. Sobolev spaces and functions of bounded variation",
   \newblock  Springer-Verlag, New York, 1989.


\end{thebibliography}
\end{document}